\documentclass[10pt]{article}

\usepackage{enumerate,xspace}
\usepackage{amsmath,amssymb,wasysym}
\usepackage[all]{xy}
\usepackage{proof}
\usepackage[svgnames]{xcolor}
\usepackage{tikz}

\usepackage{stmaryrd} %need for special interpretation bracket
\usepackage{mathtools}
\usepackage{latexsym}
\usepackage{dsfont}
\usepackage{multicol}
\usepackage{cmll}
\usepackage{multirow}
\usepackage{longtable}
\usepackage{graphicx} % provides "\rotatebox" and "\reflectbox" macros

\usepackage{lscape}
\usepackage{array}

\usepackage{hyperref} %hyperlink package
\hypersetup{
    colorlinks,
    citecolor=red,
    filecolor=red,
    linkcolor=blue,
    urlcolor=red
}

\newtheorem{observation}{Remark}[section]
\newtheorem{lemma}[observation]{Lemma}  %%share counter with remark
\newtheorem{theorem}[observation]{Theorem}
\newtheorem{definition}[observation]{Definition}
\newtheorem{example}[observation]{Example}
\newtheorem{remark}[observation]{Remark}

\newtheorem{proposition}[observation]{Proposition} 
\newtheorem{corollary}[observation]{Corollary}

\makeatletter

\newdimen\w@dth

\def\setw@dth#1#2{\setbox\z@\hbox{\scriptsize $#1$}\w@dth=\wd\z@
\setbox\@ne\hbox{\scriptsize $#2$}\ifnum\w@dth<\wd\@ne \w@dth=\wd\@ne \fi
\advance\w@dth by 1.2em}

\def\t@^#1_#2{\allowbreak\def\n@one{#1}\def\n@two{#2}\mathrel
{\setw@dth{#1}{#2}
\mathop{\hbox to \w@dth{\rightarrowfill}}\limits
\ifx\n@one\empty\else ^{\box\z@}\fi
\ifx\n@two\empty\else _{\box\@ne}\fi}}
\def\t@@^#1{\@ifnextchar_ {\t@^{#1}}{\t@^{#1}_{}}}

\def\t@left^#1_#2{\def\n@one{#1}\def\n@two{#2}\mathrel{\setw@dth{#1}{#2}
\mathop{\hbox to \w@dth{\leftarrowfill}}\limits
\ifx\n@one\empty\else ^{\box\z@}\fi
\ifx\n@two\empty\else _{\box\@ne}\fi}}
\def\t@@left^#1{\@ifnextchar_ {\t@left^{#1}}{\t@left^{#1}_{}}}

\def\two@^#1_#2{\def\n@one{#1}\def\n@two{#2}\mathrel{\setw@dth{#1}{#2}
\mathop{\vcenter{\hbox to \w@dth{\rightarrowfill}\kern-1.7ex
                 \hbox to \w@dth{\rightarrowfill}}%
       }\limits
\ifx\n@one\empty\else ^{\box\z@}\fi
\ifx\n@two\empty\else _{\box\@ne}\fi}}
\def\tw@@^#1{\@ifnextchar_ {\two@^{#1}}{\two@^{#1}_{}}}

\def\tofr@^#1_#2{\def\n@one{#1}\def\n@two{#2}\mathrel{\setw@dth{#1}{#2}
\mathop{\vcenter{\hbox to \w@dth{\rightarrowfill}\kern-1.7ex
                 \hbox to \w@dth{\leftarrowfill}}%
       }\limits
\ifx\n@one\empty\else ^{\box\z@}\fi
\ifx\n@two\empty\else _{\box\@ne}\fi}}
\def\t@fr@^#1{\@ifnextchar_ {\tofr@^{#1}}{\tofr@^{#1}_{}}}

\newdimen\W@dth
\def\setW@dth#1#2{\setbox\z@\hbox{$#1$}\W@dth=\wd\z@
\setbox\@ne\hbox{$#2$}\ifnum\W@dth<\wd\@ne \W@dth=\wd\@ne \fi
\advance\W@dth by 1.2em}

\def\T@^#1_#2{\allowbreak\def\N@one{#1}\def\N@two{#2}\mathrel
{\setW@dth{#1}{#2}
\mathop{\hbox to \W@dth{\rightarrowfill}}\limits
\ifx\N@one\empty\else ^{\box\z@}\fi
\ifx\N@two\empty\else _{\box\@ne}\fi}}
\def\T@@^#1{\@ifnextchar_ {\T@^{#1}}{\T@^{#1}_{}}}

\def\T@left^#1_#2{\def\N@one{#1}\def\N@two{#2}\mathrel{\setW@dth{#1}{#2}
\mathop{\hbox to \W@dth{\leftarrowfill}}\limits
\ifx\N@one\empty\else ^{\box\z@}\fi
\ifx\N@two\empty\else _{\box\@ne}\fi}}
\def\T@@left^#1{\@ifnextchar_ {\T@left^{#1}}{\T@left^{#1}_{}}}

\def\Tofr@^#1_#2{\def\N@one{#1}\def\N@two{#2}\mathrel{\setW@dth{#1}{#2}
\mathop{\vcenter{\hbox to \W@dth{\rightarrowfill}\kern-1.7ex
                 \hbox to \W@dth{\leftarrowfill}}%
       }\limits
\ifx\N@one\empty\else ^{\box\z@}\fi
\ifx\N@two\empty\else _{\box\@ne}\fi}}
\def\T@fr@^#1{\@ifnextchar_ {\Tofr@^{#1}}{\Tofr@^{#1}_{}}}

\def\Two@^#1_#2{\def\N@one{#1}\def\N@two{#2}\mathrel{\setW@dth{#1}{#2}
\mathop{\vcenter{\hbox to \W@dth{\rightarrowfill}\kern-1.7ex
                 \hbox to \W@dth{\rightarrowfill}}%
       }\limits
\ifx\N@one\empty\else ^{\box\z@}\fi
\ifx\N@two\empty\else _{\box\@ne}\fi}}
\def\Tw@@^#1{\@ifnextchar_ {\Two@^{#1}}{\Two@^{#1}_{}}}

\def\to{\@ifnextchar^ {\t@@}{\t@@^{}}}
\def\from{\@ifnextchar^ {\t@@left}{\t@@left^{}}}
\def\tofro{\@ifnextchar^ {\t@fr@}{\t@fr@^{}}}
\def\To{\@ifnextchar^ {\T@@}{\T@@^{}}}
\def\From{\@ifnextchar^ {\T@@left}{\T@@left^{}}}
\def\Two{\@ifnextchar^ {\Tw@@}{\Tw@@^{}}}
\def\Tofro{\@ifnextchar^ {\T@fr@}{\T@fr@^{}}}

\makeatother
\title{Properties and Characterisations of \\ Cofree Cartesian Differential Categories}
\author{Jean-Simon Pacaud Lemay}
\begin{document}
 
\allowdisplaybreaks

\maketitle

\begin{abstract}  Cartesian differential categories come equipped with a differential operator which formalises the total derivative from multivariable calculus. Cofree Cartesian differential categories always exist over a specified base category, where the general construction is based on Faà di Bruno's formula. A natural question to ask is, when given an arbitrary Cartesian differential category, how can one check if it is cofree without knowing the base category? In this paper, we provide characterisations of cofree Cartesian differential categories without specifying a base category. The key to these characterisations is, surprisingly, maps whose derivatives are zero, which we call differential constants. One characterisation is in terms of the homsets being complete ultrametric spaces, where the ultrametric is induced by differential constants, which is similar to the metric for power series. Another characterisation is as algebras of a monad. In either characterisation, the base category is the category of differential constants. We also discuss other basic properties of cofree Cartesian differential categories, such as the linear maps, and explain how many well-known Cartesian differential categories (such as polynomial or smooth functions) are not cofree. 
\end{abstract}

\noindent \small \textbf{Acknowledgements.} The author would first like to thank Robin Cockett, Robert Seely, and Geoff Cruttwell for their support of this research. The author would also like to thank the organizers of the 22nd meeting of the Yorkshire and Midlands Category Theory Seminar (YAMCATS) for inviting me to give a talk on the subject of this research paper, and to the seminar's audience for their questions and comments. In particular, the author would like to thank Steve Vickers for pointing out the metric was actually an \emph{ultra}metric. The author also thanks the anonymous reviewer for useful suggestions on improving the paper's presentation. For this research, the author was financially supported by a JSPS Postdoctoral Fellowship, Award \#: P21746. 
\tableofcontents
\newpage

%%%%%%%%%%%%%%%%%%%%%%%%%%%%%%%%%%%%%%%%%%%%%%%%%%%%%%%%%%%%%%%%%%%%%

\section{Introduction}

Cartesian differential categories, introduced by Blute, Cockett, and Seely in \cite{blute2009cartesian}, are one of the cornerstones in the theory of differential categories, which, as the name suggests, uses category theory to provide and study the foundations and applications of differentiation in a variety of contexts. Briefly, a Cartesian differential category (Definition \ref{cartdiffdef}) is a category with finite products such that each hom-set is a $k$-module, for some fixed commutative unital semiring $k$, which allows for zero maps and sums of maps (Definition \ref{LACdef} \& \ref{CLACdef}), and also comes equipped with a differential combinator $\mathsf{D}$, which for every map ${f: A \to B}$ produces its derivative $\mathsf{D}[f]: A \times A \to B$. The differential combinator satisfies seven axioms, known as \textbf{[CD.1]} to \textbf{[CD.7]}, which formalise the basic identities of the total derivative from multivariable differential calculus such as the chain rule, linearity in vector argument, symmetry of partial derivatives, etc. Cartesian differential categories have been able to formalise various concepts of differential calculus and have also found applications related to computer science, such as for the categorical semantics of the differential $\lambda$-calculus \cite{Cockett-2019}. 
 
There are many interesting examples of Cartesian differential categories, with the two primary examples relating to differentiating polynomials and differentiating smooth functions (Example \ref{ex:CDC}.(\ref{ex:poly}) \& (\ref{ex:smooth})). The central notion of study in this paper are \emph{cofree} Cartesian differential categories (Definition \ref{def:cofree}) which satisfy a couniversal property over a base category amongst Cartesian differential categories. It is worth mentioning that Cartesian differential categories are essentially algebraic, and so it follows from standard results that free Cartesian differential categories exist: they can be given by the term logic for Cartesian differential categories \cite[Section 4]{blute2009cartesian}. On the other hand, the fact that cofree Cartesian differential categories always exist does not follow from any standard theory. This is an interesting similarity that the theory of differential categories shares with the theory of differential algebras, in the sense that cofree differential algebras also exist, which are called Hurwitz series rings \cite{keigher2000hurwitz}. 

The original construction of cofree Cartesian differential categories is called the Faà di Bruno construction (Definition \ref{def:faa}), as first introduced by Cockett and Seely in \cite{cockett2011faa}. The name comes from the fact that composition in a cofree Cartesian differential category is given by a generalisation of Faà di Bruno's formula for higher order derivatives \cite[Lemma 3.4]{garner2020cartesian}. Maps in the Faà di Bruno construction are called Faà di Bruno sequences (Definition \ref{def:faasequence}), which are sequences $(f_0, f_1, f_2, \hdots)$ of maps of the base category where the $n$-th term is interpreted as the $n$-th higher order derivative of the $0$-th term (Definition \ref{def:partial}). The Faà di Bruno construction was then used by Garner and the author in \cite{garner2020cartesian} to prove two fundamental results for Cartesian differential categories. The first was characterising Cartesian differential categories as \emph{skew-}enriched categories \cite[Theorem 6.4]{garner2020cartesian}, while the second was proving that every Cartesian differential category embeds into the coKleisli category of a differential category \cite[Theorem 8.7]{garner2020cartesian}. 

Properties and applications of the Faà di Bruno construction have not yet been explored much. This is most likely due to the fact that composition and differentiation in the Faà di Bruno construction are somewhat complicated, making it somewhat challenging to work directly inside the Faà di Bruno construction. This is unsurprising since, famously, Faà di Bruno's formula is also quite complex and combinatorial. To simplify working with the Faà di Bruno construction, Cockett and Seely used term logic and trees in \cite{cockett2011faa}, while Garner and the author used some pragmatic combinatorial notation \cite{cockett2011faa}. However, in both cases, these techniques are somewhat involved and heavy on notation. Another construction of cofree Cartesian differential categories was then introduced by the author in \cite{lemay2018tangent}, in the hopes that this construction would be easier to work with. Unfortunately, this is still not quite the case. While differentiation is much simpler and composition is somewhat easier to work with, the trade-off is that maps in this alternative construction are much more complicated. Thus, working with either of these constructions of cofree Cartesian differential categories is not as straightforward as one would hope. 

The main objective of this paper is to demonstrate that one does not necessarily need to understand the internal workings of a cofree Cartesian differential category to study it. Indeed, simply from knowing that a Cartesian differential category is cofree over a base category, we may use its couniversal property to derive many interesting properties and structures. Of course, any cofree Cartesian differential category will be isomorphic to the Faà di Bruno construction of its base category (Lemma \ref{lem:cofree-faa2}), so these results may be derived by working with the Faà di Bruno construction directly. However, since working in the Faà di Bruno construction directly is somewhat tricky, the goal of this paper is to do as many proofs using only the couniversal property and avoid working directly with the Faà di Bruno construction (especially its composition and differentiation) as much as possible. In Section \ref{sec:cofreelin}, we show that in a cofree Cartesian differential category, the differential linear maps (Definition \ref{def:dlin}) correspond to the $k$-linear maps of the base category (Proposition \ref{prop:difflincofree}), while the $k$-linear maps instead correspond to pairs of $k$-linear maps of the base category (Proposition \ref{prop:klincofree}). In Section \ref{sec:ultrametric}, we explain how every hom-set of a cofree Cartesian differential category is, in fact, a complete ultrametric space (Proposition \ref{prop:um}), where the metric is induced by the functor to the base category and the higher order derivatives, similar to the metric for power series or Hurwitz series \cite{keigher2000hurwitz}. 

We also discuss how for certain base categories, a more concrete and familiar description of the cofree Cartesian differential category is possible. In particular, we provide a very simple description of the cofree Cartesian differential category over a category with finite biproducts (Proposition \ref{prop:biproductcofree}). Another example we discuss, as first shown in \cite[Proposition 4.9]{garner2020cartesian}, is that for the category of $k$-modules and arbitrary functions between them, the cofree Cartesian differential category over it is given by the coKleisli category of a comonad on the category of $k$-modules (Example \ref{ex:cofreeQ}). In fact, this is an example of a monoidal differential category \cite[Definition 2.4]{blute2006differential}, and in particular, a categorical model of differential linear logic \cite{ehrhard2017introduction} which was studied by Clift and Murfet in \cite[Section 3]{clift2020cofree}. In the case where $k$ is an algebraically closed field of characteristic zero, such as the complex numbers, the category of (finitely generated) cofree cocommutative $k$-coalgebras is a cofree Cartesian $k$-differential category over the category of (finite-dimensional) $k$-vector spaces and arbitrary set functions between them (Example \ref{ex:cofreereal}). 

Another main objective of this paper is to provide a base independent characterisation of cofree Cartesian differential categories. Indeed, a natural question to ask is, when given an arbitrary Cartesian differential category, how can one check if it is cofree without knowing the base category? To do so, we must first understand how to reconstruct the base category from a cofree Cartesian differential category. The answer is given by taking the maps whose derivatives are zero, which we call differential constants (Definition \ref{def:dcon}). This is somewhat surprising since differential constants have, up till now, played no meaningful role in the theory of Cartesian differential categories, while differential linear maps are usually the more important maps. It is somewhat counter-intuitive that a Cartesian differential category is completely generated by maps whose derivatives are zero, yet this is the case for cofree Cartesian differential categories. However, in general, one cannot simply construct a subcategory of differential constants since, crucially, the identity map is not a differential constant. As such, in Section \ref{sec:dcon}, we introduce the notion of a differential constant unit (Definition \ref{def:varsigma}), which behaves like an identity map for differential constants, and if it exists, we can build the category of differential constants as desired. For a cofree Cartesian differential category, its category of differential constants is isomorphic to its base category (Proposition \ref{prop:cofree-diffcon}). As a result, we then show that a Cartesian differential category is cofree if and only if it has a differential constant unit and is cofree over its category of differential constants (Theorem \ref{thm:diffcon1}). From this characterisation, we can easily check when a Cartesian differential category is not cofree by checking if it has a differential constant or not. As such, we can explain why the Cartesian differential categories of polynomials or smooth functions are both not cofree (Example \ref{ex:notcofree}).  

To provide a completely internal and base independent characterisation of cofree Cartesian differential categories, we also need the ability to describe maps as converging infinite sums of their higher order derivative. From the point of view of the Faà di Bruno construction, we wish to able to say that a Faà di Bruno sequence is an infinite converging sum of the form $(f_0, f_1, \hdots) = (f_0, 0, \hdots) + (0, f_1, 0, \hdots) + \hdots$ (Example \ref{ex:cofreeChCD}.(\ref{ex:faaChCD})). In Section \ref{sec:ChDC}, we explain that this can be done in a Cartesian differential category that has a differential constant unit which induces an ultrametric on the homsets, which we call being differential constant complete (Definition \ref{def:dconcom}), and such that the differential constants are well-behaved enough, which we call having convenient differential constants (Definition \ref{def:dconconv}). Therefore a Cartesian differential category is cofree if and only if it has a differential constant unit, is differential constant complete, and has convenient differential constants (Theorem \ref{thm:2}).   

In Section \ref{sec:cofreealg}, we also provide a characterisation of cofree Cartesian differential categories as the algebras of a monad on the category of Cartesian differential categories. This monad arises from the Faà di Bruno adjunction \cite[Corollary 3.13]{garner2020cartesian}, which is the adjunction where the left adjoint is the forgetful functor and the right adjoint is given by the Faà di Bruno construction. The comonad of the Faà di Bruno adjunction was first studied by Cockett and Seely in \cite[Theorem 2.2.2]{cockett2011faa}, where they showed that the coalgebras of this comonad were precisely Cartesian $k$-differential categories \cite[Theorem 3.2.4]{cockett2011faa}. In more categorical terms, this says that the Faà di Bruno adjunction is comonadic. For the induced monad of the Faà di Bruno adjunction, we show that a Cartesian differential category is an algebra of this monad if and only if it is cofree (Theorem \ref{thm:faaalg-cofree}). From this, it follows that the Faà di Bruno adjunction is also monadic (Proposition \ref{prop:monadic}), and that the Faà di Bruno comonad is of effective descent type \cite[Section 2]{mesablishvili2006monads}.

The characterisations introduced in this paper should simplify working with cofree Cartesian differential categories and hopefully lead to interesting new results and applications for cofree Cartesian differential categories.

\section{Cartesian Differential Categories}\label{sec:CDC}

In this background section, we review the basics of Cartesian differential categories. For a more in-depth introduction to Cartesian differential categories, we refer the reader to \cite{blute2009cartesian,garner2020cartesian}. In this paper, we will work with Cartesian differential categories relative to a fixed commutative unital semiring $k$, as was done in \cite{garner2020cartesian}, so in particular, we do not necessarily assume that we have negatives. When $k = \mathbb{N}$, the semiring of natural numbers, we obtain precisely Blute, Cockett, and Seely's original definition and theory from \cite{blute2009cartesian}. 

The underlying structure of a Cartesian differential category is that of a Cartesian left $k$-linear category, which can be described as a category with finite products which is \emph{skew}-enriched over the category of $k$-modules and $k$-linear maps between them \cite{garner2020cartesian}. Essentially, this means that each hom-set is a $k$-module, so in particular, we have zero maps and can take the sum of maps, but also allow for maps which do not preserve zeroes or sums. Maps which do preserve the module structure are called $k$-linear maps. 

\begin{definition} \label{LACdef} \cite[Section 2.1]{garner2020cartesian} A \textbf{left $k$-linear category} is a category $\mathbb{A}$ such that each hom-set $\mathbb{A}(A,B)$ is a $k$-module with scalar multiplication $\cdot : k \times  \mathbb{A}(A,B) \to  \mathbb{A}(A,B)$, addition ${+: \mathbb{A}(A,B) \times \mathbb{A}(A,B) \to \mathbb{A}(A,B)}$, and zero $0 \in \mathbb{A}(A,B)$, and such that for any map $x: A^\prime \to A$, pre-composition ${\_ \circ x: \mathbb{A}(A,B) \to  \mathbb{A}(A^\prime,B)}$ is a $k$-linear morphism, that is, for all $r,s \in k$ and $f,g \in \mathbb{A}(A,B)$ the following equality holds\footnote{In an arbitrary category, we use the classical notation for composition $\circ$ as opposed to diagrammatic order which was used in other papers on Cartesian differential categories, such as in \cite{blute2009cartesian,lemay2018tangent} for example. We denote identity maps as ${1_A: A \to A}$.}:
\begin{align}
(r \cdot f + s \cdot g) \circ x = r \cdot (f \circ x) + s \cdot (g \circ x) 
\end{align}
In a left $k$-linear category $\mathbb{A}$, a map $f: A\to B$ is said to be \textbf{$k$-linear} if post-composition $f \circ \_ : \mathbb{A}(A^\prime,A) \to \mathbb{A}(A^\prime,B)$ is a $k$-linear morphism, that is, for all $r,s \in k$ and $x,y \in \mathbb{A}(A^\prime,A)$ the following equality holds:
\begin{align}
 f \circ (r \cdot x + s \cdot y) =   r \cdot (f \circ x) + s \cdot (f \circ y)  
\end{align}
For a left $k$-linear category $\mathbb{A}$, let $k\text{-}\mathsf{lin}\left[ \mathbb{A} \right]$ be the subcategory of $k$-linear maps of $\mathbb{A}$ and $\mathcal{I}_{\mathbb{A}}: k\text{-}\mathsf{lin}\left[ \mathbb{A} \right] \to \mathbb{A}$ be the inclusion functor. A \textbf{$k$-linear category} is a left $k$-linear category $\mathbb{A}$ such that every map in $\mathbb{A}$ is $k$-linear, so $k\text{-}\mathsf{lin}\left[ \mathbb{A} \right] = \mathbb{A}$. 
\end{definition}

A list of basic properties of $k$-linear maps can be found in \cite[Proposition 1.1.2]{blute2009cartesian}, such as the fact they are closed under composition and $k$-linear structure. Also, note that a $k$-linear category is precisely a category enriched over $k$-modules. It is worth mentioning that a category can be a left $k$-linear category in possibly different ways. So being a left $k$-linear category is a structure rather than a property. 

We now add finite products to the story and require that the projections be $k$-linear. For a category with finite products, we denote the terminal object as $\ast$, the product by $A_0 \times \hdots \times A_n$ with projections $\pi_j: A_0 \times \hdots \times A_n \to A_j$, and denote the tupling operation by $\langle -, -, \hdots, - \rangle$. 

\begin{definition} \label{CLACdef} \cite[Section 2.1]{garner2020cartesian} A \textbf{Cartesian left $k$-linear category} is a left $k$-linear category $\mathbb{A}$ such that $\mathbb{A}$ has finite products and all projections ${\pi_j: A_0 \times \hdots \times A_n \to A_j}$ are $k$-linear maps. 
\end{definition}

As mentioned above, Cartesian left $\mathbb{N}$-linear categories are precisely Cartesian left additive categories \cite[Definition 1.2.1]{blute2009cartesian}, and in this case, the $\mathbb{N}$-linear maps are precisely additive maps \cite[Definition 1.1.1]{blute2009cartesian}. Conversely, every Cartesian left $k$-linear category is a Cartesian left additive category, and every $k$-linear map is an additive map. Properties of Cartesian left $k$-linear categories can be found in \cite[Section 1.2]{blute2009cartesian}, while a list of examples can be found in \cite[Example 2.3]{garner2020cartesian}. Again, also note that while products are unique up to isomorphism, a category can possibly have multiple $k$-linear structures, which makes it into a Cartesian left $k$-linear category. That said, we will only refer to Cartesian left $k$-linear categories by their underlying category since, in this paper, there should be no confusion. 

In a Cartesian left $k$-linear category, since not every map is $k$-linear or additive, the product $\times$ is not a coproduct and, therefore, not a biproduct. That said, for any Cartesian left $k$-linear category $\mathbb{A}$, its subcategory of $k$-linear maps $k\text{-}\mathsf{lin}\left[ \mathbb{A} \right]$ is a $k$-linear category with finite biproducts. Conversely, a $k$-linear category with finite biproducts is precisely a Cartesian left $k$-linear category where every map is $k$-linear \cite[Example 2.3.(ii)]{garner2020cartesian}. 

A key concept for the story of this paper is the notion of functors between Cartesian left $k$-linear categories that preserve both the $k$-linear structure and finite products \emph{strictly}. Of course, one could instead consider functors that preserve the products up to isomorphism. However, since this approach only adds little to the story of characterising cofree Cartesian differential categories, we will do as in \cite{garner2020cartesian,lemay2018tangent} and work with strict structure preserving functors. Furthermore, it follows that these functors also preserve $k$-linear maps. 

\begin{definition} \cite[Section 3]{lemay2018tangent} A \textbf{Cartesian $k$-linear functor} between Cartesian left $k$-linear categories $\mathbb{A}$ and $\mathbb{A}^\prime$ is a functor $\mathcal{F}: \mathbb{A} \to \mathbb{A}^\prime$ which preserves the product structure strictly, that is, $\mathcal{F}(A_0 \times \hdots \times A_n) = \mathcal{F}(A_0) \times \hdots \times \mathcal{F}(A_n)$ and $\mathcal{F}(\pi_j) = \pi_j$, and also preserves the $k$-linear structure strictly, that is, $\mathcal{F}(r \cdot f+ s\cdot g) = r\cdot \mathcal{F}(f) + s\cdot \mathcal{F}(g)$ for all parallel maps $f,g \in \mathbb{A}$ and $r,s \in k$. 
\end{definition}

\begin{lemma}\cite[Lemma 1.3.2]{blute2009cartesian}\label{lem:funcadd} Let $\mathcal{F}: \mathbb{A} \to \mathbb{A}^\prime$ be a Cartesian $k$-linear functor. Then if $f$ is a $k$-linear map in $\mathbb{A}$, $\mathcal{F}(f)$ is a $k$-linear map in $\mathbb{A}^\prime$. Therefore, there is a Cartesian $k$-linear functor $k\text{-}\mathsf{lin}\left[\mathcal{F}\right]: k\text{-}\mathsf{lin}\left[ \mathbb{A} \right] \to k\text{-}\mathsf{lin}\left[ \mathbb{A}^\prime \right]$, defined on objects and maps as $k\text{-}\mathsf{lin}\left[\mathcal{F}\right](-) = \mathcal{F}(-)$, and such that the following diagram commutes: 
 \begin{align*} \xymatrixcolsep{5pc}\xymatrix{ k\text{-}\mathsf{lin}\left[ \mathbb{A} \right] \ar[r]^-{k\text{-}\mathsf{lin}\left[\mathcal{F}\right]} \ar[d]_-{\mathcal{I}_{\mathbb{A}}} & k\text{-}\mathsf{lin}\left[ \mathbb{A}^\prime \right] \ar[d]^-{\mathcal{I}_{\mathbb{A}^\prime}} \\ 
 \mathbb{A} \ar[r]_-{\mathcal{F}} & \mathbb{A}^\prime }
\end{align*}
\end{lemma}

Observe of course that ${\mathcal{I}_{\mathbb{A}}: k\text{-}\mathsf{lin}\left[ \mathbb{A} \right] \to \mathbb{A}}$ is trivially a Cartesian $k$-linear functor. 

Cartesian differential categories are Cartesian left $k$-linear categories that also come equipped with a differential combinator, which is an operator that sends maps to their derivative. The axioms of a differential combinator are analogues of the basic properties of the total derivative from multivariable differential calculus. There are various equivalent ways of expressing the axioms of a Cartesian differential category. Here, we have chosen the one found in \cite[Section 2.2]{garner2020cartesian}. Terminology-wise, in \cite{garner2020cartesian} we used the term Cartesian $k$-linear differential category, but for brevity, here we will use the term Cartesian $k$-differential category. It is also essential to notice that in this paper, unlike in \cite{blute2009cartesian,cockett2011faa} and other early works on Cartesian differential categories, we incorporate the convention used in more recent works where the linear argument of the derivative is its second argument rather than its first argument. 

\begin{definition}\label{cartdiffdef} \cite[Section 2.2]{garner2020cartesian} A \textbf{differential combinator} $\mathsf{D}$ on a Cartesian left $k$-linear category $\mathbb{X}$ is a family of functions between the hom-sets:
\begin{align*} \mathsf{D}: \mathbb{X}(X,Y) \to \mathbb{X}(X \times X,Y) && \frac{f : X \to Y}{\mathsf{D}[f]: X \times X \to Y}
\end{align*}
where $\mathsf{D}[f]$ is called the \textbf{derivative} of $f$, and such that the following seven axioms hold:  
\begin{enumerate}[{\bf [CD.1]}]
\item \label{CDCax1} $\mathsf{D}[r\cdot f + s \cdot g] = r \cdot \mathsf{D}[f] + s \cdot \mathsf{D}[g]$ for all $r,s \in k$ 
\item \label{CDCax2} $\mathsf{D}[f] \circ \langle x, r\cdot y+ s\cdot z \rangle= r\cdot \left( \mathsf{D}[f] \circ \langle x, y \rangle \right) + s \cdot \left( \mathsf{D}[f] \circ \langle x, z \rangle \right)$ for all $r,s \in k$ and suitable maps $x$, $y$, and $z$ 
\item \label{CDCax3} $\mathsf{D}[1_X]=\pi_1$ and $\mathsf{D}[\pi_j] = \pi_j \circ \pi_1 = \pi_{n+j+1}$ 
\item \label{CDCax4} $\mathsf{D}[\left\langle f_0, \hdots, f_n \right \rangle] = \left \langle  \mathsf{D}[f_0], \hdots, \mathsf{D}[f_n] \right \rangle$
\item \label{CDCax5} $\mathsf{D}[g \circ f] = \mathsf{D}[g] \circ \langle f \circ \pi_0, \mathsf{D}[f] \rangle$
\item \label{CDCax6} $\mathsf{D}\left[\mathsf{D}[f] \right] \circ \left \langle x, y, 0, z \right \rangle=  \mathsf{D}[f] \circ \langle x, z \rangle$ for all suitable maps $x$, $y$, and $z$ 
\item \label{CDCax7} $\mathsf{D}\left[\mathsf{D}[f] \right] \circ \left \langle x, y, z, 0 \right \rangle = \mathsf{D}\left[\mathsf{D}[f] \right] \circ \left \langle x, z, y, 0 \right \rangle$ for all suitable maps $x$, $y$, and $z$ 
\end{enumerate}
A \textbf{Cartesian $k$-differential category} is a pair $(\mathbb{X}, \mathsf{D})$ consisting of a Cartesian left $k$-linear category $\mathbb{X}$ and differential combinator $\mathsf{D}$ on $\mathbb{X}$. 
\end{definition}

The convention of denoting Cartesian differential categories as pairs is not necessarily the norm in previous papers on the subject. We adopt this notation for two reasons: (1) to easily distinguish between Cartesian left $k$-linear categories and Cartesian $k$-differential categories, and (2) later we will encounter cases of Cartesian left $k$-linear categories that have more than one possible differential combinator. That said, we will often abuse notation and write $\mathsf{D}$ for all differential combinators. We only use different notations for differential combinators when we wish to emphasize that there is more than one possible differential combinator of interest. 

Briefly, the intuition for axioms of a differential combinator are that: \textbf{[CD.1]} the differential combinator is a $k$-linear morphism, \textbf{[CD.2]} derivatives are $k$-linear in their second argument, \textbf{[CD.3]} what the derivative of identity maps and projections are, \textbf{[CD.4]} the derivative of a tuple is the tuple of the derivatives, \textbf{[CD.5]} the chain rule for the derivative of composition, \textbf{[CD.6]} derivatives are \emph{differential} linear in their second argument, and lastly \textbf{[CD.7]} is the symmetry of the partial derivatives. Again, we mention that when $k = \mathbb{N}$, a Cartesian $\mathbb{N}$-differential category is precisely a Cartesian differential category in the original sense of Blute, Cockett, and Seely \cite[Definition 2.1.1]{blute2009cartesian}, and conversely, every Cartesian $k$-differential category is a Cartesian differential category in the original sense. It is also worth mentioning that there is a sound and complete term logic for Cartesian differential categories \cite[Section 4]{blute2009cartesian}.   

An important class of maps in a Cartesian differential category is the class of \emph{differential} linear maps \cite[Definition 2.2.1]{blute2009cartesian}. In this paper, we borrow the terminology from \cite{garner2020cartesian} and will instead call them $\mathsf{D}$-linear maps, to better distinguish them from $k$-linear maps. 

\begin{definition} \label{def:dlin}\cite[Definition 2.7]{garner2020cartesian} In a Cartesian $k$-differential category $(\mathbb{X}, \mathsf{D})$, a map $f: X \to Y$ is \textbf{differential linear}, or simply \textbf{$\mathsf{D}$-linear} for short, if $\mathsf{D}[f] = f \circ \pi_1$. Let $\mathsf{D}\text{-}\mathsf{lin}\left[ \mathbb{X} \right]$ to be the subcategory of $\mathsf{D}$-linear maps of $\mathbb{X}$ and let ${\mathcal{J}_{(\mathbb{X}, \mathsf{D})}: \mathsf{D}\text{-}\mathsf{lin}\left[ \mathbb{X} \right] \to \mathbb{X}}$ the inclusion functor. 
\end{definition}

Properties of differential linear maps can be found in \cite[Lemma 2.6]{cockett2020linearizing}, such as the fact that they are closed under composition, $k$-linear structure, and product structure. In particular, every differential linear map is also $k$-linear \cite[Lemma 2.6.i]{cockett2020linearizing}, and as result, it follows that for Cartesian $k$-differential category $(\mathbb{X}, \mathsf{D})$, its subcategory of $\mathsf{D}$-linear maps $\mathsf{D}\text{-}\mathsf{lin}\left[ \mathbb{X} \right]$ will be a $k$-linear category with finite biproducts \cite[Corollary 2.2.3]{blute2009cartesian}. It is important to note that although $k$-linear maps and differential linear maps often coincide in many examples, in an arbitrary Cartesian differential category, not every $k$-linear map is necessarily differential linear. 

\begin{example}\label{ex:CDC} \normalfont Here are some well-known main examples of Cartesian differential categories. See \cite{cockett2020linearizing,garner2020cartesian} for lists of other examples of Cartesian differential categories. 
\begin{enumerate}[{\em (i)}]
\item \label{ex:poly} Let $k\text{-}\mathsf{POLY}$ be the category whose objects are $n \in \mathbb{N}$ and where a map ${P: n \to m}$ is a $m$-tuple of polynomials in $n$ variables, that is, $P = \langle p_1(\vec x), \hdots, p_m(\vec x) \rangle$ with $p_i(\vec x) \in k[x_1, \hdots, x_n]$. Then $k\text{-}\mathsf{POLY}$ is a Cartesian $k$-differential category where on objects the product is given by addition, $n \times m = n +m$, and the differential combinator is given by the standard differentiation of polynomials, that is, for $P = \langle p_1(\vec x), \hdots, p_m(\vec x) \rangle: n \to m$, its derivative $\mathsf{D}[P]: 2n \to m$ (where recall that $n \times n = 2n$) is the tuple of the sum of the partial derivatives of the polynomials $p_i(\vec x)$:
\begin{align*}
\mathsf{D}[P](\vec x, \vec y) := \left( \sum \limits^n_{i=1} \frac{\partial p_1(\vec x)}{\partial x_i} y_i, \hdots, \sum \limits^n_{i=1} \frac{\partial p_n(\vec x)}{\partial x_i} y_i \right) 
\end{align*} 
where $\sum \limits^n_{i=1} \dfrac{\partial p_j (\vec x)}{\partial x_i} y_i \in k[x_1, \hdots, x_n, y_1, \hdots, y_n]$. A map $P: n \to m$ is $\mathsf{D}$-linear if and only if it is a tuple of homogenous polynomials of degree 1, so of the form $p_i(\vec x) = \sum\limits^n_{i=1} a_i x_i$. As such if $P: n \to m$ is $\mathsf{D}$-linear, then it corresponds to a $k$-linear morphism $k^n \to k^m$, and so $\mathsf{D}\text{-}\mathsf{lin}[k\text{-}\mathsf{POLY}]$ is equivalent to the category of finite-dimensional free $k$-modules and $k$-linear morphisms between them. However, there could be polynomials that are $k$-linear but not $\mathsf{D}$-linear. For example, if $k=\mathbb{Z}_2$, then $x^2$ is $\mathbb{Z}_2$-linear but not $\mathsf{D}$-linear. 
\item \label{ex:smooth} Let $\mathbb{R}$ be the set of real numbers. Define $\mathsf{SMOOTH}$ as the category whose objects are the Euclidean spaces $\mathbb{R}^n$ and whose maps are smooth functions between them. $\mathsf{SMOOTH}$ is a Cartesian $\mathbb{R}$-differential category where the differential combinator is defined as the total derivative of a smooth function. For a smooth function ${F: \mathbb{R}^n \to \mathbb{R}^m}$, which is in fact an $m$-tuple $F = \langle f_1, \hdots, f_m \rangle$ of smooth functions $f_i: \mathbb{R}^n \to \mathbb{R}$, the derivative ${\mathsf{D}[F]: \mathbb{R}^n \times \mathbb{R}^n \to \mathbb{R}^m}$ is defined as:
\[\mathsf{D}[F](\vec x, \vec y) := \left \langle \sum \limits^n_{i=1} \frac{\partial f_1}{\partial x_i}(\vec x) y_i, \hdots, \sum \limits^n_{i=1} \frac{\partial f_n}{\partial x_i}(\vec x) y_i \right \rangle\]
Note that $\mathbb{R}\text{-}\mathsf{POLY}$ is a sub-Cartesian differential category of $\mathsf{SMOOTH}$. A smooth function $F: \mathbb{R}^n \to \mathbb{R}^m$ is $\mathsf{D}$-linear if and only if it is $\mathbb{R}$-linear in the classical sense. Therefore, $\mathsf{D}\text{-}\mathsf{lin}[\mathsf{SMOOTH}]$ is equivalent to the category of finite-dimensional real vector spaces and $\mathbb{R}$-linear morphisms between them.
\item \label{ex:biproductdiff} Any $k$-linear category $\mathbb{B}$ with finite biproduct is a Cartesian $k$-differential category where the differential combinator $\mathsf{D}^{\mathsf{lin}}$ is defined by precomposing with the second projection map: $ \mathsf{D}^{\mathsf{lin}}[f] = f \circ \pi_1$. Therefore, every map is $\mathsf{D}^{\mathsf{lin}}$-linear by definition and so $\mathsf{D}^{\mathsf{lin}}\text{-}\mathsf{lin}[\mathbb{B}] = \mathbb{B}$. As an explicit example, let $k\text{-}\mathsf{MOD}$ be the category of $k$-modules and $k$-linear morphisms between them. Then $k\text{-}\mathsf{MOD}$ is a Cartesian $k$-differential category where for a $k$-linear morphism $f: M \to N$, its derivative $\mathsf{D}^{\mathsf{lin}}[f]: M \times M \to N$ is defined as $\mathsf{D}^{\mathsf{lin}}[f](m,n) = f(n)$. In fact, a Cartesian $k$-differential category in which every map is differential linear is a $k$-linear category with finite biproducts. That said, as we will see later, it is possible to equip certain $k$-linear categories with finite biproducts with other differential combinators $\mathsf{D}$ such that not every map is $\mathsf{D}$-linear.  
\item An important source of examples of Cartesian differential categories are those which arise as the coKleisli category of a differential category \cite{blute2006differential}. We will discuss a particular example of this in Example \ref{ex:cofreeQ}. 
\end{enumerate} 
\end{example}

We now turn our attention to functors between Cartesian differential categories, which are Cartesian $k$-linear functors that also commute with the differential combinator. It immediately follows that such functors also preserve differential linear maps. 

\begin{definition} \cite[Section 4]{lemay2018tangent} For Cartesian $k$-differential categories $(\mathbb{X}, \!\mathsf{D})$ and $(\mathbb{X}^\prime\! , \!\mathsf{D})$, a \textbf{Cartesian $k$-differential functor} $\mathcal{F}\!:\! (\mathbb{X}, \mathsf{D}) \to (\mathbb{X}^\prime, \mathsf{D})$ between them is a Cartesian $k$-linear functor ${\mathcal{F} : \mathbb{X} \to \mathbb{X}^\prime}$ such that $\mathcal{F}\left(\mathsf{D}[f] \right) = \mathsf{D}\left[ \mathcal{F}(f) \right]$.  
\end{definition}

\begin{lemma}\label{lem:funclin} Let $\mathcal{F}: (\mathbb{X}, \mathsf{D}) \to (\mathbb{X}^\prime, \mathsf{D})$ be a Cartesian $k$-differential functor. Then if $f$ is a differential linear map in $(\mathbb{X}, \mathsf{D})$, $\mathcal{F}(f)$ is a differential linear map in $(\mathbb{X}^\prime, \mathsf{D})$. Therefore, there is a Cartesian $k$-linear functor $\mathsf{D}\text{-}\mathsf{lin}\left[\mathcal{F}\right]: \mathsf{D}\text{-}\mathsf{lin}\left[ \mathbb{X} \right] \to \mathsf{D}\text{-}\mathsf{lin}\left[ \mathbb{X}^\prime \right]$, defined on objects and maps as $\mathsf{D}\text{-}\mathsf{lin}\left[\mathcal{F}\right](-) = \mathcal{F}(-)$, and such that the following diagram commutes: 
 \begin{align*} \xymatrixcolsep{5pc}\xymatrix{ \mathsf{D}\text{-}\mathsf{lin}\left[ \mathbb{X} \right] \ar[r]^-{\mathsf{D}\text{-}\mathsf{lin}\left[\mathcal{F}\right]} \ar[d]_-{\mathcal{J}_{(\mathbb{X}, \mathsf{D})}} & \mathsf{D}\text{-}\mathsf{lin}\left[ \mathbb{X}^\prime \right] \ar[d]^-{\mathcal{J}_{(\mathbb{X}^\prime, \mathsf{D})}} \\ 
 \mathbb{X} \ar[r]_-{\mathcal{F}} & \mathbb{X}^\prime }
\end{align*}
\end{lemma}

Per Example \ref{ex:CDC}.(\ref{ex:biproductdiff}), for any Cartesian $k$-differential category $(\mathbb{X}, \mathsf{D})$, $(\mathsf{D}\text{-}\mathsf{lin}\left[ \mathbb{X} \right], \mathsf{D}^{\mathsf{lin}})$ is a Cartesian $k$-differential category and, moreover, ${\mathcal{J}_{(\mathbb{X}, \mathsf{D})}: (\mathsf{D}\text{-}\mathsf{lin}\left[ \mathbb{X} \right], \mathsf{D}^{\mathsf{lin}}) \to (\mathbb{X}, \mathsf{D})}$ is a Cartesian $k$-differential functor. 

We conclude this section by discussing higher-order derivatives in a Cartesian differential category. There are two kinds: the total higher derivative and the partial higher derivative. The difference is that the former is simply successive applications of the differential combinator, while the latter is successive partial derivations of the non-linear argument. Indeed, in a Cartesian $k$-differential category $(\mathbb{X}, \mathsf{D})$, we can define the partial derivatives by inputting zeroes in the total derivative. So for a map $f: X_0 \times \hdots \times X_n \to Y$, its \textbf{$j$-th partial derivative} \cite[Definition 2.7]{garner2020cartesian} in the argument $X_j$ is the map $\mathsf{D}_{j}[f]: X_0 \times \hdots \times X_n \times X_j \to Y$ defined as:
\begin{align}
    \mathsf{D}_{j}[f] = \mathsf{D}[f] \circ \langle \pi_0, \pi_1, \hdots, \pi_n, 0, \hdots, 0,  \pi_{n+1}, 0, \hdots, 0 \rangle 
\end{align}
where $\pi_{n+1}$ is in the $n+j$-th spot. We also say that a map $f: X_0 \times \hdots \times X_n \to Y$ is \textbf{$\mathsf{D}$-linear in its $j$-th argument} \cite[Definition 2.6]{garner2020cartesian} if $\mathsf{D}_{j}[f] = f \circ \langle \pi_0, \pi_1, \hdots, \pi_{j-1}, \pi_{n+1}, \pi_{j+1}, \hdots \pi_n \rangle$. As mentioned above, \textbf{[CD.6]} is equivalent to the statement that derivatives are differential linear in their second argument, that is, $\mathsf{D}_1\left[ \mathsf{D}[f] \right] = \mathsf{D}[f] \circ \langle \pi_0, \pi_2 \rangle$. 

Now for $n \in \mathbb{N}$, we denote $X^n$ as a shorthand for the product of $n$-copies of $X$. Then for an arbitrary map $f: X \to Y$, applying the differential combinator $n$-times results in a map of type $\mathsf{D}^n[f]: X^{2^n} \to Y$ called the $n$-th total derivative. However, as explained in \cite[Section 3.1]{garner2020cartesian}, it follows from \textbf{[CD.6]} that there is a lot of redundant information in $\mathsf{D}^n[f]$, since it has many differential linear arguments. On the other hand, we can differentiate $f$ once to get $\mathsf{D}[f]: X \times X \to Y$, and then take the partial derivative in the non-differential linear argument to get a map of type $\mathsf{D}_{0}\left[\mathsf{D}[f] \right]: X \times X \times X \to Y$, and so on.  This gives the higher-order partial derivatives of $f$. 

\begin{definition}\label{def:partial}\cite[Definition 3.1]{garner2020cartesian} In a Cartesian $k$-differential category $(\mathbb{X}, \mathsf{D})$, for a map $f: X \to Y$ and $n \in \mathbb{N}$, the \textbf{$n$-th derivative} of $f$ is the map $\partial^n[f]: X \times X^n \to Y$ defined as $\partial^n[f] = \underbrace{\mathsf{D}_0[\mathsf{D}_0[ \hdots \mathsf{D}_0}_{n \text{ times}}[f ] \hdots ] ]$, where by convention $\partial^0[f]=f$ and $\partial^1[f]=\mathsf{D}[f]$.
\end{definition}

\begin{example} \normalfont To highlight the difference between $\mathsf{D}^n$ and $\partial^n$, let us work out a basic example. Consider the polynomial function $p(x) = x^2$. Then $\mathsf{D}$ will differentiate all variables, while $\partial$ will only differentiate the first variable:  
\begin{align*}
& \mathsf{D}^0[p](x) = x^2 && \partial^0[p](x) = x^2 \\
& \mathsf{D}^1[p](x,y) = 2xy && \partial^1[p](x,y) = 2xy \\
& \mathsf{D}^2[p](x,y,z,w) = 2yz + 2xw && \partial^2[p](x,y,z) = 2yz \\
&  \mathsf{D}^3[p](x,y,z,w, a,b,c,d) = 2zb + 2yc + 2wa + 2xd && \partial^3[p](x,y,z,w) = 0 
\end{align*}
    In general, on polynomials, $\mathsf{D}$ will never ``stop'', while $\partial$ will eventually hit zero. 
\end{example}

The higher order total derivative $\mathsf{D}^n$ and the higher order partial derivative $\partial^n$ can actually be defined from one another \cite[Lemma  3.2]{garner2020cartesian}. Indeed, $\partial^n[f]$ can be defined by inputting zeroes in the appropriate arguments of $\mathsf{D}^n[f]$, while $\mathsf{D}^n[f]$ can be expressed as a sum of the $\partial^{k}[f]$, for $1 \leq k \leq n$. While they are ``interchangeable'', we will tell the story of this paper using the higher order partial derivatives $\partial^n$, for reasons we explain below. 

Here are now some useful identities for higher-order partial derivatives. \textbf{[HD.1]} to \textbf{[HD.7]} are the higher-order versions of \textbf{[CD.1]} to \textbf{[CD.7]}, while \textbf{[HD.8]} tells us what the derivative of a higher-order partial derivative is. In particular, \textbf{[HD.5]} is Faà di Bruno's Formula, which expresses the higher-order chain rule. To help write down Faà di Bruno's Formula, let's introduce some notation. For every $n \in \mathbb{N}$, let $[0]=\emptyset$ and let $[n+1] = \lbrace 1< \hdots < n+1 \rbrace$. Now for every subset $I = \lbrace i_1 < \hdots<  i_m \rbrace \subseteq [n+1]$, define $\pi_I: X^n \to X^m$ to be $\pi_I := \langle \pi_{i_1}, \hdots, \pi_{i_m} \rangle$. We denote a \emph{non-empty} partition of $[n+1]$ as $[n+1] = A_1 \vert \hdots \vert A_k$, and let $\vert A_j \vert$ be the cardinality of $A_j$. Then Faà di Bruno's Formula \cite[Lemma 3.14]{garner2020cartesian} for the $n+1$-th derivative is expressed as a sum over the non-empty partitions of $[n+1]$. 

\begin{lemma}\label{lem:partial} \cite[Section 3]{garner2020cartesian} In a Cartesian $k$-differential category $(\mathbb{X}, \mathsf{D})$,  
\begin{enumerate}[{\bf [HD.1]}]
\item  $\partial^n[r\cdot f + s \cdot g] = r \cdot \partial^n[f] + s \cdot \partial^n[g]$ for all $r,s \in k$ 
\item For all $1 \leq j \leq n$, $r,s \in k$ and suitable maps $x_0, x_1, \hdots, x_n$ and $x^\prime_j$:
\begin{gather*}
    \partial^n[f] \circ \langle x_0, x_1, \hdots, r \cdot x_j + s \cdot x^\prime_j, \hdots, x_n\rangle = r \cdot \left(\partial^n[f] \circ \langle x_0, x_1, \hdots, x_j , \hdots, x_n\rangle \right) + s \cdot \left(\partial^n[f] \circ \langle x_0, x_1, \hdots, x^\prime_j , \hdots, x_n\rangle \right)
\end{gather*} 
\item If $f$ is $\mathsf{D}$-linear then $\partial^0[f]=f$, $\partial^1[f]= f \circ \pi_1$, and $\partial^{n+2}[f]=0$ 
\item $\partial^n[\left\langle f_0, \hdots, f_n \right \rangle] = \left \langle  \partial^n[f_0], \hdots, \partial^n[f_n] \right \rangle$
\item The following equality holds:
\begin{gather*}
\partial^n[g \circ f] = \sum \limits_{ [n]=A_1 \vert \hdots \vert A_k}  \partial^k[g] \circ \left \langle f \circ \pi_0,  \partial^{\vert A_1 \vert}[f] \circ \langle \pi_0, \pi_{A_1} \rangle, \hdots,  \partial^{\vert A_k \vert}[f] \circ \langle \pi_0, \pi_{A_k} \rangle \right \rangle 
\end{gather*} 
\item $\partial^n[f]: X \times X^n \to Y$ is $\mathsf{D}$-linear in each of its last $n$-arguments.
\item For all permutations $\sigma: \lbrace 1, \hdots, n \rbrace \xrightarrow{\cong} \lbrace 1, \hdots, n \rbrace$, and all suitable maps $x_0$, ..., and $x_n$, the following equality holds:
 \[ \partial^{n}[f] \circ \langle x_0, x_1, \hdots, x_n\rangle = \partial^{n}[f]  \circ \langle x_0, x_{\sigma(1)}, \hdots, x_{\sigma(n)}\rangle   \]
\item The following equalities hold: 
\begin{align*}
    \mathsf{D}[\partial^n[f]] &=~ \partial^n \left[ \mathsf{D}[f] \right] \circ \left \langle \pi_0, \pi_{n+1}, \pi_1, \pi_{n+2}, \hdots, \pi_{n}, \pi_{2n+1} \right \rangle \\
    &=~ \partial^{n+1}[f] \circ \langle \pi_0, \pi_1, \hdots, \pi_n, \pi_{n+1} \rangle + \sum \limits^n_{j=1} \partial^{n}[f] \circ \langle \pi_0, \pi_1, \hdots, \pi_{j-1}, \pi_{n+j+1}, \pi_{j+1}, \hdots, \pi_n \rangle.  
\end{align*}
\end{enumerate}
\end{lemma}

\section{Cofree Cartesian Differential Categories}

A Cartesian $k$-differential category is cofree if it satisfies a couniversal property amongst Cartesian $k$-differential categories over a base Cartesian left $k$-linear category. 

\begin{definition}\label{def:cofree} A \textbf{cofree Cartesian $k$-differential category} over a Cartesian left $k$-linear category $\mathbb{A}$ is a pair $\left((\mathbb{X}, \mathsf{D}), \mathcal{E}\right)$ consisting of a Cartesian $k$-differential category $(\mathbb{X}, \mathsf{D})$ and a Cartesian $k$-linear functor $\mathcal{E}: \mathbb{X} \to \mathbb{A}$ which satisfies the following couniversal property: for any Cartesian $k$-differential category $(\mathbb{Y},\mathsf{D})$ and Cartesian $k$-linear functor ${\mathcal{F}: \mathbb{Y} \to \mathbb{A}}$, there exists a unique Cartesian $k$-differential functor $\mathcal{F}^\flat: \mathbb{Y} \to \mathbb{X}$ such that the following diagram commutes: 
\begin{align*} \xymatrixcolsep{5pc}\xymatrix{\mathbb{Y} \ar[dr]_-{\mathcal{F}}  \ar@{-->}[r]^-{\exists! ~ \mathcal{F}^\flat}  & \mathbb{X} \ar[d]^-{\mathcal{E}} \\
  & \mathbb{A} }
\end{align*}
A Cartesian $k$-differential category $(\mathbb{X}, \mathsf{D})$ is said to be \textbf{cofree} if there exists a Cartesian left $k$-linear category $\mathbb{A}$ and a Cartesian left $k$-linear functor $\mathcal{E}: \mathbb{X} \to \mathbb{A}$ such that $\left((\mathbb{X}, \mathsf{D}), \mathcal{E}\right)$ is a cofree Cartesian $k$-differential category over $\mathbb{A}$. 
\end{definition}

In \cite{cockett2011faa}, Cockett and Seely provided a general construction of a cofree Cartesian $k$-differential category over any Cartesian left $k$-linear category, which they called the \textbf{Faà di Bruno construction}. Therefore, we may state that:  

\begin{proposition} For every Cartesian left $k$-linear category, there exists a cofree Cartesian $k$-differential category over it.
\end{proposition}

As such, the Faà di Bruno construction induces a right adjoint to the forgetful functor between the category of Cartesian $k$-differential categories and the category of Cartesian left $k$-linear categories \cite[Corollary 3.13]{garner2020cartesian}. Before reviewing the Faà di Bruno construction, let us mention some basic yet important properties of cofree Cartesian differential categories. While these are of course immediate consequences of the couniversal property, or equivalently simply standard results about adjunctions, we take the pain of recording them here since we will make use of them frequently in proofs throughout the paper. 

\begin{lemma} \label{lem:cofree-lem0} Let $\left((\mathbb{X}, \mathsf{D}), \mathcal{E}\right)$ be a cofree Cartesian $k$-differential category over a Cartesian left $k$-linear category $\mathbb{A}$.
\begin{enumerate}[{\em (i)}]
\item \label{lem:cofree-lem1} If $(\mathbb{Y}, \mathsf{D})$ is a Cartesian $k$-differential category, then for any Cartesian $k$-differential functors $\mathcal{F}: (\mathbb{Y}, \mathsf{D}) \to (\mathbb{X}, \mathsf{D})$ and $\mathcal{G}: (\mathbb{Y}, \mathsf{D}) \to (\mathbb{X}, \mathsf{D})$, $\mathcal{F} = \mathcal{G}$ if and only if the following diagram commutes: 
\begin{align*} \xymatrixcolsep{5pc}\xymatrix{\mathbb{Y} \ar[d]_-{\mathcal{G}}  \ar[r]^-{\mathcal{F}}  & \mathbb{X} \ar[d]^-{\mathcal{E}} \\
 \mathbb{X} \ar[r]_-{\mathcal{E}} & \mathbb{A} }
\end{align*}
\item \label{lem:cofree-lem2.i} If $\left((\mathbb{X}^\prime, \mathsf{D}), \mathcal{E}^\prime\right)$ is also a cofree Cartesian $k$-differential categories over $\mathbb{A}$, then there exists a unique Cartesian $k$-differential isomorphism $\mathcal{F}: (\mathbb{X},\mathsf{D}) \to(\mathbb{X}^\prime,\mathsf{D})$, so $\mathbb{X} \cong \mathbb{X}^\prime$, such that the following diagram commutes: 
 \begin{align*} \xymatrixcolsep{5pc}\xymatrix{    \mathbb{X}  \ar[dr]_-{\mathcal{E}} \ar@{-->}[rr]^-{\exists! ~\mathcal{F}}_-{\cong} &&  \mathbb{X}^\prime  \ar[dl]^-{\mathcal{E}^\prime}  \\
& \mathbb{A} }
\end{align*}
\item \label{lem:cofree-lem2.ii} If $\left((\mathbb{X}, \mathsf{D}), \mathcal{E}\right)$ is a cofree Cartesian $k$-differential category over $\mathbb{A}$ and there exists a Cartesian $k$-differential isomorphism $\mathcal{F}: (\mathbb{X},\mathsf{D}) \to(\mathbb{X}^\prime,\mathsf{D})$, so $\mathbb{X} \cong \mathbb{X}^\prime$, then $\left((\mathbb{X}^\prime, \mathsf{D}), \mathcal{E} \circ \mathcal{F}^{-1}\right)$ is a cofree Cartesian $k$-differential category over $\mathbb{A}$.
\item \label{lem:cofreebase.ii} If $\left((\mathbb{X}, \mathsf{D}), \mathcal{E}\right)$ is a cofree Cartesian $k$-differential category over $\mathbb{A}$ and there exists a Cartesian $k$-linear isomorphism $\mathcal{A}: \mathbb{A} \to \mathbb{A}^\prime$, so $\mathbb{A} \cong \mathbb{A}^\prime$, then $\left((\mathbb{X}, \mathsf{D}), \mathcal{A} \circ \mathcal{E} \right)$ is a cofree Cartesian $k$-differential category over $\mathbb{A}^\prime$. 
\end{enumerate}
\end{lemma}

Later in Corollary \ref{lem:cofreebase}, we will explain how the base categories are also unique up to isomorphism (which is not necessarily true for arbitrary adjunctions).  

Let us now review Cockett and Seely's Faà di Bruno construction, as first introduced in \cite{cockett2011faa}. For a complete detailed account of the Faà di Bruno construction over Cartesian left $k$-linear categories, we refer the reader to \cite{cockett2011faa, garner2020cartesian}. It is also worth mentioning the slight differences in presentation in those references. In \cite{cockett2011faa}, Cockett and Seely use the term calculus for Cartesian differential categories \cite[Section 4]{blute2009cartesian}, while in \cite{garner2020cartesian}, Garner and the author use basic combinatorial and categorical notation. The objects of the Faà di Bruno construction are the same as the objects of the base category. The maps in the Faà di Bruno construction consist of Faà di Bruno sequences, which are sequences that generalise sequences of the form $(f, \mathsf{D}[f], \partial^2[f], \partial^3[f], \hdots)$ in a Cartesian $k$-differential category. The basic underlying property of a Faà di Bruno sequence is that the maps are multilinear and symmetric. 

\begin{definition}\label{def:faasequence} \cite[Section 3.1]{garner2020cartesian} In a Cartesian left $k$-linear category $\mathbb{A}$, a map of type ${f: A \times A^n \to B}$ is said to be:
\begin{enumerate}[{\em (i)}]
\item \textbf{$k$-multilinear} in its last $n$-arguments if for all $1 \leq j \leq n$, $r,s \in k$, and suitable maps $x_0, x_1 \hdots, x_n$ and $x^\prime_j$, the following equality holds: 
\begin{gather*}
    f \circ \langle x_0, x_1, \hdots, r \cdot x_j + s \cdot x^\prime_j, \hdots, x_n\rangle = r\cdot \left(f \circ \langle x_0, x_1, \hdots, x_j , \hdots, x_n\rangle \right) + s\cdot \left(f \circ \langle x_0, x_1, \hdots, x^\prime_j , \hdots, x_n\rangle \right) 
\end{gather*}  
 \item \textbf{symmetric} in its last $n$-arguments if for all permutations $\sigma \!: \! \lbrace 1, \hdots, \!n \rbrace \! \xrightarrow{\cong} \! \lbrace 1, \hdots, \!n \rbrace$, and all suitable maps $x_0, x_1 \hdots, x_n$, the following equality holds:
 \[ f \circ \langle x_0, x_1, \hdots, x_n\rangle = f \circ \langle x_0, x_{\sigma(1)}, \hdots, x_{\sigma(n)}\rangle   \]
\end{enumerate}
 For objects $A$ and $B$ in $\mathbb{A}$, an \textbf{$\mathbb{A}$-Faà di Bruno sequence} from $A$ to $B$, denoted ${f_{\bullet}: A \to B}$, is a sequence of maps in $\mathbb{A}$, $f_{\bullet} = (f_n)_{n \in \mathbb{N}} = (f_0,f_1, f_2, \hdots)$ where $f_0: A \to B$ is an arbitrary map and ${f_{n+1}: A \times A^n \to B}$ is a map which is $k$-multilinear and symmetric in its last $n$-arguments. For $\mathbb{A}$-Faà di Bruno sequences $f_{\bullet}: A \to B$ and ${g_{\bullet}: A \to B}$, we say that $f_{\bullet} = g_{\bullet}$ if for all $n \in \mathbb{N}$, $f_n = g_n$. 
\end{definition}

The rest of the Cartesian differential structure of the Faà di Bruno construction is given by generalising the properties of the higher-order partial derivatives from Lemma \ref{lem:partial}. In particular, composition is given by Faà di Bruno's formula \textbf{[HD.5]}, while the differential combinator is given by \textbf{[HD.8]}. The functor from the Faà di Bruno construction back down to the base category is defined as the identity on objects, while mapping a Faà di Bruno sequence to its $0$-th term. 

\begin{definition}\label{def:faa}\cite[Definition 3.6]{garner2020cartesian} For a Cartesian left $k$-linear category $\mathbb{A}$, the \textbf{Faà di Bruno construction} over $\mathbb{A}$ is the Cartesian $k$-differential category $(\mathsf{F}[\mathbb{A}], \mathsf{D})$ where:
\begin{enumerate}[{\em (i)}]
\item The objects of $\mathsf{F}[\mathbb{A}]$ are the same as the objects of $\mathbb{A}$;
\item A map from $A$ to $B$ in $\mathsf{F}[\mathbb{A}]$ is an $\mathbb{A}$-Faà di Bruno sequence $f_{\bullet}: A \to B$;
\item Composition of $\mathbb{A}$-Faà di Bruno sequences $f_{\bullet}: A \to B$ and $g_{\bullet}: B \to C$ is the $\mathbb{A}$-Faà di Bruno sequence $f_{\bullet} \circ g_{\bullet}: A \to C$ defined as: 
\begin{gather*}
   f_{\bullet} \circ g_{\bullet} = \left( \sum \limits_{ [n]=A_1 \vert \hdots \vert A_k}  g^k \circ \left \langle f \circ \pi_0,  f_{\vert A_1 \vert} \circ \langle \pi_0, \pi_{A_1} \rangle, \hdots,  f_{\vert A_k \vert} \circ \langle \pi_0, \pi_{A_k} \rangle \right \rangle \right)_{n \in \mathbb{N}}
\end{gather*}
\item The identity maps are the $\mathbb{A}$-Faà di Bruno sequences ${1_A}_\bullet: A \to A$ defined as ${1_A}_\bullet = (1_A, \pi_1, 0, 0, \hdots)$;
\item The terminal object and the product of objects are the same as in $\mathbb{A}$;
\item The projections are the $\mathbb{A}$-Faà di Bruno sequences ${\pi_j}_\bullet: A_0 \times \hdots \times A_n \to A_j$ defined as ${\pi_j}_\bullet = (\pi_j, \pi_{n+j+1}, 0, 0, \hdots)$, while the tuple of $\mathbb{A}$-Faà di Bruno sequences is defined pointwise, $\langle {f_{0}}_\bullet, \hdots, {f_{m}}_\bullet \rangle = \left( \langle {f_0}_n, \hdots, {f_m}_n \rangle \right)_{n \in \mathbb{N}}$;
\item The $k$-module structure of $\mathbb{A}$-Faà di Bruno sequences is defined pointwise, so addition is $f_{\bullet} + g_{\bullet} = \left( f_n + g_n \right)_{n \in \mathbb{N}}$, scalar multiplication is $r \cdot f_{\bullet} = (r \cdot f_n)_{n \in \mathbb{N}}$, and zero is $0_\bullet= (0,0, \hdots)$;
\item The differential combinator $\mathsf{D}$ sends an $\mathbb{A}$-Faà di Bruno sequence $f_{\bullet}: A \to B$ to the $\mathbb{A}$-Faà di Bruno sequence $\mathsf{D}[f_{\bullet}]: A \times A \to B$ defined as follows: 
\begin{gather*}
\mathsf{D}[f_{\bullet}]  =  \left(f_{n+1} \circ \langle \pi_0, \hdots , \pi_{n+1} \rangle  +  \sum \limits^n_{j=1} f_n \circ \langle \pi_0, \hdots, \pi_{j-1}, \pi_{n+j+1}, \pi_{j+1}, \hdots, \pi_n \rangle \right)_{n \in \mathbb{N}}
\end{gather*}
\end{enumerate}
Let $\mathcal{E}_{\mathbb{A}}: \mathsf{F}[\mathbb{A}] \to \mathbb{A}$ be the functor defined on objects as $\mathcal{E}_\mathbb{A}(A) =A$ and on maps as $\mathcal{E}_\mathbb{A}(f_{\bullet}) = f_0$. 
\end{definition}

For the couniversal property of the Faà di Bruno construction \cite[Section 3.4]{garner2020cartesian}, in order to construct the necessary unique Cartesian $k$-differential functors, we require the following two results: (1) that the higher order partial derivatives do indeed form a Faà di Bruno sequence (which is \textbf{[HD.2]} and \textbf{[HD.7]}), and (2) that Cartesian $k$-linear functors preserve Faà di Bruno sequences. 

\begin{lemma}\label{lem:partialbullet}\cite[Lemma 3.2.(i)]{garner2020cartesian} In a Cartesian $k$-differential category $(\mathbb{X}, \mathsf{D})$, for any map $f: X \to Y$, the sequence $\partial^{\bullet}[f]: X \to Y$, defined as $\partial^{\bullet}[f] = (\partial^n[f])_{n \in \mathbb{N}}$, is an $\mathbb{X}$-Faà di Bruno sequence. 
\end{lemma}

\begin{lemma}\cite[Section 4]{garner2020cartesian} Let $\mathcal{F}: \mathbb{A} \to \mathbb{A}^\prime$ be a Cartesian $k$-linear functor. If a map $f: A_0 \times A_1 \times \hdots \times A_n \to B$ in $\mathbb{A}$ is $k$-multilinear and/or symmetric in its last $n$-arguments, then $\mathcal{F}(f): \mathcal{F}(A_0) \times \mathcal{F}(A_1) \times \hdots \times \mathcal{F}(A_n) \to \mathcal{F}(B)$ is $k$-multilinear and/or symmetric in its last $n$-arguments. Therefore, if $f_{\bullet}: A \to B$ is an $\mathbb{A}$-Faà di Bruno sequence, the sequence $\mathcal{F}(f_{\bullet}): \mathcal{F}(A) \to \mathcal{F}(B)$, defined as $\mathcal{F}(f_{\bullet}) = (\mathcal{F}(f_n))_{n \in \mathbb{N}}$, is an $\mathbb{A}^\prime$-Faà di Bruno sequence. 
\end{lemma}

We may now properly state the couniversal property of the Faà di Bruno construction.

\begin{proposition} \cite[Theorem 3.12]{garner2020cartesian} For any Cartesian left $k$-linear category $\mathbb{A}$, the triple $\left((\mathsf{F}[\mathbb{A}], \mathsf{D}), \mathcal{E}_{\mathbb{A}}\right)$ is a cofree Cartesian $k$-differential category over $\mathbb{A}$, where in particular for any Cartesian $k$-differential category $(\mathbb{Y},\mathsf{D})$ and Cartesian $k$-linear functor ${\mathcal{F}: \mathbb{Y} \to \mathbb{A}}$, the unique Cartesian $k$-differential functor $\mathcal{F}^\flat: (\mathbb{Y}, \mathsf{D}) \to (\mathsf{F}[\mathbb{A}], \mathsf{D})$ such that the following diagram commutes: 
\begin{align*} \xymatrixcolsep{5pc}\xymatrix{\mathbb{Y} \ar[dr]_-{\mathcal{F}}  \ar@{-->}[r]^-{\exists! ~ \mathcal{F}^\flat}  & \mathsf{F}[\mathbb{A}] \ar[d]^-{\mathcal{E}_{\mathbb{A}}} \\
  & \mathbb{A} }
\end{align*}
is defined on objects as $\mathcal{F}^\flat(Y) = \mathcal{F}(Y)$ and on maps as $\mathcal{F}^\flat(f) = \mathcal{F}(\partial^{\bullet}[f])$. Furthermore, for every $\mathbb{A}$-Faà di Bruno sequence $f_{\bullet}$, we have that $\mathcal{E}(\partial^n[f_{\bullet}]) = \partial^n[f_{\bullet}]_0 = f_n$. 
\end{proposition}

It is worth mentioning that the functoriality of $\mathcal{F}^\flat: (\mathbb{Y}, \mathsf{D}) \to (\mathsf{F}[\mathbb{A}], \mathsf{D})$ follows from Faà di Bruno's formula. 

Since cofree Cartesian $k$-differential categories are unique up to isomorphism, a Cartesian $k$-differential category is cofree if and only if it is isomorphic to a Faà di Bruno construction. This provides our first (yet obvious) characterisation of cofree Cartesian differential categories. Again, this follows from standard facts about adjunctions, though we record it here for future use. 

\begin{lemma}\label{lem:cofree-faa2} Let $\mathbb{A}$ be a Cartesian left $k$-linear category, $\mathbb{X}$ a Cartesian $k$-differential category, and $\mathcal{E}: \mathbb{X} \to \mathbb{A}$ a Cartesian $k$-linear functor. Then $\left((\mathbb{X}, \mathsf{D}), \mathcal{E}\right)$ a cofree Cartesian $k$-differential category over $\mathbb{A}$ if and only if the unique Cartesian $k$-differential functor $\mathcal{E}^\flat: (\mathbb{X}, \mathsf{D}) \to (\mathsf{F}[\mathbb{A}], \mathsf{D})$ which makes the following diagram commute:
\begin{align*} \xymatrixcolsep{5pc}\xymatrix{\mathbb{X} \ar[dr]_-{\mathcal{E}}  \ar@{-->}[r]^-{\exists! ~ \mathcal{E}^\flat}  & \mathsf{F}[\mathbb{A}] \ar[d]^-{\mathcal{E}_{\mathbb{A}}} \\
  & \mathbb{A} }
\end{align*}
is an isomorphism. Therefore, a Cartesian $k$-differential category $(\mathbb{X}, \mathsf{D})$ is cofree if and only if there exists a Cartesian left $k$-linear category $\mathbb{A}$ for which there exists a Cartesian $k$-differential isomorphism $(\mathsf{F}[\mathbb{A}], \mathsf{D}) \cong (\mathbb{X}, \mathsf{D})$. 
\end{lemma}

An immediate consequence of this fact is that cofree Cartesian differential categories have essentially the same objects as their base category. 

\begin{corollary}\label{cor:Eobj} Let $\left((\mathbb{X}, \mathsf{D}), \mathcal{E}\right)$ be a cofree Cartesian $k$-differential category over a Cartesian left $k$-linear category $\mathbb{A}$. Then $\mathcal{E}$ is bijective on objects. 
\end{corollary}
\begin{proof} By Lemma \ref{lem:cofree-faa2}, $\mathcal{E}^\flat: \mathbb{X} \to \mathsf{F}[\mathbb{A}]$ is an isomorphism, and so bijective on objects. While $\mathcal{E}_{\mathbb{A}}: \mathsf{F}[\mathbb{A}] \to \mathbb{A}$ is clearly bijective on objects as well. Thus since $\mathcal{E} = \mathcal{E}_{\mathbb{A}} \circ \mathcal{E}^\flat$, we get that $\mathcal{E}$ is also bijective on objects. 
\end{proof}

There is another general construction of cofree Cartesian differential categories given in \cite{lemay2018tangent}. In this construction, the maps are instead given by \textbf{$\mathcal{D}$-sequences} \cite[Definition 4.2]{lemay2018tangent}, which instead generalise sequences of the form $(f, \mathsf{D}[f], \mathsf{D}^2[f], \hdots)$. As discussed in \cite[Section 3.4]{garner2020cartesian}, there is a bijective correspondence between $\mathcal{D}$-sequences and Faà di Bruno sequences, given precisely by the same formulas relating $\mathsf{D}^n$ and $\partial^n$. The advantage of Faà di Bruno sequences is that they are simpler since there are no connections between the maps of a Faà di Bruno sequence, and Faà di Bruno sequences often provide a more understandable story. The disadvantage of Faà di Bruno sequences is that they are sometimes difficult to work with since composition and the differential combinator are given by somewhat complex formulas. On the other hand, the advantage of $\mathcal{D}$-sequences is that they are easier to work with since composition is given by a simpler formula and the differential combinator is defined by shifting the sequence to the left. The disadvantage of $\mathcal{D}$-sequences is that the components of a $\mathcal{D}$-sequence are related to one another, and it is more difficult to check if a sequence is a $\mathcal{D}$-sequence. Since in this paper we will not need to work with the composition or differentiation of Faà di Bruno sequences directly, we have elected to work with them instead since they provide a simpler story to tell. That said, we record the fact that cofree Cartesian differential categories can also be characterised as being isomorphic to the $\mathcal{D}$-sequences construction. 

\begin{proposition} A Cartesian $k$-differential category $(\mathbb{X}, \mathsf{D})$ is cofree if and only if there exists a Cartesian left $k$-linear category $\mathbb{A}$ for which there exists a Cartesian $k$-differential isomorphism $(\mathcal{D}[\mathbb{A}], \mathsf{D}) \cong (\mathbb{X}, \mathsf{D})$, where $(\mathcal{D}[\mathbb{A}], \mathsf{D})$ is the Cartesian $k$-differential category of $\mathcal{D}$-sequences of $\mathbb{A}$ as defined in \cite[Definition 4.7]{lemay2018tangent}.
\end{proposition}  

Here are now some examples of cofree Cartesian $k$-differential categories that can be described more concretely. Our first such examples was discussed in \cite[Section 4.3]{garner2020cartesian}, which provides a connection between cofree Cartesian $k$-differential category and differential linear logic \cite{ehrhard2017introduction}. 

\begin{example} \normalfont \label{ex:cofreeQ} Let $k\text{-}\mathsf{MOD}^\omega$ be the category whose objects are $k$-modules and whose maps are arbitrary set functions between them. Then $k\text{-}\mathsf{MOD}^\omega$ is a Cartesian left $k$-linear category \cite[Example 2.3.(i)]{garner2020cartesian} where the finite product structure is given by the standard product of $k$-modules (which is not a biproduct in $k\text{-}\mathsf{MOD}^\omega$), and where the $k$-linear structure is given pointwise.
%that is, $(r \cdot f + s \cdot g)(x) = r \cdot f(x) + s \cdot g(x)$. 
The $k$-linear maps in the Cartesian left $k$-linear category sense are precisely the $k$-linear morphisms between $k$-modules, in other words $k\text{-}\mathsf{lin}(k\text{-}\mathsf{MOD}^\omega) = k\text{-}\mathsf{MOD}$. The cofree Cartesian $k$-differential category over $k\text{-}\mathsf{MOD}^\omega$ can be described as the coKleisli category of a certain comonad $\mathcal{Q}$ on $k\text{-}\mathsf{MOD}$ which is defined in \cite[Definition 4.8]{garner2020cartesian}. For a $k$-module $M$, let $\mathcal{S}(M)$ be the free symmetric $k$-algebra over $M$, $\mathcal{S}(M) = \bigoplus\limits_{n \in \mathbb{N}} \mathcal{S}_n(M)$, where $\mathcal{S}_n(M)$ is the $n$-th symmetric power of $M$. Then define $\mathcal{Q}(M)$ as the coproduct indexed by the elements of $M$ of copies of $\mathcal{S}(M)$:
\[ \mathcal{Q}(M) := \bigoplus \limits_{x \in M} \mathcal{S}(M) \]
Following \cite[Definition 2.14]{clift2020cofree}, for $x_0, x_1, \hdots, x_n \in M$, we denote $\vert x_1, \hdots, x_n \rangle_{x_0}$ to represent the pure symmetrized tensor $x_1 \otimes^s \hdots \otimes^s x_n \in \mathcal{S}_n(M)$ in the $x_0$-component copy of $\mathcal{S}(M)$ in $\mathcal{Q}(M)$, and $\vert 1 \rangle_{x_0}$ for the multiplicative unit of $\mathcal{S}(M)$ in the $x_0$ -component. Furthermore, $\mathcal{Q}$ is in fact a \textbf{differential modality} \cite[Section 4.2]{garner2020cartesian} and therefore comes equipped with a deriving transformation $\mathsf{d}_M: \mathcal{Q}(M) \otimes_k M \to \mathcal{Q}(M)$, whose axioms are analogues of the basic algebraic properties of differentiation. This makes $k\text{-}\mathsf{MOD}$ into a monoidal differential category \cite[Definition 2.4]{blute2006differential}, and in particular a categorical model of differential linear logic \cite[Section 3]{clift2020cofree}. In general, the coKleisli category of a differential modality is a Cartesian $k$-differential category \cite[Proposition 3.2.1]{blute2009cartesian}. Therefore, the coKleisli category $k\text{-}\mathsf{MOD}^\mathcal{Q}$ is indeed a Cartesian $k$-differential category and also a categorical model of the differential $\lambda$-calculus \cite{Cockett-2019}. Recall that a map between $k$-modules $M$ and $M^\prime$ in $k\text{-}\mathsf{MOD}^\mathcal{Q}$ is a $k$-linear morphism $f: \mathcal{Q}(M) \to M^\prime$. Then its derivative is the $k$-linear morphism $\mathsf{D}[f]: \mathcal{Q}(M \times M) \to M^\prime$ defined as in \cite[Definition 5.11]{clift2020cofree}, which is worked out on pure symmetrized tensors to be: 
\begin{align*}
  \mathsf{D}[f]\left( \vert 1 \rangle_{(x_0,y_0)} \right) &=~ f( \vert y_0 \rangle_{x_0}) \\
  \mathsf{D}[f]\left(\vert (x_1, y_1), \hdots, (x_n, y_n) \rangle_{(x_0,y_0)} \right) &=~ f\left( \vert x_1, x_2, \hdots, x_n, y_0 \rangle_{x_0} \right) + \sum \limits^n_{j=1} f \left( \vert x_1, x_2, \hdots, y_j, \hdots, x_n \rangle_{x_0} \right) 
\end{align*}
The functor $\mathcal{E}_\mathcal{Q}: k\text{-}\mathsf{MOD}^\mathcal{Q} \to k\text{-}\mathsf{MOD}^\omega$ is the identity on objects, $\mathcal{E}_\mathcal{Q}(M) = M$, and sends a $k$-linear morphism $f: \mathcal{Q}(M) \to M^\prime$ to the function $\mathcal{E}_\mathcal{Q}(f): M \to M^\prime$ defined as $\mathcal{E}_\mathcal{Q}(f)(x) = f( \vert 1 \rangle_{x})$. As explained in \cite[Proposition 4.9]{garner2020cartesian}, $\left( (k\text{-}\mathsf{MOD}^\mathcal{Q}, \mathsf{D}), \mathcal{E}^\mathcal{Q} \right)$ is a cofree Cartesian $k$-differential category over $k\text{-}\mathsf{MOD}^\omega$. We can restrict this example to the finite-dimensional case and rewrite this example in terms of polynomials since $\mathcal{S}(k) = k[x_1, \hdots, x_n]$. It is also worth mentioning that for any Cartesian left $k$-linear category $\mathbb{A}$, its Faà di Bruno construction $\mathsf{F}[\mathbb{A}]$ can be embedded into the coKleisli category of a generalisation of the comonad $\mathcal{Q}$ on the category of $k\text{-}\mathsf{MOD}$ valued presheaves of $\mathbb{A}$ \cite[Section 4.4]{garner2020cartesian}. This result was then used to show that a Cartesian $k$-differential category embeds into the coKleisli category of the differential modality of a monoidal differential category \cite[Theorem 8.7]{garner2020cartesian}.  
\end{example}

\begin{example} \label{ex:cofreereal}\normalfont When $k$ is an algebraically closed field of characteristic zero, $\mathcal{Q}(V)$ is the cofree cocommutative $k$-coalgebra over a $k$-vector space $V$ \cite[Proposition 2.11]{clift2020cofree}. In particular, this makes $\mathcal{Q}$ a free exponential modality, and the resulting categorical models of differential linear logic and the differential $\lambda$-calculus were studied by Clift and Murfet in \cite[Section 3 \& 5]{clift2020cofree}. In this case, the coKleisli category of $\mathcal{Q}$ is equivalent to the category of cofree cocommutative $k$-coalgebras. Therefore, the category of (finitely generated) cofree cocommutative $k$-coalgebras is a cofree Cartesian $k$-differential category over the category of (finite-dimensional) $k$-vector spaces and arbitrary set functions between them. 
\end{example}

We conclude this section by describing cofree Cartesian $k$-differential categories over $k$-linear categories with finite biproducts, which is a novel result. To do so, we must first define a new construction of a Cartesian $k$-differential category over any Cartesian left $k$-linear category. We will also make use of this construction to help characterise the $k$-linear maps of a cofree Cartesian $k$-differential category and in the proof of Proposition \ref{prop:cofree-diffcon}. It is important to note that the differential combinator in the following definition does indeed copy the second map. 

\begin{lemma} \label{lem:Adelta} Let $\mathbb{A}$ be a Cartesian left $k$-linear category. Then $(\mathbb{A}^\Delta, \mathsf{D}^\Delta)$ is a Cartesian $k$-differential category where: 
\begin{enumerate}[{\em (i)}]
\item The objects of $\mathbb{A}^\Delta$ are the same as $\mathbb{A}$;
\item The maps in $\mathbb{A}^\Delta$ are pairs $(f,g): A \to B$ where $f: A \to B$ is an arbitrary map and $g: A \to B$ is a $k$-linear map in $\mathbb{A}$, so $\mathbb{A}^\Delta(A,B) = \mathbb{A}(A,B) \times k\text{-}\mathsf{lin}\left[ \mathbb{A} \right](A,B)$;
\item Composition is defined point-wise, $(f_1,g_1) \circ (f_2,g_2) = (f_1 \circ f_2, g_1 \circ g_2)$;
\item The identity of $A$ is the pair $(1_A, 1_A)$.
\item The terminal object and the product of objects are the same as in $\mathbb{A}$;
\item The projections are the pairs $(\pi_j, \pi_j): A_0 \times \hdots \times A_n \to A_j$ and where the tupling is defined point-wise, $\left \langle (f_0,g_0), \hdots, (f_n ,g_n) \right \rangle = \left( \langle f_0, \hdots, f_n \rangle,  \langle g_0, \hdots, g_n \rangle \right)$;
\item The $k$-module structure is defined point-wise, so the addition is $(f_1, g_1) + (f_2, g_2) = (f_1 +f_2, g_1 + g_2)$, the scalar multiplication is $r \cdot (f,g) = (r \cdot f, r \cdot g)$, and the zero is the pair $(0,0)$;
\item The differential combinator $\mathsf{D}^\Delta$ is defined as $\mathsf{D}^\Delta[(f,g)] = (g \circ \pi_1, g \circ \pi_1)$. 
\end{enumerate}
Furthermore, the functor $\mathcal{P}_{\mathbb{A}}: \mathbb{A}^\Delta \to \mathbb{A}$ defined on objects as $\mathcal{P}_{\mathbb{A}}(A) = A$ and on maps as $\mathcal{P}_{\mathbb{A}}(f,g) = f$, is a Cartesian $k$-linear functor. 
\end{lemma}    
\begin{proof} It is clear that $\mathbb{A}^\Delta$ is indeed a Cartesian left $k$-linear category, in fact, it is a sub-Cartesian left $k$-linear category of $\mathbb{A} \times k\text{-}\mathsf{lin}\left[ \mathbb{A} \right]$. For the differential combinator, since the second map is $k$-linear and the differential combinator involves the second projection, which gets rid of the first argument terms, it is straightforward to check that $\mathsf{D}^\Delta$ satisfies the seven axioms \textbf{[CD.1]} to \textbf{[CD.7]}. So we leave this an exercise for the reader. Therefore, we conclude that $(\mathbb{A}^\Delta, \mathsf{D}^\Delta)$ is indeed a Cartesian $k$-differential category. 
\end{proof}

If $\mathbb{B}$ is a $k$-linear category with finite biproducts, recall that $k\text{-}\mathsf{lin}\left[ \mathbb{B} \right] = \mathbb{B}$. Therefore, maps in $\mathbb{B}^\Delta$ are simply pairs of parallel maps in $\mathbb{B}$, so $\mathbb{B}^\Delta(A,B) = \mathbb{B}(A,B) \times \mathbb{B}(A,B)$. In particular, this means that $\mathbb{B}^\Delta$ will also be a $k$-linear category with finite biproducts. As such, another way of proving that $(\mathbb{B}^\Delta, \mathsf{D}^\Delta)$ is a Cartesian $k$-differential category is by using \cite[Corollary 2.3.2]{blute2009cartesian}, which states that for a $k$-linear category with finite biproducts, to give a differential combinator is precisely to give Cartesian $k$-linear endofunctor which is the identity on objects and idempotent. For $\mathsf{D}^\Delta$, its associated functor is the one that sends maps $(f,g)$ to $(g,g)$. Furthermore, it is important to note that the differential combinator $\mathsf{D}^\Delta$ is different from the canonical differential combinator for biproducts $\mathsf{D}^{\mathsf{lin}}$ as defined in Example \ref{ex:CDC}.(\ref{ex:biproductdiff}). Indeed, $\mathsf{D}^\Delta[(f,g)] = (g \circ \pi_1, g \circ \pi_1)$, which gets rid of the first map, while $\mathsf{D}^{\mathsf{lin}}[(f,g)] = (f \circ \pi_1, g \circ \pi_1)$, which keeps the first map. 

We will now prove that $(\mathbb{B}^\Delta, \mathsf{D}^\Delta)$ is a cofree Cartesian $k$-differential category over $\mathbb{B}$. This may be somewhat surprising since the maps in the Faà di Bruno construction were infinite sequences, while the maps in $\mathbb{B}^\Delta$ are simply pairs. The reason for this collapse is that in any $k$-linear category with finite biproducts, the only $k$-multilinear maps in their last $n+2$ arguments are zero maps. Therefore, a Faà di Bruno sequence must be of the form $(f_0, f_1, 0, 0, \hdots)$.  Moreover in this case, the couniversal property can be described using the induced \textbf{linearization combinator} $\mathsf{L}$ \cite[Definition 3.1]{cockett2020linearizing}, which is an operator that makes any map differential linear. Indeed, in a Cartesian $k$-differential category $(\mathbb{X}, \mathsf{D})$, for any map $f: X \to Y$, define $\mathsf{L}[f]: A \to A$, called the \textbf{linearization} of $f$, as $\mathsf{L}[f] = \mathsf{D}[f] \circ \langle 0,1_X \rangle$. Then $\mathsf{L}[f]$ is $\mathsf{D}$-linear, and furthermore, $f$ is $\mathsf{D}$-linear if and only if $f = \mathsf{L}[f]$ \cite[Lemma 2.8]{cockett2020linearizing}. 

\begin{proposition}\label{prop:biproductcofree} Let $\mathbb{B}$ be a $k$-linear category with finite biproducts. Then the triple $\left((\mathbb{B}^\Delta, \mathsf{D}^\Delta), \mathcal{P}_{\mathbb{B}} \right)$ is a cofree Cartesian $k$-differential category over $\mathbb{B}$, where in particular for any Cartesian $k$-differential category $(\mathbb{Y},\mathsf{D})$ and Cartesian $k$-linear functor ${\mathcal{F}: \mathbb{Y} \to \mathbb{B}}$, the unique Cartesian $k$-differential functor $\mathcal{F}^\flat: \mathbb{Y} \to \mathbb{B}^\Delta$ such that the following diagram commutes: 
\begin{align*} \xymatrixcolsep{5pc}\xymatrix{\mathbb{Y} \ar[dr]_-{\mathcal{F}}  \ar@{-->}[r]^-{\exists! ~ \mathcal{F}^\flat}  & \mathbb{B}^\Delta \ar[d]^-{\mathcal{P}_{\mathbb{B}}} \\
  & \mathbb{B} }
\end{align*}
is defined on objects as $\mathcal{F}^\flat(Y) = \mathcal{F}(Y)$ and on maps as $\mathcal{F}^\flat(f) = \left(\mathcal{F}(f), \mathcal{F}\left( \mathsf{L}[f]  \right) \right)$. 
\end{proposition}
\begin{proof} We must first prove that $\mathcal{F}^\flat$ is a functor. Since identity maps are always differential linear by \textbf{[CD.3]}, we have that $\mathsf{L}[1] = 1$. Therefore, it follows that $\mathcal{F}^\flat$ preserves identities. To show that $\mathcal{F}^\flat$ also preserves composition, first note that for all maps $f$ of $\mathbb{Y}$, using \textbf{[CD.2]} and that all maps in $\mathbb{B}$ are $k$-linear, we can compute that for any appropriate map $x,y$ in $\mathbb{B}$: 
\begin{align*}
    \mathcal{F}(\mathsf{D}[f]) \circ \langle x,y \rangle &=~  \mathcal{F}(\mathsf{D}[f]) \circ \left(  \langle x, 0 \rangle  +  \langle 0, y \rangle  \right) \\
     &=~ \left( \mathcal{F}(\mathsf{D}[f]) \circ \langle x, 0 \rangle  \right) + \left( \mathcal{F}(\mathsf{D}[f]) \circ \langle 0, y\rangle \right) \\
    &=~ \left( \mathcal{F}(\mathsf{D}[f]) \circ \langle \pi_0, 0 \rangle \circ x \right) + \left( \mathcal{F}(\mathsf{D}[f]) \circ \langle 0, y\rangle \right) \\
    &=~ \left(\mathcal{F}\left( \mathsf{D}[f] \circ \langle \pi_0, 0 \rangle  \right) \circ x \right) + \left( \mathcal{F}(\mathsf{D}[f]) \circ \langle 0, y\rangle \right) \\
    &=~ \left( \mathcal{F}\left(0 \right) \circ x \right) + \left( \mathcal{F}(\mathsf{D}[f]) \circ \langle 0, y \rangle \right) \\
    &=~ (0 \circ x) + \left( \mathcal{F}(\mathsf{D}[f]) \circ \langle 0, y\rangle  \right) \\
    &=~ 0 + \left( \mathcal{F}(\mathsf{D}[f]) \circ \langle 0, y \rangle  \right) \\
    &=~ \mathcal{F}(\mathsf{D}[f]) \circ \langle 0, y \rangle
\end{align*}
So $\mathcal{F}(\mathsf{D}[f])$ is independent of its first variable, so $ \mathcal{F}^\flat(\mathsf{D}[f]) \circ \langle x,y \rangle = \mathcal{F}^\flat(\mathsf{D}[f]) \circ \langle 0, y \rangle$. Now using this and \textbf{[CD.5]}, we can compute that: 
\begin{align*}
 \mathcal{F}^\flat(g \circ f) &=~  \left(\mathcal{F}(g 
 \circ f), \mathcal{F}\left( \mathsf{L}[g \circ f]  \right) \right) \\ 
 &=~  \left(\mathcal{F}(g 
 \circ f), \mathcal{F}\left( \mathsf{D}[g \circ f] \circ \langle 0,1 \rangle  \right) \right) \\
&=~ \left(\mathcal{F}(g \circ f), \mathcal{F}\left( \mathsf{D}[g \circ f] \right)  \circ \langle 0,1 \rangle \right)   \\
 &=~ \left(\mathcal{F}(g) \circ \mathcal{F}(f), \mathcal{F}\left( \mathsf{D}[g] \circ \left\langle f \circ \pi_0, \mathsf{D}[f]  \right\rangle \right) \circ \langle 0,1 \rangle \right) \\
&=~ \left(\mathcal{F}(g) \circ \mathcal{F}(f), \mathcal{F}\left( \mathsf{D}[g] \right) \circ \left\langle \mathcal{F}\left(f \circ \pi_0\right), \mathcal{F}\left(\mathsf{D}[f]\right)  \right\rangle  \circ \langle 0,1 \rangle \right) \\
&=~ \left(\mathcal{F}(g) \circ \mathcal{F}(f), \mathcal{F}\left( \mathsf{D}[g] \right) \circ \left\langle 0, \mathcal{F}\left(\mathsf{D}[f]\right)  \right\rangle  \circ \langle 0,1 \rangle \right) \\
&=~ \left(\mathcal{F}(g) \circ \mathcal{F}(f), \mathcal{F}\left( \mathsf{D}[g] \right) \circ \langle 0, 1  \rangle \circ \mathcal{F}\left(\mathsf{D}[f]\right) \circ \langle 0,1 \rangle \right) \\
&=~ \left(\mathcal{F}(g), \mathcal{F}\left( \mathsf{D}[g]  \right) \circ \langle 0,1 \rangle \right) \circ \left(\mathcal{F}(f), \mathcal{F}\left( \mathsf{D}[f]  \right) \circ \langle 0,1 \rangle \right) \\
&=~ \left(\mathcal{F}(g), \mathcal{F}\left( \mathsf{D}[g] \circ \langle 0,1 \rangle \right)  \right) \circ \left(\mathcal{F}(f), \mathcal{F}\left( \mathsf{D}[f] \circ \langle 0,1   \rangle\right) \right) \\
&=~ \left(\mathcal{F}(g), \mathcal{F}\left( \mathsf{L}[g]  \right) \right) \circ \left(\mathcal{F}(f), \mathcal{F}\left( \mathsf{L}[f]  \right) \right) \\
&=~ \mathcal{F}^\flat(g) \circ \mathcal{F}^\flat(f) 
\end{align*}
So $\mathcal{F}^\flat$ is indeed a functor. The linearizing combinator $\mathsf{L}$ satisfies six axioms \textbf{[L.1]} to \textbf{[L.6]} \cite[Definition 3.1]{cockett2020linearizing} which are analogues of the first six axioms \textbf{[CD.1]} to \textbf{[CD.6]} of the differential combinator. Since $\mathcal{F}$ is a Cartesian $k$-linear functor and the linearizing combinator $\mathsf{L}$ preserves the $k$-linear structure and projections, it follows that $\mathcal{F}^\flat$ is also a Cartesian $k$-linear functor. To show that $\mathcal{F}^\flat$ also preserves the differential combinator, first observe that by \textbf{[CD.6]} and \textbf{[CD.7]} it follows that the linearization of a derivative is $\mathsf{L}[\mathsf{D}[f]] = \mathsf{D}[f] \circ \langle 0, \pi_1 \rangle$. Also, since $\mathcal{F}(\mathsf{D}[f])$ is independent of its first variable, we have that $\mathcal{F}^\flat(\mathsf{D}[f]) = \mathcal{F}^\flat(\mathsf{D}[f]) \circ \langle 0, \pi_1 \rangle$. So we compute: 
\begin{align*}
\mathcal{F}^\flat\left( \mathsf{D}[f]\right) &=~ \left(\mathcal{F}\left( \mathsf{D}[f]  \right), \mathcal{F}\left( \mathsf{L}\left[ \mathsf{D}[f] \right]  \right)  \right) \\
&=~ \left(\mathcal{F}\left( \mathsf{D}[f]  \right)\circ \langle 0, \pi_1 \rangle , \mathcal{F}\left(  \mathsf{D}[f] \circ\langle 0, \pi_1 \rangle  \right)  \right) \\
&=~ \left(\mathcal{F}\left( \mathsf{D}[f]  \right) \circ \langle 0,1 \rangle \circ \pi_1, \mathcal{F}\left( \mathsf{D}[f]  \right) \circ \langle 0,1 \rangle \circ \pi_1\right) \\
&=~ \left(\mathcal{F}\left( \mathsf{D}[f] \circ \langle 0,1 \rangle \right)  \circ \pi_1, \mathcal{F}\left( \mathsf{D}[f] \circ \langle 0,1 \rangle \right)  \circ \pi_1\right) \\
&=~ \left(\mathcal{F}\left( \mathsf{L}[f]  \right)  \circ \pi_1, \mathcal{F}\left( \mathsf{L}[f]  \right)  \circ \pi_1\right) \\
&=~ \mathsf{D}^\Delta\left[\left(\mathcal{F}(f), \mathcal{F}\left( \mathsf{L}[f]  \right)  \right) \right] \\
&=~  \mathsf{D}\left[\mathcal{F}^\flat(f) \right]  
\end{align*}
So $\mathcal{F}^\flat$ is a Cartesian $k$-differential functor. By definition, we also have that $\mathcal{P}_{\mathbb{B}} \circ \mathcal{F}^\flat = \mathcal{F}$, as desired. For uniqueness, suppose there was another Cartesian $k$-differential functor $\mathcal{G}$, such that $\mathcal{P}_{\mathbb{B}} \circ \mathcal{G} = \mathcal{F}$. This implies that on objects we have that $\mathcal{G}(Y) = \mathcal{F}^\flat(Y)$ and on maps we have that $\mathcal{G}(f) = (\mathcal{F}(f), f^\prime)$ for some map $f^\prime$. Note we have that $f^\prime \circ \pi_1 = \mathcal{P}_{\mathbb{B}}( \mathsf{D}^\Delta[\mathcal{G}(f)] )$. Since $\mathcal{G}$ also preserves the differential combinator, we compute: 
\begin{align*}
f^\prime &=~ f^\prime \circ \pi_1 \circ \langle 0,1 \rangle \\
&=~ \mathcal{P}_{\mathbb{B}}( \mathsf{D}^\Delta[\mathcal{G}(f)] ) \circ \langle 0,1 \rangle \\
&=~  \mathcal{P}_{\mathbb{B}}( \mathcal{G}( \mathsf{D}[f]) ) \circ \langle 0,1 \rangle \\
&=~ \mathcal{F}\left( \mathsf{D}[f]  \right) \circ \langle 0,1 \rangle \\
&=~ \mathcal{F}\left( \mathsf{D}[f]  \circ \langle 0,1 \rangle \right) \\
&=~ \mathcal{F}\left( \mathsf{L}[f]  \right) 
\end{align*}
Thus $\mathcal{G}(f) = \left(\mathcal{F}(f), \mathcal{F}\left( \mathsf{L}[f]  \right) \right) = \mathcal{F}^\flat(f)$. So $\mathcal{G} = \mathcal{F}^\flat$ and therefore $\mathcal{F}^\flat$ is unique. So we conclude that $\left( (\mathbb{B}^\Delta, \mathsf{D}^\Delta), \mathcal{P}_{\mathbb{B}}\right)$ is a cofree Cartesian $k$-differential category over $\mathbb{B}$. 
\end{proof}

\section{Linear Maps in Cofree Cartesian Differential Categories}\label{sec:cofreelin}

In this section, we describe the $k$-linear maps and differential linear maps in a cofree Cartesian $k$-differential category. We will show that the differential linear maps correspond precisely to the $k$-linear maps in the base category, while the $k$-linear maps instead correspond to pairs of $k$-linear maps in the base category. More precisely, we will show that the subcategory of differential linear maps is isomorphic to the subcategory of $k$-linear maps of the base category, while the subcategory of $k$-linear maps is isomorphic to the cofree Cartesian $k$-differential category of the subcategory of $k$-linear maps of the base category. 

\begin{proposition}\label{prop:difflincofree} Let $\left((\mathbb{X}, \mathsf{D}), \mathcal{E}\right)$ be a cofree Cartesian $k$-differential category over a Cartesian left $k$-linear category $\mathbb{A}$. Then the functor $\mathcal{E}_{\mathsf{lin}}: \mathsf{D}\text{-}\mathsf{lin}\left[ \mathbb{X} \right] \to k\text{-}\mathsf{lin}\left[ \mathbb{A} \right]$ defined on objects and maps as $\mathcal{E}_{\mathsf{lin}}(-) = \mathcal{E}(-)$, is a Cartesian $k$-linear isomorphism, so $\mathsf{D}\text{-}\mathsf{lin}\left[ \mathbb{X} \right] \cong k\text{-}\mathsf{lin}\left[ \mathbb{A} \right]$, and it is the unique Cartesian $k$-linear isomorphism such that the following diagram commutes: 
 \begin{align*} \xymatrixcolsep{5pc}\xymatrix{ \mathsf{D}\text{-}\mathsf{lin}\left[ \mathbb{X} \right] \ar@{-->}[r]^-{\exists! ~\mathcal{E}_{\mathsf{lin}}}_-{\cong} \ar[d]_-{\mathcal{J}_{(\mathbb{X}, \mathsf{D})}} & k\text{-}\mathsf{lin}\left[ \mathbb{A} \right] \ar[d]^-{\mathcal{I}_{\mathbb{A}}} \\ 
 \mathbb{X} \ar[r]_-{\mathcal{E}} & \mathbb{A} }
\end{align*}
In other words, for every $\mathsf{D}$-linear map $f$ in $\mathbb{X}$, $\mathcal{E}(f)$ is $k$-linear in $\mathbb{A}$, and conversely, for every $k$-linear map $f$ in $\mathbb{A}$, there exists a unique $\mathsf{D}$-linear map $f^\sharp$ in $\mathbb{X}$ such that $\mathcal{E}(f^\sharp) = f$.  
\end{proposition}
\begin{proof} As explained in Section \ref{sec:CDC}, $k\text{-}\mathsf{lin}\left[ \mathbb{A} \right]$ and $\mathsf{D}\text{-}\mathsf{lin}\left[ \mathbb{X} \right]$ are $k$-linear categories with finite biproducts, and thus are also Cartesian left $k$-linear categories. Now let us explain why $\mathcal{E}_{\mathsf{lin}}$ is well-defined. On objects, this is clear. On maps, recall that if $f$ is $\mathsf{D}$-linear, then $f$ is also $k$-linear, and therefore by Lemma \ref{lem:funcadd}, $\mathcal{E}(f)$ is $k$-linear in $\mathbb{A}$. Therefore, $\mathcal{E}_{\mathsf{lin}}$ is well-defined, and since $\mathcal{E}$ is a Cartesian $k$-linear functor, $\mathcal{E}_{\mathsf{lin}}$ is also. Furthermore, by definition we also have that $\mathcal{I}_{\mathbb{A}} \circ \mathcal{E}_{\mathsf{lin}} = \mathcal{E} \circ \mathcal{J}_{(\mathbb{X}, \mathsf{D})}$, and since $\mathcal{I}_{\mathbb{A}}$ is monic, $\mathcal{E}_{\mathsf{lin}}$ is the unique such functor. To define the inverse of $\mathcal{E}_{\mathsf{lin}}$, we use the fact that $(k\text{-}\mathsf{lin}\left[ \mathbb{A} \right], \mathsf{D}^{\mathsf{lin}})$ is a Cartesian $k$-differential category, as defined in Example \ref{ex:CDC}.(\ref{ex:biproductdiff}). Therefore, by the couniversal property of $\left((\mathbb{X}, \mathsf{D}), \mathcal{E}\right)$, there exists a unique Cartesian $k$-differential functor $\mathcal{I}_{\mathbb{A}}^\flat: (k\text{-}\mathsf{lin}\left[ \mathbb{A} \right], \mathsf{D}^{\mathsf{lin}})  \to (\mathbb{X}, \mathsf{D})$ which makes the following diagram commute: 
\begin{align*} \xymatrixcolsep{5pc}\xymatrix{k\text{-}\mathsf{lin}\left[ \mathbb{A} \right] \ar[dr]_-{\mathcal{I}_{\mathbb{A}}}  \ar@{-->}[r]^-{\exists! ~ \mathcal{I}_{\mathbb{A}}^\flat}  & \mathbb{X} \ar[d]^-{\mathcal{E}} \\
  & \mathbb{A} }
\end{align*}
Then define $\mathcal{E}^{-1}_{\mathsf{lin}}: k\text{-}\mathsf{lin}\left[ \mathbb{A} \right] \to \mathsf{D}\text{-}\mathsf{lin}\left[ \mathbb{X} \right]$ on objects and maps as $\mathcal{E}^{-1}_{\mathsf{lin}}(-) = \mathcal{I}_{\mathbb{A}}^\flat(-)$. On objects $\mathcal{E}^{-1}_{\mathsf{lin}}$ is well-defined, while on maps, since every map $f$ in $k\text{-}\mathsf{lin}\left[ \mathbb{A} \right]$ is $\mathsf{D}^{\mathsf{lin}}$-linear by definition, by Lemma \ref{lem:funclin}, $\mathcal{I}_{\mathbb{A}}^\flat(f)$ will also be $\mathsf{D}$-linear in $\mathbb{X}$. Therefore, $\mathcal{E}^{-1}_{\mathsf{lin}}$ is well-defined, and since $\mathcal{I}_{\mathbb{A}}^\flat$ is a Cartesian $k$-linear functor, $\mathcal{E}^{-1}_{\mathsf{lin}}$ is as well. Since $\mathcal{E} \circ \mathcal{I}_{\mathbb{A}}^\flat = \mathcal{I}_{\mathbb{A}}$, it follows that $\mathcal{E}_{\mathsf{lin}} \circ \mathcal{E}^{-1}_{\mathsf{lin}} = 1_{\mathsf{D}\text{-}\mathsf{lin}\left[ \mathbb{X} \right]}$. For the other direction, recall that $(\mathsf{D}\text{-}\mathsf{lin}\left[ \mathbb{X} \right], \mathsf{D}^{\mathsf{lin}})$ is also a Cartesian $k$-differential category, and note that both $\mathcal{I}_{\mathbb{A}}^\flat \circ \mathcal{E}_{\mathsf{lin}}$ and $\mathcal{J}_{(\mathbb{X}, \mathsf{D})}$ are Cartesian $k$-differential functors of type $(\mathsf{D}\text{-}\mathsf{lin}\left[ \mathbb{X} \right], \mathsf{D}^{\mathsf{lin}}) \to (\mathbb{X}, \mathsf{D})$. Then, by Lemma \ref{lem:cofree-lem0}.(\ref{lem:cofree-lem1}), since the diagram on the left commutes, the diagram on the right commutes: 
\begin{align*}   
   \begin{array}[c]{c}\xymatrixcolsep{5pc}\xymatrix{\mathsf{D}\text{-}\mathsf{lin}\left[ \mathbb{X} \right] \ar[d]_-{\mathcal{J}_{(\mathbb{X}, \mathsf{D})}} \ar[r]^-{\mathcal{E}_{\mathsf{lin}}} & k\text{-}\mathsf{lin}\left[ \mathbb{A} \right] \ar[dr]_-{\mathcal{I}_{\mathbb{A}}}  \ar[r]^-{\mathcal{I}_{\mathbb{A}}^\flat}  & \mathbb{X} \ar[d]^-{\mathcal{E}} \\
 \mathbb{X} \ar[rr]_-{\mathcal{E}} & & \mathbb{A} }  \end{array} \! \stackrel{\text{Lem.\ref{lem:cofree-lem0}.(\ref{lem:cofree-lem1})}}{\Longrightarrow} \! \begin{array}[c]{c}\xymatrixcolsep{5pc}\xymatrix{\mathsf{D}\text{-}\mathsf{lin}\left[ \mathbb{X} \right] \ar[dr]_-{\mathcal{J}_{(\mathbb{X}, \mathsf{D})}} \ar[r]^-{\mathcal{E}_{\mathsf{lin}}} & k\text{-}\mathsf{lin}\left[ \mathbb{A} \right]  \ar[d]^-{\mathcal{I}_{\mathbb{A}}^\flat}  \\
&  \mathbb{X}  } \end{array}
\end{align*}
Then it follows from the diagram on the right that $\mathcal{E}^{-1}_{\mathsf{lin}} \circ \mathcal{E}_{\mathsf{lin}} = 1_{k\text{-}\mathsf{lin}\left[ \mathbb{A} \right]}$. So we conclude that $\mathcal{E}_{\mathsf{lin}}$ is a Cartesian $k$-linear isomorphism, and so $\mathsf{D}\text{-}\mathsf{lin}\left[ \mathbb{X} \right] \cong k\text{-}\mathsf{lin}\left[ \mathbb{A} \right]$. 
\end{proof}

% \begin{corollary} Let $\left((\mathbb{X}, \mathsf{D}), \mathcal{E}\right)$ be a cofree Cartesian $k$-differential category over a Cartesian left $k$-linear category $\mathbb{A}$. Then: 
% \begin{enumerate}[{\em (i)}]
% \item For every $k$-linear map $f$ in $\mathbb{A}$, there exists a unique differential linear map $g$ in $\mathbb{X}$ such that $\mathcal{E}(g) =f$;
% \item For every pair of parallel $k$-linear maps $f_1$ and $f_2$ in $\mathbb{A}$, there exists a unique $k$-linear map $g$ in $\mathbb{X}$ such that $\mathcal{E}(g) =f_1$ and $\mathcal{E}(\mathsf{D}[g]) = f_2 \circ \pi_1$. 
% \end{enumerate}
% \end{corollary}

To prove the desired result for $k$-linear maps, we first need to discuss the derivative of $k$-linear maps. It turns out that the derivative of a $k$-linear map is also a $k$-linear map. Therefore, the subcategory of $k$-linear maps is a Cartesian $k$-differential category.  

\begin{lemma}\label{lem:klindiff} \cite[Proposition 2.3.1]{blute2009cartesian} Let $(\mathbb{X}, \mathsf{D})$ be a Cartesian $k$-differential category. If $f$ is a $k$-linear map in $\mathbb{X}$, then $\mathsf{D}[f]$ is also a $k$-linear map in $\mathbb{X}$. Therefore, $(k\text{-}\mathsf{lin}\left[ \mathbb{X} \right], \mathsf{D})$ is a Cartesian $k$-differential category and the inclusion functor $\mathcal{I}_{\mathbb{X}}: (k\text{-}\mathsf{lin}\left[ \mathbb{X} \right], \mathsf{D}) \to (\mathbb{X}, \mathsf{D})$ is a Cartesian $k$-differential category. 
\end{lemma}

Using Proposition \ref{prop:biproductcofree}, we may show that $k$-linear maps in a cofree Cartesian $k$-differential category correspond to pairs of $k$-linear maps in the base category. 

\begin{proposition}\label{prop:klincofree} Let $\left((\mathbb{X}, \mathsf{D}), \mathcal{E}\right)$ be a cofree Cartesian $k$-differential category over a Cartesian left $k$-linear category $\mathbb{A}$. Since $(k\text{-}\mathsf{lin}\left[ \mathbb{X} \right], \mathsf{D})$ is a Cartesian $k$-differential category as defined in Lemma \ref{lem:klindiff}, consider the induced unique Cartesian $k$-differential functor $\mathcal{E}^\flat_{k\text{-}\mathsf{lin}}: (k\text{-}\mathsf{lin}\left[ \mathbb{X} \right], \mathsf{D}) \to (k\text{-}\mathsf{lin}\left[ \mathbb{A} \right]^\Delta, \mathsf{D}^\Delta)$ that makes the following diagram commute:  
\begin{align*} \xymatrixcolsep{5pc}\xymatrix{k\text{-}\mathsf{lin}\left[ \mathbb{X} \right] \ar[dr]_-{\mathcal{E}_{k\text{-}\mathsf{lin}}}  \ar@{-->}[r]^-{\exists! ~ \mathcal{E}_{k\text{-}\mathsf{lin}}^\flat}  & k\text{-}\mathsf{lin}\left[ \mathbb{A} \right]^\Delta \ar[d]^-{\mathcal{P}_{k\text{-}\mathsf{lin}\left[ \mathbb{A} \right]}} \\
  & k\text{-}\mathsf{lin}\left[ \mathbb{A} \right] }
\end{align*}
as defined in Proposition \ref{prop:biproductcofree}. Explicitly, $\mathcal{E}^\flat_{k\text{-}\mathsf{lin}}$ is defined on objects as $\mathcal{E}^\flat_{k\text{-}\mathsf{lin}}(X) = \mathcal{E}(X)$ and on maps as $\mathcal{E}^\flat_{k\text{-}\mathsf{lin}}(f) = \left(\mathcal{E}(f), \mathcal{E}\left( \mathsf{L}[f]  \right) \right)$. Then $\mathcal{E}_{k\text{-}\mathsf{lin}}$ is a Cartesian $k$-differential isomorphism, so $k\text{-}\mathsf{lin}\left[ \mathbb{X} \right] \cong k\text{-}\mathsf{lin}\left[ \mathbb{A} \right]^\Delta$, and it is the unique Cartesian $k$-differential isomorphism such that the following diagram commutes: 
 \begin{align*} \xymatrixcolsep{5pc}\xymatrix{ \mathsf{k}\text{-}\mathsf{lin}\left[ \mathbb{X} \right] \ar@{-->}[r]^-{\exists! ~\mathcal{E}^\flat_{k\text{-}\mathsf{lin}}}_-{\cong} \ar[dd]_-{\mathcal{I}_{\mathbb{X}}} & k\text{-}\mathsf{lin}\left[ \mathbb{A} \right]^\Delta \ar[d]^-{\mathcal{P}_{k\text{-}\mathsf{lin}\left[ \mathbb{A} \right]}}  \\ 
 & k\text{-}\mathsf{lin}\left[ \mathbb{A} \right] \ar[d]^-{\mathcal{I}_{\mathbb{A}}}\\ 
 \mathbb{X} \ar[r]_-{\mathcal{E}} & \mathbb{A} }
\end{align*}
In other words, for every $k$-linear map $f$ in $\mathbb{X}$, $\mathcal{E}(f)$ and $\mathcal{E}\left( \mathsf{L}[f]  \right)$ are $k$-linear maps in $\mathbb{A}$, and conversely for every pair of parallel $k$-linear maps $f$ and $g$ in $\mathbb{A}$, there exists a unique $k$-linear map $(f,g)^\sharp$ in $\mathbb{X}$ such that $\mathcal{E}((f,g)^\sharp) = f$ and $\mathcal{E}\left( \mathsf{D}[(f,g)^\sharp] \right) = g \circ \pi_1$. 
\end{proposition}
\begin{proof} First observe that by definition and Lemma \ref{lem:funcadd}, it follows that $\mathcal{I}_{\mathbb{A}} \circ \mathcal{P}_{k\text{-}\mathsf{lin}\left[ \mathbb{A} \right]} \circ \mathcal{E}^\flat_{k\text{-}\mathsf{lin}} = \mathcal{E} \circ \mathcal{I}_{\mathbb{X}}$ as desired. Since $\mathcal{I}_{\mathbb{A}}$ is monic and by Lemma \ref{lem:cofree-lem0}.(\ref{lem:cofree-lem1}), we also have that $\mathcal{E}^\flat_{k\text{-}\mathsf{lin}}$ is the unique such Cartesian $k$-differential functor. To construct the inverse of $\mathcal{E}^\flat_{k\text{-}\mathsf{lin}}$, consider the unique Cartesian $k$-differential functor $\mathcal{R}: (k\text{-}\mathsf{lin}\left[ \mathbb{A} \right]^\Delta, \mathsf{D}^\Delta) \to (\mathbb{X}, \mathsf{D})$ such that the following diagram commutes: 
\begin{align*} \xymatrixcolsep{5pc}\xymatrix{k\text{-}\mathsf{lin}\left[ \mathbb{A} \right]^\Delta \ar[d]_-{\mathcal{P}_{k\text{-}\mathsf{lin}\left[ \mathbb{A} \right]}}  \ar@{-->}[r]^-{\exists! ~ \mathcal{R}}  & \mathbb{X} \ar[d]^-{\mathcal{E}} \\
 k\text{-}\mathsf{lin}\left[ \mathbb{A} \right] \ar[r]_-{\mathcal{I}_{\mathcal{A}}}  & \mathbb{A} }
\end{align*}
Now recall that every map in $k\text{-}\mathsf{lin}\left[ \mathbb{A} \right]^\Delta$ is $k$-linear, therefore $k\text{-}\mathsf{lin}\left[k\text{-}\mathsf{lin}\left[ \mathbb{A} \right]^\Delta \right] = k\text{-}\mathsf{lin}\left[ \mathbb{A} \right]^\Delta$. Therefore we have that $\mathcal{R}_{ k\text{-}\mathsf{lin}}: k\text{-}\mathsf{lin}\left[ \mathbb{A} \right]^\Delta \to k\text{-}\mathsf{lin}\left[ \mathbb{X} \right]$. We will now show that $\mathcal{R}_{ k\text{-}\mathsf{lin}}$ and $\mathcal{E}^\flat_{k\text{-}\mathsf{lin}}$ are inverses of each other. First note that by definition of $\mathcal{R}$, we have that on objects $\mathcal{E}(\mathcal{R}(A)) = A$ and on maps $\mathcal{E}(\mathcal{R}(f,g)) = f$. Since $\mathcal{R}$ also commutes with the differential combinator, we compute that: 
\begin{gather*} \mathcal{E}\left( \mathsf{D}\left[ \mathcal{R}(f,g) \right]  \right) = \mathcal{E}\left( \mathcal{R}\left( \mathsf{D}^\Delta\left[ (f,g) \right] \right)  \right) = \mathcal{E}\left( \mathcal{R}\left( g \circ \pi_1, g \circ \pi_1 \right)  \right) = \mathcal{E}(\mathcal{R}(g,g)) \circ  \mathcal{E}(\mathcal{R}(\pi_1, \pi_1)) = g \circ \pi_1 
\end{gather*}
So $\mathcal{E}\left( \mathsf{D}\left[ \mathcal{R}(f,g) \right]  \right) = g \circ \pi_1$ and therefore it follows that $\mathcal{E}\left( \mathsf{L}\left[ \mathcal{R}(f,g) \right]  \right) = g$. Therefore, on object $\mathcal{E}^\flat_{k\text{-}\mathsf{lin}} \left( \mathcal{R}_{ k\text{-}\mathsf{lin}}(A) \right) =A$, while on maps we compute: 
\begin{gather*}
    \mathcal{E}^\flat_{k\text{-}\mathsf{lin}} \left( \mathcal{R}_{ k\text{-}\mathsf{lin}}(f,g) \right) =  \left(\mathcal{E}\left( \mathcal{R}(f,g) \right), \mathcal{E}\left( \mathsf{L}\left[ \mathcal{R}(f,g) \right]  \right) \right)= (f,g)
\end{gather*}
So $\mathcal{E}^\flat_{k\text{-}\mathsf{lin}} \circ \mathcal{R}_{ k\text{-}\mathsf{lin}} = 1_{k\text{-}\mathsf{lin}\left[ \mathbb{A} \right]^\Delta}$. For the other direction, note that $\mathcal{I}_\mathbb{X}$ and $\mathcal{R} \circ \mathcal{E}^\flat_{k\text{-}\mathsf{lin}}$ are Cartesian $k$-differential functors of type $(k\text{-}\mathsf{lin}\left[ \mathbb{X} \right] , \mathsf{D}) \to (\mathbb{X}, \mathsf{D})$. Then, by Lemma \ref{lem:cofree-lem0}.(\ref{lem:cofree-lem1}), since the diagram on the left commutes, the diagram on the right commutes: 
\begin{align*}   
   \begin{array}[c]{c}\xymatrixcolsep{5pc}\xymatrix{k\text{-}\mathsf{lin}\left[ \mathbb{X} \right] \ar[dd]_-{\mathcal{I}_{\mathbb{X}}} \ar[r]^-{\mathcal{E}^\flat_{k\text{-}\mathsf{lin}}} & k\text{-}\mathsf{lin}\left[ \mathbb{A} \right]^\Delta  \ar[d]_-{\mathcal{P}_{k\text{-}\mathsf{lin}\left[ \mathbb{A} \right]}} \ar[r]^-{\mathcal{R}}  & \mathbb{X} \ar[dd]^-{\mathcal{E}} \\
   & k\text{-}\mathsf{lin}\left[ \mathbb{A} \right] \ar[dr]_-{\mathcal{I}_{\mathbb{A}}}  & \\
 \mathbb{X} \ar[rr]_-{\mathcal{E}} & & \mathbb{A} }  \end{array} \!\! \stackrel{\text{Lem.\ref{lem:cofree-lem0}.(\ref{lem:cofree-lem1})}}{\Longrightarrow} \!\! \begin{array}[c]{c}\xymatrixcolsep{5pc}\xymatrix{k\text{-}\mathsf{lin}\left[ \mathbb{X} \right] \ar[dr]_-{\mathcal{I}_{\mathbb{X}}} \ar[r]^-{\mathcal{E}_{\mathsf{lin}}} & k\text{-}\mathsf{lin}\left[ \mathbb{A} \right]^\Delta  \ar[d]^-{\mathcal{R}}  \\
&  \mathbb{X}  } \end{array}
\end{align*}
Then it follows that $\mathcal{R} \circ \mathcal{E}^\flat_{k\text{-}\mathsf{lin}} \!= \! 1_{k\text{-}\mathsf{lin}\left[ \mathbb{X} \right]}$. Thus we conclude that $\mathcal{E}^\flat_{k\text{-}\mathsf{lin}}$ is a Cartesian differential isomorphism and so $k\text{-}\mathsf{lin}\left[ \mathbb{X} \right] \cong k\text{-}\mathsf{lin}\left[ \mathbb{A} \right]^\Delta$. 
\end{proof}

\begin{remark} \normalfont Proposition \ref{prop:klincofree} also implies that the subcategory of $k$-linear maps of a cofree Cartesian $k$-differential category is a cofree Cartesian $k$-differential category over the subcategory of $k$-linear maps of the base category. 
\end{remark}

We have shown that differential linear maps correspond to $k$-linear maps in the base. But since every differential linear map is also a $k$-linear map, it also corresponds to a pair of $k$-linear maps in the base category. The following corollary explains that for differential linear maps, its associated pair is given by two copies of the same $k$-linear map in the base category. 

\begin{corollary} Let $\left((\mathbb{X}, \mathsf{D}), \mathcal{E}\right)$ be a cofree Cartesian $k$-differential category over a Cartesian left $k$-linear category $\mathbb{A}$. Then the following diagram commutes:
\begin{align*} \xymatrixcolsep{5pc}\xymatrix{\mathsf{D}\text{-}\mathsf{lin}\left[ \mathbb{X} \right] \ar[d]_-{{\mathcal{J}_{(\mathbb{X}, \mathsf{D})}}_{\mathsf{lin}}} \ar[r]^-{\mathcal{E}^\flat_{k\text{-}\mathsf{lin}}}_-{\cong}  &  k\text{-}\mathsf{lin}\left[ \mathbb{A} \right] \ar[d]^-{1^\flat_{ k\text{-}\mathsf{lin}\left[ \mathbb{A} \right]}}  \\
 k\text{-}\mathsf{lin}\left[ \mathbb{X} \right]  \ar[r]_-{\mathcal{E}_{\mathsf{lin}}}^-{\cong}  & k\text{-}\mathsf{lin}\left[ \mathbb{A} \right]^\Delta }
\end{align*}
where $1^\flat_{ k\text{-}\mathsf{lin}\left[ \mathbb{A} \right]}: k\text{-}\mathsf{lin}\left[ \mathbb{A} \right] \to k\text{-}\mathsf{lin}\left[ \mathbb{A} \right]^\Delta$ is defined on objects as $1^\flat_{ k\text{-}\mathsf{lin}\left[ \mathbb{A} \right]}(A) =A$ and on maps as $1^\flat_{ k\text{-}\mathsf{lin}\left[ \mathbb{A} \right]}(f) = (f,f)$, and where ${\mathcal{J}_{(\mathbb{X}, \mathsf{D})}}_{\mathsf{lin}}: \mathsf{D}\text{-}\mathsf{lin}\left[ \mathbb{X} \right] \to k\text{-}\mathsf{lin}\left[ \mathbb{X} \right]$ is the inclusion functor. 
\end{corollary}
\begin{proof} Recall that if $f$ is a $\mathsf{D}$-linear map in $\mathbb{X}$ then $\mathsf{L}[f] =f$. Therefore, it follows that the diagram commutes. 
\end{proof}

\begin{example} \normalfont Here are the differential linear maps and $k$-linear maps in our examples of cofree Cartesian $k$-differential categories. 
\begin{enumerate}[{\em (i)}]
\item For a Cartesian left $k$-linear category $\mathbb{A}$, in its Faà di Bruno construction, the $k$-linear maps are the $\mathbb{A}$-Faà di Bruno sequences of the form $(f, g \circ \pi_1, 0, 0, \hdots)$, for any parallel $k$-linear maps $f$ and $g$ in $\mathbb{A}$, while the differential linear maps are the $\mathbb{A}$-Faà di Bruno sequences of the form $(f, f \circ \pi_1, 0, 0, \hdots)$, for any $k$-linear map $f$ in $\mathbb{A}$.
\item For a $k$-linear category $\mathbb{B}$ with finite biproducts, in $\mathbb{B}^\Delta$, all maps $(f,g)$ are $k$-linear maps, while the $\mathsf{D}^\Delta$-linear maps are those of the form $(f,f)$. 
\item In $k\text{-}\mathsf{MOD}^\mathcal{Q}$, the $k$-linear maps correspond to the maps ${f: \mathcal{Q}(M) \to M^\prime}$ in $k\text{-}\mathsf{MOD}$ such that ${\mathcal{E}_\mathcal{Q}(f)\!:\! M \!\to\! M^\prime}$ is a $k$-linear morphism, that is, $f(\vert 1 \rangle_{r \cdot x + s \cdot y}) = r \cdot f(\vert 1 \rangle_x) + s \cdot f(\vert 1 \rangle_y)$, and also for pure symmetrized tensors of degree $n \geq 1$, the following equalities hold (note that in the first case, the index changes to zero on the right): 
\begin{align*}
    f( \vert x_1, \hdots, x_n \rangle_{x_0} ) = \begin{cases}f( \vert x_1 \rangle_{0} ) & \text{ if } n=1 \\
        0 & \text{ if } n > 1
    \end{cases}
\end{align*}
The differential linear maps are the $k$-linear maps such that $f(\vert 1 \rangle_x) = f(\vert x \rangle_0)$. 
\end{enumerate}
\end{example}

\section{Complete Ultrametric}\label{sec:ultrametric}

In this section, we show that each hom-set of a cofree Cartesian differential category is canonically a complete ultrametric space, where the ultrametric is similar to the canonical ultrametric for formal power series or Hurwitz series \cite[Theorem 1.1]{keigher2000hurwitz}. In Section \ref{sec:ChDC} we will use this ultrametric to show that every map in a cofree Cartesian differential category can be decomposed into a converging infinite sum of its higher order derivatives. 

The key to this ultrametric is that maps in a cofree Cartesian differential category are completely determined by the image of their higher-order derivatives in the base category. 

\begin{lemma}\label{lem:E=} Let $\left((\mathbb{X}, \mathsf{D}), \mathcal{E}\right)$ be a cofree Cartesian $k$-differential category over a Cartesian left $k$-linear category $\mathbb{A}$. Then for two parallel maps $f: X \to Y$ and $g: X \to Y$ in $\mathbb{X}$, the following are equivalent: 
\begin{enumerate}[{\em (i)}]
\item\label{Elem1.=} $f=g$; 
\item\label{Elem1.partial} For all $n \in \mathbb{N}$, $\mathcal{E}\left( \partial^n[f] \right) = \mathcal{E}\left( \partial^n[g] \right)$; 
\item\label{Elem1.D} For all $n \in \mathbb{N}$, $\mathcal{E}\left( \mathsf{D}^n[f] \right) = \mathcal{E}\left( \mathsf{D}^n[g] \right)$. 
\end{enumerate}
\end{lemma}
\begin{proof} We will prove that $(\ref{Elem1.=}) \Rightarrow (\ref{Elem1.D}) \Rightarrow (\ref{Elem1.partial}) \Rightarrow (\ref{Elem1.=})$. Obviously, $(\ref{Elem1.=}) \Rightarrow (\ref{Elem1.D})$ is automatic. For $(\ref{Elem1.D}) \Rightarrow (\ref{Elem1.partial})$, recall that, as discussed in Section \ref{sec:CDC}, the $n$-th derivative $\partial^n[f]$ can be obtained by inserting zeroes into the appropriate arguments of the total $n$-th derivative $\mathsf{D}^n[f]$. A bit more explicitly, for a map $f: X \to Y$, $\partial^n[f] = \mathsf{D}^n[f] \circ z_n$, where ${z_n: X \times X^n \to X^{2^n}}$ is the ``injection'' as defined in \cite[Lemma  3.2.(iii)]{garner2020cartesian}, and so is a tuple of the projections $\pi_j$ and zeroes $0$ in the appropriate order. Note that $z_n$ can be defined for any object in any Cartesian left $k$-linear category. As such, since $\mathcal{E}$ is a Cartesian $k$-linear functor, we have that $\mathcal{E}(z_n) = z_n$, and thus $\mathcal{E}\left( \partial^n[f] \right) = \mathcal{E}\left( \mathsf{D}^n[f] \right) \circ z_n$. Therefore, if for all $n \in \mathbb{N}$, $\mathcal{E}\left( \mathsf{D}^n[f] \right) = \mathcal{E}\left( \mathsf{D}^n[g] \right)$, then for all $n \in \mathbb{N}$, $\mathcal{E}\left( \partial^n[f] \right) = \mathcal{E}\left( \partial^n[g] \right)$. 

For $(\ref{Elem1.partial}) \Rightarrow (\ref{Elem1.=})$, we will need to do some more work. For each hom-set $\mathbb{X}(X,Y)$, define the equivalence relation $\thicksim_{\mathcal{E}}$ as follows: for $f, g \in \mathbb{X}(X,Y)$, $f \thicksim_{\mathcal{E}} g$ if and only if for all $n \in \mathbb{N}$, $\mathcal{E}\left( \partial^n[f] \right) = \mathcal{E}\left( \partial^n[g] \right)$. The equivalence relation $\thicksim_{\mathcal{E}}$ is compatible with the Cartesian $k$-differential structure by the properties of the higher-order partial derivative (Lemma \ref{lem:partial}) in the following sense: 
\begin{enumerate}[{\em (i)}]
\item Faà di Bruno's formula \textbf{[HD.5]} tells us that $\partial^n[g \circ f]$ is a sum involving all the lower or equal degree derivatives $\partial^m[g]$ and $\partial^p[f]$, where $0 \leq m, p \leq n$, and some projections. Since $\mathcal{E}$ preserves sums and projections, it follows that if $f_1 \thicksim_{\mathcal{E}} g_1$ and $f_2 \thicksim_{\mathcal{E}} g_2$, then $f_1 \circ f_2 \thicksim_{\mathcal{E}} g_1 \circ g_2$. 
\item By \textbf{[HD.1]} and that $\mathcal{E}$ preserves the $k$-linear structure, it follows that if $f_1 \thicksim_{\mathcal{E}} g_1$ and $f_2 \thicksim_{\mathcal{E}} g_2$, then we have that $r \cdot f_1 + s \cdot f_2 \thicksim_{\mathcal{E}} r \cdot g_1 + s \cdot g_2$. 
\item By \textbf{[HD.4]} and that $\mathcal{E}$ preserves tuplings, it follows that if $f_0 \thicksim_{\mathcal{E}} g_0$, ..., and $f_n \thicksim_{\mathcal{E}} g_n$, then $\langle f_0, \hdots, f_n \rangle \thicksim_{\mathcal{E}} \langle g_0 \hdots, g_n \rangle$. 
\item Note that $\partial^0\left[\mathsf{D}[f] \right] = \partial^1[f]$ and that \textbf{[HD.8]} tells us that $\partial^{n+1}\left[\mathsf{D}[f] \right]$ is the sum of $\partial^{n+2}[f]$ and copies of $\partial^{n+1}[f]$, as well of some appropriate projections. Since $\mathcal{E}$ preserves sums and projections, it follows that if $f \thicksim_{\mathcal{E}} g$, then $\mathsf{D}[f] \thicksim_{\mathcal{E}} \mathsf{D}[g]$. 
\end{enumerate}

Now define the Cartesian $k$-differential category $\left(\mathbb{X}/\!\thicksim_{\mathcal{E}}, \mathsf{D} \right)$ whose objects are the same as $\mathbb{X}$ and where maps are equivalence classes $\llbracket f \rrbracket$ of maps $f$ of $\mathbb{X}$, so the hom-sets are the quotient sets $\mathbb{X}/\!\thicksim_{\mathcal{E}}(X,Y) = \mathbb{X}(X,Y)/\!\thicksim_{\mathcal{E}}$. Composition, identities, the $k$-linear structure, the product structure, and the differential combinator are defined on representatives as in $\mathbb{X}$, which is all well-defined since $\thicksim_{\mathcal{E}}$ is compatible with the Cartesian $k$-differential structure as described above. We have a Cartesian $k$-differential functor $\mathcal{Q}_{\thicksim_{\mathcal{E}}}: \left(\mathbb{X}, \mathsf{D} \right) \to \left(\mathbb{X}/\!\thicksim_{\mathcal{E}}, \mathsf{D} \right)$ which is defined on objects as $\mathcal{Q}_{\thicksim_{\mathcal{E}}}(X) = X$ and on maps as $\mathcal{Q}_{\thicksim_{\mathcal{E}}}(f) = \llbracket f \rrbracket$, and we also have a Cartesian $k$-linear functor ${\mathcal{E}_{\thicksim_{\mathcal{E}}}: \mathbb{X}/\!\thicksim_{\mathcal{E}} \to \mathbb{A}}$ which is defined on objects as $\mathcal{E}_{\thicksim_{\mathcal{E}}}(X) = \mathcal{E}(X)$ and on maps $\mathcal{E}(\llbracket f \rrbracket) = \mathcal{E}(f)$. This functor is well-defined since if $f \thicksim_{\mathcal{E}} g$, the $n=0$ case states that $\mathcal{E}(f) = \mathcal{E}(g)$. We also have that the following diagram commutes: 
\begin{align*} \xymatrixcolsep{5pc}\xymatrix{\mathbb{X} \ar[r]^-{\mathcal{Q}_{\thicksim_{\mathcal{E}}}}  \ar[dr]_-{\mathcal{E}}  & \mathbb{X}/\!\thicksim_{\mathcal{E}} \ar[d]^-{\mathcal{E}_{\thicksim_{\mathcal{E}}}} \\
  & \mathbb{A} }
\end{align*}
Moreover, by the couniversal property of $\left((\mathbb{X}, \mathsf{D}), \mathcal{E}\right)$, there exists a unique Cartesian $k$-differential functor $\mathcal{E}^\flat_{\thicksim_{\mathcal{E}}}: \left(\mathbb{X}/\!\thicksim_{\mathcal{E}}, \mathsf{D} \right) \to \left(\mathbb{X}, \mathsf{D} \right)$ which makes the following diagram commute: 
\begin{align*} \xymatrixcolsep{5pc}\xymatrix{\mathbb{X}/\!\thicksim_{\mathcal{E}} \ar[dr]_-{\mathcal{E}_{\thicksim_{\mathcal{E}}}}  \ar@{-->}[r]^-{\exists! ~ \mathcal{E}^\flat_{\thicksim_{\mathcal{E}}}}  & \mathbb{X} \ar[d]^-{\mathcal{E}} \\
  & \mathbb{A} }
\end{align*}
Then, by Lemma \ref{lem:cofree-lem0}.(\ref{lem:cofree-lem1}), since the diagram on the left commutes, the diagrams on the right commutes: 
\begin{align*}   
   \begin{array}[c]{c}\xymatrixcolsep{5pc}\xymatrix{\mathbb{X} \ar[dr]_-{\mathcal{E}}  \ar[r]^-{ \mathcal{Q}_{\thicksim_{\mathcal{E}}}}  & \mathbb{X}/\!\thicksim_{\mathcal{E}} \ar[r]^-{\mathcal{E}^\flat_{\thicksim_{\mathcal{E}}}} \ar[d]^-{ \mathcal{E}_{\thicksim_{\mathcal{E}}}}  & \mathbb{X} \ar[dl]^-{\mathcal{E}} \\
  & \mathbb{A} }  \end{array}  \stackrel{\text{Lem.\ref{lem:cofree-lem0}.(\ref{lem:cofree-lem1})}}{\Longrightarrow}  \begin{array}[c]{c}\xymatrixcolsep{5pc}\xymatrix{ \mathbb{X} \ar@{=}[dr] \ar[r]^-{\mathcal{Q}_{\thicksim_{\mathcal{E}}}} & \mathbb{X}/\!\thicksim_{\mathcal{E}}  \ar[d]^-{\mathcal{E}^\flat_{\thicksim_{\mathcal{E}}}}  \\
&  \mathbb{X}  } \end{array}
\end{align*}
Now suppose that for all $n \in \mathbb{N}$, $\mathcal{E}\left( \partial^n[f] \right) = \mathcal{E}\left( \partial^n[g] \right)$. In other words, $f \thicksim_{\mathcal{E}} g$, which implies that $\mathcal{Q}_{\thicksim_{\mathcal{E}}}(f) = \mathcal{Q}_{\thicksim_{\mathcal{E}}}(g)$. By the above diagram on the right, $\mathcal{Q}_{\thicksim_{\mathcal{E}}}$ is monic, and therefore, $f=g$ as desired. 
\end{proof}

An ultrametric is a special kind of metric where the triangle inequality is replaced with a stronger inequality. So let $\mathbb{R}_{\geq 0}$ be the set of real non-negative numbers. An ultrametric on a set $M$ is a function $\mathsf{d}: M \times M \to \mathbb{R}_{\geq 0}$ such that for all $x,y,z \in M$ the following axioms hold:
\begin{enumerate}[{\em (i)}]
\item Symmetry: $\mathsf{d}(x,y) = \mathsf{d}(y,x)$;
\item Non-Distinguishable: $\mathsf{d}(x,y) =0$ if and only if $x=y$;
\item Strong Triangle Inequality: $\mathsf{d}(x,z) \leq \mathsf{max}\lbrace \mathsf{d}(x,y), \mathsf{d}(y,z)  \rbrace$. 
\end{enumerate}
An ultrametric space is a pair $(M, \mathsf{d})$ consisting of a set $M$ and an ultrametric $\mathsf{d}$ on $M$. A complete ultrametric space is an ultrametric space where all Cauchy sequences converge. It is also useful to note that in an ultrametric space $(M, \mathsf{d})$, a sequence $(x_n)_{n \in \mathbb{N}}$ is Cauchy if and only if $\lim\limits_{n \to \infty} \mathsf{d}(x_n, x_{n+1}) =0$. In the proof below, we will also need to consider the completion of an ultrametric space. For an ultrametric space $(M, \mathsf{d})$, define $\mathsf{C}[(M, \mathsf{d})]$ to be the set of all Cauchy sequences $(x_n)_{n \in \mathbb{N}}$ in $(M, \mathsf{d})$. Then define an equivalence relation $\thicksim$ on $\mathsf{C}[(M, \mathsf{d})]$ as follows: $(x_n)_{n \in \mathbb{N}} \thicksim (y_n)_{n \in \mathbb{N}}$ if and only if $\lim\limits_{n \to \infty} \mathsf{d}(x_n, y_n) =0$. Then define the completion of $(M, \mathsf{d})$ as the quotient set $\overline{M} := \mathsf{C}[(M, \mathsf{d})]/ \thicksim$, which comes equipped with an ultrametric $\overline{\mathsf{d}}: \overline{M} \times \overline{M} \to \mathbb{R}_{\geq 0}$ defined as $\overline{\mathsf{d}}\left( [(x_n)_{n \in \mathbb{N}}], [(y_n)_{n \in \mathbb{N}}] \right) = \lim\limits_{n \to \infty} \mathsf{d}(x_n, y_n)$. Then $(\overline{M}, \overline{\mathsf{d}})$ is a complete ultrametric space. 

\begin{proposition}\label{prop:um} Let $\left((\mathbb{X}, \mathsf{D}), \mathcal{E}\right)$ be a cofree Cartesian $k$-differential category over a Cartesian left $k$-linear category $\mathbb{A}$. For each hom-set of $\mathbb{X}$, define the function:
\begin{gather*}
    {\mathsf{d}_\mathcal{E}: \mathbb{X}(X,Y) \times \mathbb{X}(X,Y) \to \mathbb{R}_{\geq 0}}\\
    \mathsf{d}_\mathcal{E}(f,g) = \begin{cases} 2^{-n} & \text{ where $n \in \mathbb{N}$ is the smallest natural number} \\
& \text{~such that $\mathcal{E}\left( \partial^n[f] \right) \neq \mathcal{E}\left( \partial^n[g] \right)$} \\ 
& \\ 
0 & \text{ if all $n \in \mathbb{N}$, $\mathcal{E}\left( \partial^n[f] \right) = \mathcal{E}\left( \partial^n[g] \right)$} 
\end{cases}
\end{gather*}
Then $\left( \mathbb{X}(X,Y), \mathsf{d}_\mathcal{E} \right)$ is a complete ultrametric space, and furthermore: 
\begin{enumerate}[{\em (i)}]
\item Composition is non-expansive:
\[\mathsf{d}_\mathcal{E}\left( g_1 \circ f_1, g_2 \circ f_2 \right) \leq \mathsf{max}\lbrace \mathsf{d}_\mathcal{E}(g_1, g_2), \mathsf{d}_\mathcal{E}(f_1, f_2) \rbrace\]
\item Pairing is an isometry:
\[\mathsf{d}_\mathcal{E}\left( \langle f_0, \hdots, f_n \rangle, \langle g_0, \hdots, g_n \rangle \right) = \mathsf{max}\lbrace \mathsf{d}_\mathcal{E}(f_0, g_0), \hdots, \mathsf{d}_\mathcal{E}(f_n, g_n) \rbrace\]
\item Addition and scalar multiplication is non-expansiveL
\[\mathsf{d}_\mathcal{E}(r \cdot f_1 + s \cdot g_1, r\cdot f_2 + s\cdot g_2) \leq \mathsf{max}\lbrace \mathsf{d}_\mathcal{E}(f_1, f_2), \mathsf{d}_\mathcal{E}(g_1, g_2) \rbrace\]
\item For the differential combinator, the following equality holds: 
\begin{align*}
\mathsf{d}_\mathcal{E}\left( \mathsf{D}[f], \mathsf{D}[g] \right) = \begin{cases} 2^{-n} & \text{ where $n \in \mathbb{N}$ is the smallest natural number} \\
& ~  \text{such that $\mathcal{E}\left( \partial^{n+1}[f] \right) \neq \mathcal{E}\left( \partial^{n+1}[g] \right)$}\\ 
& \\ 
0 & \text{ if all $n \in \mathbb{N}$, $\mathcal{E}\left( \partial^n[f] \right) = \mathcal{E}\left( \partial^n[g] \right)$} 
\end{cases} 
\end{align*}
In particular, if $\mathsf{d}_\mathcal{E}(f,g)=0$ then $\mathsf{d}_\mathcal{E}\left( \mathsf{D}[f], \mathsf{D}[g] \right) =0$, and if $\mathsf{d}_\mathcal{E}(f,g)=2^{-(n+1)}$ then $\mathsf{d}_\mathcal{E}\left( \mathsf{D}[f], \mathsf{D}[g] \right) = 2^{-n}$. 
\item A sequence $(f_n)_{n\in \mathbb{N}}$ in $\left( \mathbb{X}(X,Y), \mathsf{d}_\mathcal{E} \right)$ is Cauchy if and only if for all $m \in \mathbb{N}$, there exists a $N_m \in \mathbb{N}$ such that for all $n \geq N_m$, $\mathsf{d}_\mathcal{E}(f_n,f_{n+1}) < 2^{-m}$, or equivalently, $\mathcal{E}\left( \partial^{j}[f_n] \right) = \mathcal{E}\left( \partial^{j}[f_{n+1}] \right)$ for all $0 \leq j \leq m$. 
\end{enumerate}  
\end{proposition}
\begin{proof} We first explain why $\mathsf{d}_\mathcal{E}$ is an ultrametric. Symmetry is automatic. The non-distinguishable axiom is precisely Lemma \ref{lem:E=}. For the strong triangle inequality, consider $\mathsf{d}_\mathcal{E}(f,g)$. If $\mathsf{d}_\mathcal{E}(f,g)=0$ then by definition for any map $h$, $0 \leq \mathsf{max}\lbrace \mathsf{d}_\mathcal{E}(f,h), \mathsf{d}_\mathcal{E}(h,g)  \rbrace$. If $\mathsf{d}_\mathcal{E}(f,g) = 2^{-n}$, consider another map $h$. If $\mathsf{d}_\mathcal{E}(f,h)=0$ or $\mathsf{d}_\mathcal{E}(h,g)=0$, then $f=h$ or $h=g$, so in either case we have $\mathsf{max}\lbrace \mathsf{d}_\mathcal{E}(f,h), \mathsf{d}_\mathcal{E}(h,g)  \rbrace = \mathsf{d}_\mathcal{E}(f,g)$. So suppose that $\mathsf{d}_\mathcal{E}(f,h)=2^{-m}$ and $\mathsf{d}(h,g)=2^{-p}$. If $\mathsf{max}\lbrace 2^{-m},2^{-p} \rbrace < 2^{-n}$, then $n < \mathsf{min}\lbrace m,p \rbrace$ which would imply that $\mathcal{E}\left( \partial^n[f] \right) = \mathcal{E}\left( \partial^n[h] \right)$ and $\mathcal{E}\left( \partial^n[h] \right) = \mathcal{E}\left( \partial^n[g] \right)$. But this is a contradiction since $\mathsf{d}_\mathcal{E}(f,g) = 2^{-n}$ means that $\mathcal{E}\left( \partial^n[f] \right) \neq \mathcal{E}\left( \partial^n[g] \right)$. So we must have that $2^{-n} \leq \mathsf{max}\lbrace 2^{-m},2^{-p} \rbrace$ as desired. So $\left( \mathbb{X}(X,Y), \mathsf{d}_\mathcal{E} \right)$ is an ultrametric space. Next, let's prove that the ultrametric is compatible with the Cartesian $k$-differential structure by using the identities of the higher-order partial derivative from Lemma \ref{lem:partial}. \\

\noindent {\em (i)} By Faà di Bruno's formula \textbf{[HD.5]}, we have that $\partial^n[g \circ f]$ is a sum involving all the lower or equal degree derivatives $\partial^m[g]$ and $\partial^p[f]$, where $0 \leq m, p \leq n$, and some projections. Clearly if $\mathsf{d}_\mathcal{E}\left( g_1 \circ f_1, g_2 \circ f_2 \right) =0$, then by definition we have that $0 \leq \mathsf{max}\lbrace \mathsf{d}_\mathcal{E}(g_1, g_2), \mathsf{d}_\mathcal{E}(f_1, f_2) \rbrace$. So suppose that $\mathsf{d}\left( g_1 \circ f_1, g_2 \circ f_2 \right) = 2^{-n}$, which means that $\mathcal{E}\left( \partial^n[g_1 \circ f_1] \right) \neq \mathcal{E}\left( \partial^n[g_2 \circ f_2] \right)$. If $\mathsf{d}_\mathcal{E}(g_1, g_2)=0$, then $g_1 =g_2$, and therefore we have that $\mathsf{max}\lbrace \mathsf{d}_\mathcal{E}(g_1, g_2), \mathsf{d}_\mathcal{E}(f_1, f_2) \rbrace = \mathsf{d}_\mathcal{E}(f_1, f_2)$. Since $\mathcal{E}\left( \partial^n[g_1 \circ f_1] \right) \neq \mathcal{E}\left( \partial^n[g_1 \circ f_2] \right)$, it must be the case that $\mathcal{E}\left( \partial^m[f_1] \right) \neq \mathcal{E}\left( \partial^m[f_2] \right)$ for some $m \leq n$. This implies that $2^{-n} \leq 2^{-m} \leq \mathsf{d}_\mathcal{E}(f_1, f_2)= \mathsf{max}\lbrace \mathsf{d}_\mathcal{E}(g_1, g_2), \mathsf{d}_\mathcal{E}(f_1, f_2) \rbrace$. Similarly if $\mathsf{d}_\mathcal{E}(f_1, f_2)=0$, then $2^{-n} \leq \mathsf{d}_\mathcal{E}(g_1, g_2)= \mathsf{max}\lbrace \mathsf{d}_\mathcal{E}(g_1, g_2), \mathsf{d}_\mathcal{E}(f_1, f_2) \rbrace$. Now suppose that $\mathsf{d}_\mathcal{E}(g_1, g_2)=2^{-p}$ and $\mathsf{d}_\mathcal{E}(f_1, f_2)=2^{-m}$. If $\mathsf{max}\lbrace 2^{-m},2^{-p} \rbrace < 2^{-n}$, then $n < \mathsf{min}\lbrace m,p \rbrace$ which would imply that for all $0 \leq j \leq n$, $\mathcal{E}\left( \partial^j[g_1] \right) = \mathcal{E}\left( \partial^j[g_2] \right)$ and $\mathcal{E}\left( \partial^j[f_1] \right) = \mathcal{E}\left( \partial^j[f_2] \right)$. But then, by Faà di Bruno's formula, we would have that $\mathcal{E}\left( \partial^n[g_1 \circ f_1] \right) = \mathcal{E}\left( \partial^n[g_2 \circ f_2] \right)$, which is a contraction. So we must have that $2^{-n} \leq \mathsf{max}\lbrace 2^{-m},2^{-p} \rbrace$ as desired. So, we conclude that composition is non-expansive. \\

\noindent {\em (ii)} For simplicity, let us work with pairings $\langle f_0,f_1 \rangle$ and $\langle g_0, g_1 \rangle$. If $\mathsf{d}_\mathcal{E}\left( \langle f_0, f_1 \rangle, \langle g_0, g_1 \rangle \right)=0$ then by definition $0 \leq  \mathsf{max}\lbrace \mathsf{d}_\mathcal{E}(f_0, g_0),\mathsf{d}_\mathcal{E}(f_1, g_1) \rbrace$. So suppose that $\mathsf{d}_\mathcal{E}\left( \langle f_0,  f_1 \rangle, \langle g_0, g_1 \rangle \right) = 2^{-m}$. By \textbf{[HD.4]} and that $\mathcal{E}$ preserves pairings, we get that for all $0 \leq j < m$ that $\left\langle \mathcal{E}\left( \partial^j[f_0] \right), \mathcal{E}\left( \partial^j[f_1] \right) \right \rangle=\left\langle \mathcal{E}\left( \partial^j[g_0] \right), \mathcal{E}\left( \partial^j[g_1] \right) \right \rangle$, and also that $ \left\langle \mathcal{E}\left( \partial^m[f_0] \right), 
\mathcal{E}\left( \partial^m[f_1] \right) \right \rangle \neq \left\langle \mathcal{E}\left( \partial^m[g_0] \right), \mathcal{E}\left( \partial^m[g_1] \right) \right \rangle$. However, this implies that for both $i=0$ and $i=1$ that $\mathcal{E}\left( \partial^j[f_i] \right)=\mathcal{E}\left( \partial^j[g_i] \right)$ for all $0 \leq j < m$, but that $\mathcal{E}\left( \partial^m[f_0] \right) \neq \mathcal{E}\left( \partial^m[g_0] \right)$ or $\mathcal{E}\left( \partial^m[f_1] \right) \neq \mathcal{E}\left( \partial^m[g_1] \right)$. This implies that $\mathsf{max}\lbrace \mathsf{d}_\mathcal{E}(f_0, g_0), \mathsf{d}_\mathcal{E}(f_1, g_1) \rbrace = 2^{-n}$ as desired. So, we conclude that pairing is an isometry. Similarly, one can also show that for arbitrary tuplings that $\mathsf{d}_\mathcal{E}\left( \langle f_0, \hdots, f_n \rangle, \langle g_0, \hdots, g_n \rangle \right) = \mathsf{max}\lbrace \mathsf{d}_\mathcal{E}(f_0, g_0), \hdots, \mathsf{d}_\mathcal{E}(f_n, g_n) \rbrace$, and so tupling is an isometry as desired. \\

\noindent {\em (iii)} First observe that by the non-distinguishable axiom and that composition is non-expansive, it follows that for post-composition, we have that $\mathsf{d}_\mathcal{E}(g \circ f_1, g \circ f_2) \leq \mathsf{d}_\mathcal{E}(f_1, f_2)$. Next, observe that in any Cartesian left $k$-linear category that $r \cdot f + s \cdot g = (r \cdot \pi_0 + s \cdot \pi_1) \circ \langle f,g \rangle$. Using this identity, and since both post-composition and tuplings are non-expansive, it follows that $\mathsf{d}_\mathcal{E}(r \cdot f_1 + s \cdot g_1, r\cdot f_2 + s\cdot g_2) \leq \mathsf{max}\lbrace \mathsf{d}_\mathcal{E}(f_1, f_2), \mathsf{d}_\mathcal{E}(g_1, g_2) \rbrace$. So, we conclude that addition and scalar multiplication are non-expansive. \\

\noindent {\em (iv)} Recall that $\partial^0\left[\mathsf{D}[-] \right] = \partial^1[-]$, and by \textbf{[HD.8]} that $\partial^{n+1}\left[\mathsf{D}[-] \right]$ is computed out to be a sum in terms of $\partial^{n+2}[-]$ and $\partial^{n+1}[-]$, and involving some appropriate projections. Starting with the case that if all $n \in \mathbb{N}$, $\mathcal{E}\left( \partial^n[f] \right) = \mathcal{E}\left( \partial^n[g] \right)$, then $\mathsf{d}_\mathcal{E}(f,g)=0$, which implies $f=g$. Then $\mathsf{D}[f] = \mathsf{D}[g]$, and thus $\mathsf{d}_\mathcal{E}\left( \mathsf{D}[f], \mathsf{D}[g] \right) = 0$. On the other hand, if $n$ is the smallest natural number such that $\mathcal{E}\left( \partial^{n+1}[f] \right) \neq \mathcal{E}\left( \partial^{n+1}[g] \right)$, then for all $0 \leq j \leq n$ we have that $\mathcal{E}\left( \partial^{j}[f] \right) = \mathcal{E}\left( \partial^{j}[g] \right)$. Then by \textbf{[HD.8]}, since $\mathcal{E}$ preserves sums and projections, for all $0 \leq j \leq n$ we have that $\mathcal{E}\left( \partial^{j}\left[\mathsf{D}[f] \right] \right) = \mathcal{E}\left( \partial^{j}\left[\mathsf{D}[g] \right] \right)$ but that $\mathcal{E}\left( \partial^{n}\left[\mathsf{D}[f] \right] \right) \neq \mathcal{E}\left( \partial^{n}\left[\mathsf{D}[g] \right] \right)$, since $\partial^{n}\left[\mathsf{D}[-] \right]$ involves $\partial^{n+1}\left[-\right]$. Therefore, $\mathsf{d}_\mathcal{E}\left( \mathsf{D}[f], \mathsf{D}[g] \right) = 2^{-n}$. \\

\noindent {\em (v)} Note that $\mathsf{d}_\mathcal{E}\left( f, g \right) < 2^{-m}$ means precisely that for all $0 \leq j \leq m$, $\mathcal{E}\left( \partial^{j}[f] \right) = \mathcal{E}\left( \partial^{j}[g] \right)$. Now for the $\Rightarrow$ direction, since $(f_n)_{n \in \mathbb{N}}$ is Cauchy, for every $m \in \mathbb{N}$, there exists a smallest natural number $N_m$ such that for all $n \geq N_m$, $\mathsf{d}_\mathcal{E}\left( f_n, f_{n+1} \right) < 2^{-m}$ as desired. For the $\Leftarrow$ direction, let $\varepsilon >0$ be an arbitrary non-negative real number and let $m$ be the smallest natural number such that $2^{-m} \leq \varepsilon$. Then by assumption for all $n \geq N_m$, $\mathsf{d}_\mathcal{E}\left( f_n, f_{n+1} \right) < 2^{-m}$. As such, since we are working in an ultrametric space, we conclude that $(f_n)_{n \in \mathbb{N}}$ is Cauchy. \\

In order to prove that the hom-sets are also complete ultrametric spaces, let us define a new Cartesian $k$-differential category whose hom-sets are the completion of the hom-sets of $\mathbb{X}$. So define the Cartesian $k$-differential category $(\overline{\mathbb{X}}, \mathsf{D})$ to be the category whose objects are the same as $\mathbb{X}$ and whose hom-sets are $\overline{\mathbb{X}(X,Y)}$. So a map from $X$ to $Y$ in $\overline{\mathbb{X}}$ is an equivalence class of a Cauchy sequence $[(f_n)_{n \in \mathbb{N}}]$ in $\left( \mathbb{X}(X,Y), \mathsf{d}_\mathcal{E} \right)$. Composition, identities, the $k$-linear structure, and the product structure are defined on representatives component-wise as in $\mathbb{X}$ and the differential combinator $\mathsf{D}$ is defined similarly, that is, if $\mathsf{D}\left[ [(f_n)_{n \in \mathbb{N}}] \right] = \left[  \left( \mathsf{D}[f_n] \right)_{n \in \mathbb{N}} \right]$. We need to explain why this all is well-defined.

Since non-expansive maps and isometries preserve Cauchy sequences and the equivalence relation $\thicksim$, it follows that composition, tupling, addition, and scalar multiplication are all well-defined, so $\overline{\mathbb{X}}$ is indeed a Cartesian left $k$-linear category. It remains to explain why the differential combinator is well-defined, that is, why the differential combinator of $\mathbb{X}$ preserves Cauchy sequences and $\thicksim$. First observe that by $(iii)$, if $\mathsf{d}_\mathcal{E}(f,g) \leq 2^{-(n+1)}$ then $\mathsf{d}_\mathcal{E}\left( \mathsf{D}[f], \mathsf{D}[g] \right) \leq 2^{-n}$. Now let $(f_n)_{n \in \mathbb{N}}$ be a Cauchy sequence in $\left( \mathbb{X}(X,Y), \mathsf{d}_\mathcal{E} \right)$ and $m \in \mathbb{N}$. Then by $(v)$, we have an $N_{m+1} \in \mathbb{N}$ such that for all $n \geq N_{m+1}$, $\mathsf{d}(f_{n}, f_{n+1}) < 2^{-(m+1)}$. Therefore for all $n \geq N_{m+1}$, $\mathsf{d}_\mathcal{E}\left( \mathsf{D}[f_{n}], \mathsf{D}[f_{n+1}] \right) \leq 2^{-(m)}$, and so by $(iv)$ we conclude that $\left( \mathsf{D}[f_n] \right)_{n \in \mathbb{N}}$ is a Cauchy sequence in $\left( \mathbb{X}(X \times X,Y), \mathsf{d}_\mathcal{E} \right)$. Using a similar argument, we can also show $\mathsf{D}$ preserves $\thicksim$, that is, $(f_n)_{n \in \mathbb{N}} \thicksim (g_n)_{n \in \mathbb{N}}$ then $\left(\mathsf{D}[f]_n \right)_{n \in \mathbb{N}} \thicksim \left(\mathsf{D}[g]_n \right)_{n \in \mathbb{N}}$. Therefore, the differential combinator is well-defined, and thus $(\overline{\mathbb{X}}, \mathsf{D})$ is indeed a Cartesian $k$-differential category. 

Now observe that by $(v)$, if $(f_n)_{n \in \mathbb{N}}$ is a Cauchy sequence in $\left( \mathbb{X}(X,Y), \mathsf{d}_\mathcal{E} \right)$, then there exists a smallest natural number $N_0$ such that for all $n \leq N_0$ that $\mathcal{E}(f_n) = \mathcal{E}(f_{N_0})$. Therefore we obtain a Cartesian $k$-linear functor $\overline{\mathcal{E}}: \overline{\mathbb{X}} \to \mathbb{X}$ defined on objects as $\overline{\mathcal{E}}(X)$ and on maps as $\overline{\mathcal{E}}\left( [(f_n)_{n \in \mathbb{N}}] \right) = \mathcal{E}(f_{N_0})$, which indeed well-defined. By the couniversal property of $\left((\mathbb{X}, \mathsf{D}), \mathcal{E}\right)$, there exists a unique Cartesian $k$-differential functor ${\mathcal{E}^\flat_{\thicksim_{\mathcal{E}}}: (\overline{\mathbb{X}}, \mathsf{D}) \to \left(\mathbb{X}, \mathsf{D} \right)}$ which makes the following diagram commute: 
\begin{align*} \xymatrixcolsep{5pc}\xymatrix{\overline{\mathbb{X}}  \ar[dr]_-{\overline{\mathcal{E}}}  \ar@{-->}[r]^-{\exists! ~ \overline{\mathcal{E}}^\flat}  & \mathbb{X} \ar[d]^-{\mathcal{E}} \\
  & \mathbb{A} }
\end{align*}
We will now show that for a Cauchy sequence $(f_n)_{n \in \mathbb{N}}$ in $\left( \mathbb{X}(X,Y), \mathsf{d}_\mathcal{E} \right)$, it converges to the image of its equivalence class $\overline{\mathcal{E}}^\flat\left( [(f_n)_{n \in \mathbb{N}}]  \right)$. So let $\varepsilon > 0$ be an arbitrary real positive number and let $m$ be the smallest natural number such that $2^{-m} \leq \varepsilon$. By $(v)$, there exists a natural number $N$ such that for all $n \geq N$, $\mathsf{d}_{\mathcal{E}}(f_n, f_{n+1}) < 2^{-(m+1)}$, which implies that for all $0 \leq j \leq m+1$ that $\mathcal{E}\left( \partial^{j}[f_n] \right) = \mathcal{E}\left( \partial^{j}[f_{n+1}] \right)$. Now since $(f_n)_{n \in \mathbb{N}}$ is Cauchy, $(f_{N+n})_{n \in \mathbb{N}}$ is also a Cauchy sequence and $(f_n)_{n \in \mathbb{N}} \thicksim (f_{N+n})_{n \in \mathbb{N}}$, therefore $[(f_{n})_{n \in \mathbb{N}}] = [(f_{N+n})_{n \in \mathbb{N}}]$. Then for all $n \geq N$ and all $0 \leq j \leq m+1$, we have that $\mathcal{E}\left( \partial^{j}[f_n] \right) = \mathcal{E}\left( \partial^{j}[f_N] \right)$ and since $\overline{\mathcal{E}}^\flat$ commutes with the derivative operators we compute that: 
\begin{gather*}
    \mathcal{E}\left( \partial^{j}\left[ \overline{\mathcal{E}}^\flat\left( [(f_n)_{n \in \mathbb{N}}]  \right)\right] \right) = \mathcal{E}\left( \partial^{j}\left[ \overline{\mathcal{E}}^\flat\left( [(f_{N+n})_{n \in \mathbb{N}}]  \right)\right] \right) = \mathcal{E}\left( \overline{\mathcal{E}}^\flat\left(\partial^{j}\left[  [(f_{N+n})_{n \in \mathbb{N}}]  \right] \right)\right) \\
    = \mathcal{E}\left( \partial^{j}\left[  f_{N}  \right]\right) = \mathcal{E}\left( \partial^{j}[f_n] \right) 
\end{gather*}
Therefore, $\mathcal{E}\left( \partial^{j}\left[ \overline{\mathcal{E}}^\flat\left( [(f_n)_{n \in \mathbb{N}}]  \right)\right] \right) = \mathcal{E}\left( \partial^{j}[f_n] \right)$ for all $0 \leq j \leq m+1$. Thus we have that $\mathsf{d}_\mathcal{E}\left(f_n, \overline{\mathcal{E}}^\flat\left( [(f_n)_{n \in \mathbb{N}}]  \right) \right) \leq 2^{-(m+1)} < 2^{-m} \leq \varepsilon$. So the Cauchy sequence $(f_n)_{n \in \mathbb{N}}$ converges to $\overline{\mathcal{E}}^\flat\left( [(f_n)_{n \in \mathbb{N}}]  \right)$ and we conclude that  $\left( \mathbb{X}(X,Y), \mathsf{d}_\mathcal{E} \right)$ is a complete ultrametric space.  
\end{proof}

\begin{remark} \normalfont Complete ultrametric spaces and non-expansive maps between them form a Cartesian closed category. Proposition \ref{prop:um} says that, in fact, a cofree Cartesian $k$-differential category is enriched over complete ultrametric spaces as a Cartesian left $k$-linear category but not as a Cartesian $k$-differential category since the differential combinator is not a non-expansive map. That said composition, the $k$-linear structure, the tupling operation, and the differential combinator are all \emph{continuous}. 
\end{remark}

\begin{example} \normalfont Let us take a look at what the induced ultrametric and limits of Cauchy sequences are in our examples of cofree Cartesian $k$-differential categories. 
\begin{enumerate}[{\em (i)}]
\item For a Cartesian left $k$-linear category $\mathbb{A}$, in its Faà di Bruno construction, the distance between two Faà di Bruno sequences is given by the smallest $n \in \mathbb{N}$ where the maps in the sequence are different: 
\begin{align*}
\mathsf{d}_{\mathcal{E}_\mathbb{A}}(f_{\bullet},g_{\bullet}) = \begin{cases} 2^{-n} & \text{ where $n \!\in\! \mathbb{N}$ is the smallest natural number such that $f_n \!\neq\! g_n$} \\
0 & \text{ if all $n \in \mathbb{N}$, $f_n = g_n$} 
\end{cases}
\end{align*}
A sequence $({f_n}_{\bullet})_{n \in \mathbb{N}}$ of Faà di Bruno sequences is Cauchy if and only if the terms eventually stabilise, that is, for every $m\in \mathbb{N}$, there exists a smallest natural number $N_m$ such that for all $n \geq N_m$ and $0 \leq j \leq m$, ${f_{n}}_{j} = {f_{n+1}}_{j}$. Then a Cauchy sequence $({f_n}_{\bullet})_{n \in \mathbb{N}}$ converges to the Faà di Bruno sequence ${f_{N_\bullet}}_{\bullet} = ({f_{N_0}}_{0}, {f_{N_1}}_{1}, \hdots)$. 
\item For a $k$-linear category $\mathbb{B}$ with finite biproducts, in $\mathbb{B}^\Delta$, there are only three possible distances between maps: none, halfway, and completely apart. 
\begin{align*}
\mathsf{d}_{\mathcal{P}_\mathbb{B}}\left( (f_1, f_2), (g_1, g_2) \right) = \begin{cases} 1 & \text{ if $f_1 \neq f_2$} \\
\frac{1}{2} & \text{ if $f_1= f_2$ but $g_1 \neq g_2$ } \\ 
0 & \text{ if $f_1= f_2$ and $g_1 = g_2$} 
\end{cases}
\end{align*} 
A sequence $\left( (f_n, g_n) \right)_{n \in \mathbb{N}}$ is Cauchy if and only if it is eventually constant, that is, there exists an $N \in \mathbb{N}$ such that for all $n \geq N$, $(f_n,g_n) = (f_{n+1}, g_{n+1})$. Then a Cauchy sequence $\left( (f_n, g_n) \right)_{n \in \mathbb{N}}$ converges to the map it eventually repeats. 
\item In $k\text{-}\mathsf{MOD}^\mathcal{Q}$, the distance between $f: \mathcal{Q}(M) \to M^\prime$ and $g: \mathcal{Q}(M) \to M^\prime$ is given by the smallest degree on which they differ on pure symmetrized tensors: 
\begin{align*}
\mathsf{d}_{\mathcal{E}^\mathcal{Q}}(f,g) = \begin{cases} 2^{-n} & \text{ where $n \in \mathbb{N}$ is the smallest natural number such that} \\
& ~ \text{$f( \vert x_1, \hdots, x_n \rangle_{x_0} ) \neq g( \vert x_1, \hdots, x_n \rangle_{x_0} )$ for some $x_j \in M$} \\
& \\ 
0 & \text{ if $f=g$} 
\end{cases}
\end{align*} 
A sequence $(f_n)_{n \in \mathbb{N}}$ of $k$-linear morphism $f_n: \mathcal{Q}(M) \to M^\prime$ is Cauchy if and only if for every $m\in \mathbb{N}$ there exists a smallest natural number $N_m$ such that for all $n \geq N_m$, $f_n$ and $f_{n+1}$ agree all pure symmetrized tensors of degree $0 \leq j \leq m$, that is, $f_n(\vert 1 \rangle_{x_0}) = f_{n+1}(\vert 1 \rangle_{x_0})$ and $f_n( \vert x_1, \hdots, x_j \rangle_{x_0} ) = f_{n+1}( \vert x_1, \hdots, x_j \rangle_{x_0} )$. Then a Cauchy sequence $(f_n)_{n \in \mathbb{N}}$ converges to the $k$-linear morphism $f: \mathcal{Q}(M) \to M^\prime$ which is defined on pure symmetrized tensors as $f(\vert 1 \rangle_{x_0}) = f_{N_0}(\vert 1 \rangle_{x_0})$ and $f( \vert x_1, \hdots, x_n \rangle_{x_0} )= f_{N_n}( \vert x_1, \hdots, x_n \rangle_{x_0} )$. 
\end{enumerate}
\end{example}

\section{Differential Constants}\label{sec:dcon}

In this section, we discuss differential constants, which are maps whose derivative is zero. In particular, the main result of this section is that for a cofree Cartesian $k$-differential category, the differential constants correspond precisely to the maps of the base category. Therefore, we may recapture the base category internally in a cofree Cartesian $k$-differential category using its differential constants. This will allow us to provide a base independent description of cofree Cartesian $k$-differential categories in Section \ref{sec:ChDC}. 

\begin{definition}\label{def:dcon} In a Cartesian $k$-differential category $(\mathbb{X}, \mathsf{D})$, a map $f: X \to Y$ is a \textbf{differential constant}, or $\mathsf{D}$-constant for short, if $\mathsf{D}[f] = 0$.
\end{definition}

\begin{example} \normalfont Here are the differential constants in our main examples of Cartesian $k$-differential categories. 
\begin{enumerate}[{\em (i)}]
\item In $k\text{-}\mathsf{POLY}$, every polynomial of degree zero, so a polynomial which is only a constant term, is a differential constant. However the converse is not necessarily true, that is, there could be differential constants that are not constant polynomials. For example, if $k=\mathbb{Z}_2$, then $x^2$ is a differential constant.
\item In $\mathsf{SMOOTH}$, the only differential constants are the constant functions. 
\item In a $k$-linear category $\mathbb{B}$ with finite biproducts, the only $\mathsf{D}^{\mathsf{lin}}$-constants are the zero maps. 
\item For a Cartesian left $k$-linear category $\mathbb{A}$, in its Faà di Bruno construction, the differential constants are the Faà di Bruno sequences of the form $(f, 0, 0, \hdots)$ for any map $f$ in $\mathbb{A}$. 
\item For a $k$-linear category $\mathbb{B}$ with finite biproducts, in $\mathbb{B}^\Delta$ the $\mathsf{D}^\Delta$-constants are the maps of the form $(f,0)$ for any map $f$ in $\mathbb{B}$.
\item In $k\text{-}\mathsf{MOD}^\mathcal{Q}$, the differential constants are the $k$-linear morphisms $f: \mathcal{Q}(M) \to M^\prime$ such that for all pure symmetrized tensors of degree $n \geq 1$, $f\left(\vert x_1, \hdots, x_n \rangle_{x_0} \right) = 0$. 
\end{enumerate}
\end{example}

Here are now some basic properties of differential constants. 

\begin{lemma}\label{lem:dconlem1} In a Cartesian $k$-differential category $(\mathbb{X}, \mathsf{D})$
    \begin{enumerate}[{\em (i)}]
    \item \label{dconlin} $f$ is $\mathsf{D}$-constant and $\mathsf{D}$-linear if and only if $f=0$;
\item \label{dconcomp} If $f$ or $g$ is a $\mathsf{D}$-constant, then their composite $g \circ f$ is a $\mathsf{D}$-constant;
\item \label{dconadd} The zero maps $0$ are $\mathsf{D}$-constant, and if $f$ and $g$ are $\mathsf{D}$-constants then for any $r,s \in k$, $r \cdot f + s \cdot g$ is a $\mathsf{D}$-constant; 
\item \label{dconpair} If $f_0$, ..., $f_n$ are $\mathsf{D}$-constants, then the tuple $\langle f_0, \hdots, f_n \rangle$ is a $\mathsf{D}$-constant.
\end{enumerate}
Furthermore, if $\mathcal{F}: (\mathbb{X}, \mathsf{D}) \to (\mathbb{Y}, \mathsf{D})$ is a Cartesian $k$-differential functor, if $f$ is a differential constant in $(\mathbb{X}, \mathsf{D})$ then $\mathcal{F}(f)$ is a differential constant in $(\mathbb{Y}, \mathsf{D})$.
\end{lemma}
\begin{proof} These are straightforward to prove, so we leave them as an exercise for the reader. 
\end{proof}

The above lemma says that differential constants are closed under composition, addition, scalar multiplication, and tupling. Therefore, it is tempting to construct a ``subcategory'' of differential constants that is also Cartesian $k$-left linear. However, in general, this is not possible. Crucially, what is missing is an identity map in this ``subcategory'' of differential constants. Indeed, by \textbf{[CD.3]}, identity maps are not differential constants, except for the identity map of the terminal object (since, in this case, it is also a zero map). To solve this issue, we introduce a new notion called a differential constant unit, which behaves like an identity map for differential constants. 

\begin{definition}\label{def:varsigma} For a Cartesian $k$-differential category $(\mathbb{X}, \mathsf{D})$, a \textbf{differential constant unit} $\varsigma$, or $\mathsf{D}$-constant unit for short, is a family of endomorphisms indexed by objects of $\mathbb{X}$, $\varsigma = \lbrace \varsigma_X: X \to X \vert~ X \in \mathbb{X} \rbrace$ such that $\varsigma_X$ is a $\mathsf{D}$-constant and for every $\mathsf{D}$-constant $f: X \to Y$, $\varsigma_Y \circ f = f = f \circ \varsigma_X$. 
\end{definition}

\begin{lemma}\label{lem:sigma1} For a Cartesian $k$-differential category $(\mathbb{X}, \mathsf{D})$, if a $\mathsf{D}$-constant unit exists, then it is unique. Furthermore, if $(\mathbb{X}, \mathsf{D})$ has a $\mathsf{D}$-constant unit $\varsigma$, then: 
\begin{enumerate}[{\em (i)}]
\item \label{lem:sigma1.i} $\varsigma: 1_{\mathbb{X}} \Rightarrow 1_{\mathbb{X}}$ is a natural transformation, that is, for every map $f: X \to Y$, $\varsigma_Y \circ f = f \circ \varsigma_X$;
\item \label{lem:sigma1.ii} $\varsigma_X$ is an idempotent, that is, $\varsigma_X \circ \varsigma_X = \varsigma_X$;
\item \label{lem:sigma1.iii} $\varsigma_X$ is $k$-linear;
\item \label{lem:sigma1.iv} $\varsigma_{\ast} = 1_\ast$ and $\varsigma_{X_0 \times \hdots \times X_n} = \varsigma_{X_0} \times \hdots \times \varsigma_{X_n}$. 
\end{enumerate}
\end{lemma}
\begin{proof} We begin by proving uniqueness of differential constant units. So suppose that $\varsigma$ and $\varsigma^\prime$ are both $\mathsf{D}$-constant units for $(\mathbb{X}, \mathsf{D})$. Then $\varsigma$ and $\varsigma^\prime$ are both $\mathsf{D}$-constants and so we have that: $\varsigma = \varsigma^\prime \circ \varsigma = \varsigma^\prime$. Therefore, we conclude that differential constants are unique. Now suppose that $\varsigma$ is a $\mathsf{D}$-constant unit for $(\mathbb{X}, \mathsf{D})$. For any map $f$, by Lemma \ref{lem:dconlem1}.(\ref{dconcomp}), $\varsigma \circ f$ and $f \circ \varsigma$ are both $\mathsf{D}$-constants. Then we have that $\varsigma \circ f = \varsigma \circ f \circ \varsigma = f \circ \varsigma$. So, we conclude that $\varsigma$ is a natural transformation. That $\varsigma_X$ is an idempotent is a consequence of the fact that $\varsigma_X$ is itself a $\mathsf{D}$-constant and so by definition of being a $\mathsf{D}$-constant unit, we get that $\varsigma_X \circ \varsigma_X = \varsigma_X$. While that $\varsigma_X$ is $k$-linear follows from naturality of $\varsigma$. Lastly, $\varsigma_{\ast} = 1_\ast$ is automatic by definition of a terminal object, while $\varsigma_{X_0 \times \hdots \times X_n} = \varsigma_{X_0} \times \hdots \times \varsigma_{X_n}$ follows from naturality of $\varsigma$ and the universal property of the product. 
\end{proof}

Using the differential constant unit we can build a category of differential constants of a Cartesian $k$-differential category. However, this will not be a subcategory since the identity maps do not correspond to the identity maps in the starting Cartesian $k$-differential category. So in particular, we do not always have an inclusion functor from this category of differential constants back to the Cartesian $k$-differential category.   

\begin{lemma} \label{lem:DCONsub} Let $(\mathbb{X}, \mathsf{D})$ be a Cartesian $k$-differential category with a $\mathsf{D}$-constant unit $\varsigma$. Then $\mathsf{D}\text{-}\mathsf{con}\left[ \mathbb{X} \right]$ is a Cartesian left $k$-linear category where: 
\begin{enumerate}[{\em (i)}]
\item The objects of $\mathsf{D}\text{-}\mathsf{con}\left[ \mathbb{X} \right]$ are the same as $\mathbb{X}$;
\item The maps in $\mathsf{D}\text{-}\mathsf{con}\left[ \mathbb{X} \right]$ are differential constants in $\mathbb{X}$; 
\item Composition is defined as in $\mathbb{X}$;
\item The identity of $X$ in $\mathsf{D}\text{-}\mathsf{con}\left[ \mathbb{X} \right]$ is the $\mathsf{D}$-constant unit $\varsigma_X: X \to X$;  
\item The terminal object and the product of objects are the same as in $\mathbb{X}$;
\item The projections in $\mathsf{D}\text{-}\mathsf{con}\left[ \mathbb{X} \right]$ are $\varsigma_{X_j} \circ \pi_j: X_0 \times \hdots \times X_n \to X_j$ and where the tupling is as in $\mathbb{X}$;
\item The $k$-linear structure is defined as in $\mathbb{X}$. 
\end{enumerate}
Furthermore, the functor $\mathcal{E}_{\varsigma}: \mathbb{X} \to \mathsf{D}\text{-}\mathsf{con}\left[ \mathbb{X} \right]$ defined on objects as $\mathcal{E}_{\varsigma}(X) = X$ and on maps as $\mathcal{E}_{\varsigma}(f) = \varsigma \circ f$, is a Cartesian $k$-linear functor. 
\end{lemma}    
\begin{proof} By definition of being a $\mathsf{D}$-constant unit, $\varsigma$ is a well-defined identity map in $\mathsf{D}\text{-}\mathsf{con}\left[ \mathbb{X} \right]$, and so $\mathsf{D}\text{-}\mathsf{con}\left[ \mathbb{X} \right]$ is indeed a category. The $k$-linear structure is well-defined by  Lemma \ref{lem:dconlem1}.(\ref{dconadd}), and since $\mathbb{X}$ is a left $k$-linear category, so is $\mathsf{D}\text{-}\mathsf{con}\left[ \mathbb{X} \right]$. By Lemma \ref{lem:dconlem1}.(\ref{dconcomp}), the projections are well-defined and since both $\varsigma$ and $\pi_j$ are $k$-linear, their composite $\varsigma \circ \pi_j$ is also $k$-linear. By Lemma \ref{lem:dconlem1}.(\ref{dconpair}) the tupling is well-defined and it easy to see that it will satisfy the necessary universal property since it does so in $\mathbb{X}$. So we conclude that $\mathsf{D}\text{-}\mathsf{con}\left[ \mathbb{X} \right]$ is a Cartesian left $k$-linear category. Lastly, since $\varsigma$ is idempotent and natural, $\mathcal{E}_{\varsigma}: \mathbb{X} \to \mathsf{D}\text{-}\mathsf{con}\left[ \mathbb{X} \right]$ is indeed a functor, and it is clear that it is Cartesian $k$-linear. 
\end{proof}

We will now prove that every cofree Cartesian $k$-differential category has a differential constant unit and that the base category is isomorphic to its category of differential constants. To do so we will need the following useful lemma: 

\begin{lemma}\label{lem:Econ}Let $\left((\mathbb{X}, \mathsf{D}), \mathcal{E}\right)$ be a cofree Cartesian $k$-differential category over a Cartesian left $k$-linear category $\mathbb{A}$. If $f$ and $g$ are parallel differential constants in $\mathbb{X}$, then $f=g$ if and only if $\mathcal{E}(f) = \mathcal{E}(g)$.
\end{lemma}
\begin{proof} The $\Rightarrow$ direction is automatic. For the $\Leftarrow$ direction, suppose that $\mathcal{E}(f) = \mathcal{E}(g)$. Since $f$ and $g$ are $\mathsf{D}$-constants, then for all $n \in \mathbb{N}$, $\partial^{n+1}[f] = 0 = \partial^{n+1}[g]$. Then for all $n \in \mathbb{N}$, $\mathcal{E}\left( \partial^n[f] \right) = \mathcal{E}\left( \partial^n[g] \right)$. Then by Lemma \ref{lem:E=}, $f=g$.  
\end{proof}

\begin{proposition}\label{prop:cofree-diffcon} Let $\left((\mathbb{X}, \mathsf{D}), \mathcal{E}\right)$ be a cofree Cartesian $k$-differential category over a Cartesian left $k$-linear category $\mathbb{A}$. Then $\mathbb{X}$ has a differential constant unit $\varsigma$ such that $\mathcal{E}(\varsigma_X) = 1_{\mathcal{E}(X)}$. Furthermore, the functor $\overline{\mathcal{E}}_{\varsigma}: \mathsf{D}\text{-}\mathsf{con}\left[ \mathbb{X} \right] \to \mathbb{A}$, defined on objects and maps as $\overline{\mathcal{E}}_{\varsigma}(-) = \mathcal{E}(-)$, is a Cartesian $k$-linear isomorphism, so $\mathsf{D}\text{-}\mathsf{con}\left[ \mathbb{X} \right] \cong \mathbb{A}$, and it is the unique Cartesian $k$-linear isomorphism such that the following diagram commutes: 
 \begin{align*} \xymatrixcolsep{5pc}\xymatrix{ &   \mathbb{X} \ar[dr]^-{\mathcal{E}} \ar[dl]_-{\mathcal{E}_\varsigma} \\
\mathsf{D}\text{-}\mathsf{con}\left[ \mathbb{X} \right] \ar@{-->}[rr]_-{\exists! ~\overline{\mathcal{E}}_\varsigma}^-{\cong}  & & \mathbb{A} }
\end{align*}
In other words, for every map $f$ in $\mathbb{A}$, there exists a unique $\mathsf{D}$-constant $f^\sharp$ in $\mathbb{X}$ such that $\mathcal{E}(f^\sharp)=f$. 
\end{proposition}
\begin{proof} Consider the Cartesian $k$-differential category $\left( \mathbb{X}^\Delta, \mathsf{D}^\Delta \right)$ and the Cartesian $k$-linear functor $\mathcal{P}_{\mathbb{X}}: \mathbb{X}^\Delta \to \mathbb{X}$ as defined in Lemma \ref{lem:Adelta}. Therefore by the couniversal property of $\left((\mathbb{X}, \mathsf{D}), \mathcal{E}\right)$, there exists a unique Cartesian $k$-differential functor $\mathcal{R}: \left( \mathbb{X}^\Delta, \mathsf{D}^\Delta \right) \to \left( \mathbb{X}, \mathsf{D} \right)$ such that the following diagram commutes: 
\begin{align*} \xymatrixcolsep{5pc}\xymatrix{\mathbb{X}^\Delta \ar[d]_-{\mathcal{P}_{\mathbb{X}}}  \ar@{-->}[r]^-{\exists! ~ \mathcal{R}}  & \mathbb{X} \ar[d]^-{\mathcal{E}} \\
 \mathbb{X} \ar[r]_-{\mathcal{E}}  & \mathbb{A} }
\end{align*}
Note that on objects, $\mathcal{E}(\mathcal{R}(X)) = \mathcal{E}(X)$. By Corollary \ref{cor:Eobj}, $\mathcal{E}$ is bijective on objects, and therefore $\mathcal{R}(X) = X$. Then define $\varsigma_X: 1_X \to 1_X$ as $\varsigma_X := \mathcal{R}(1_X, 0)$. First, note that by definition, we easily compute that: 
\[ \mathcal{E}(\varsigma_X) = \mathcal{E}(\mathcal{R}(1_X, 0)) = \mathcal{E}(\mathcal{P}_{\mathbb{X}}(1_X, 0)) = \mathcal{E}(1_X) = 1_{\mathcal{E}(X)} \]
So $\mathcal{E}(\varsigma_X) = 1_{\mathcal{E}(X)}$ as desired. Next observe that $(1_X, 0)$ is a differential constant in $\mathbb{X}^\Delta$, and thus by Lemma \ref{lem:dconlem1}, $\varsigma_X$ is also a differential constant. Now, for any differential constant $f: X \to Y$ in $\mathbb{X}$, note that $\varsigma_Y \circ f$ and $f\circ \varsigma_X$ are also differential constants. Since $\mathcal{E}(\varsigma_X) = 1_{\mathcal{E}(X)}$, it follows that $\mathcal{E}(\varsigma_Y \circ f) = \mathcal{E}(f) = \mathcal{E}(f\circ \varsigma_X)$. Then, by Lemma \ref{lem:Econ}, we have that $\varsigma_Y \circ f = f = f \circ \varsigma_X$. So, we conclude that $\varsigma$ is a differential constant unit. 

Now $\overline{\mathcal{E}}_{\varsigma}: \mathsf{D}\text{-}\mathsf{con}\left[ \mathbb{X} \right] \to \mathbb{A}$ will be a Cartesian $k$-linear functor since $\mathcal{E}$ is. On objects, $\overline{\mathcal{E}}_{\varsigma}(\mathcal{E}_\varsigma(X)) = \mathcal{E}(X)$, while on maps, since $\mathcal{E}(\varsigma_X) = 1_{\mathcal{E}(X)}$, it follows by definition that $\overline{\mathcal{E}}_{\varsigma}(\mathcal{E}_\varsigma(f)) = \mathcal{E}(f)$. Therefore, $\overline{\mathcal{E}}_{\varsigma} \circ \mathcal{E}_\varsigma = \mathcal{E}$. Uniqueness follows from the fact that if $f$ is differential constant, then $\mathcal{E}_\varsigma(f) = f$. To show that $\overline{\mathcal{E}}_{\varsigma}$ is an isomorphism, it is sufficient to show that it is bijective on objects and maps. By Corollary \ref{cor:Eobj}, $\overline{\mathcal{E}}_{\varsigma}$ is bijective on objects by definition. On the other hand, we will show that $\overline{\mathcal{E}}_{\varsigma}$ is injective and surjective on maps. That $\overline{\mathcal{E}}_{\varsigma}$ is injective on maps follows from Lemma \ref{lem:Econ}. For surjectivity, let $f: A \to B$ be a map in $\mathbb{A}$. Consider the Cartesian $k$-differential category $\left( \mathbb{A}^\Delta, \mathsf{D}^\Delta \right)$ and the Cartesian $k$-linear functor $\mathcal{P}_{\mathbb{A}}: \mathbb{A}^\Delta \to \mathbb{A}$ as defined in Lemma \ref{lem:Adelta}. Therefore by the couniversal property of $\left((\mathbb{X}, \mathsf{D}), \mathcal{E}\right)$, there exists a unique Cartesian $k$-differential functor $\mathcal{P}_{\mathbb{A}}^\flat: \left( \mathbb{A}^\Delta, \mathsf{D}^\Delta \right) \to \left( \mathbb{X}, \mathsf{D} \right)$ such that the following diagram commutes: 
\begin{align*} \xymatrixcolsep{5pc}\xymatrix{\mathbb{A}^\Delta \ar[dr]_-{\mathcal{P}_{\mathbb{A}}}  \ar@{-->}[r]^-{\exists! ~ \mathcal{P}_{\mathbb{A}}^\flat}  & \mathbb{X} \ar[d]^-{\mathcal{E}} \\
   & \mathbb{A} }
\end{align*}
Note that $(f,0)$ is a differential constant in $\mathbb{A}^\Delta$, and so by Lemma \ref{lem:dconlem1}, $\mathcal{P}_{\mathbb{A}}^\flat(f,0)$ is also a differential constant in $\mathbb{X}$. So $\mathcal{P}_{\mathbb{X}}^\flat(f,0)$ is a map in $\mathsf{D}\text{-}\mathsf{con}\left[ \mathbb{X} \right]$. By definition we have that $\overline{\mathcal{E}}_{\varsigma}\left( \mathcal{P}_{\mathbb{A}}^\flat(f,0)\right) = f$. As such, we have that $\overline{\mathcal{E}}_{\varsigma}$ is surjective on maps, and thus $\overline{\mathcal{E}}_{\varsigma}$ is bijective on maps. So we get that $\overline{\mathcal{E}}_{\varsigma}$ is an isomorphism as desired. Explicitly, the inverse $\overline{\mathcal{E}}^{-1}_{\varsigma}: \mathbb{A} \to \mathsf{D}\text{-}\mathsf{con}\left[ \mathbb{X} \right]$ is defined on objects as $\overline{\mathcal{E}}^{-1}_{\varsigma}(A) = \mathcal{P}_{\mathbb{A}}^\flat(A)$ and on maps as $\overline{\mathcal{E}}^{-1}_{\varsigma}(f) = \mathcal{P}_{\mathbb{A}}^\flat(f,0)$. Therefore, we conclude that $\overline{\mathcal{E}}_{\varsigma}$ is a Cartesian $k$-linear isomorphism, and so $\mathsf{D}\text{-}\mathsf{con}\left[ \mathbb{X} \right] \cong \mathbb{A}$. 
\end{proof}

Note that since differential constant units are unique, the differential constant unit of a cofree Cartesian $k$-differential category is independent of the base Cartesian left $k$-linear category. As such, it more or less immediately follows that for cofree Cartesian $k$-differential categories, base categories are unique up to isomorphism.

\begin{corollary}\label{lem:cofreebase} Let $\mathbb{A}$ and $\mathbb{A}^\prime$ be Cartesian left $k$-linear categories, and let $(\mathbb{X}, \mathsf{D})$ be a Cartesian $k$-differential category. Then if $\left((\mathbb{X}, \mathsf{D}), \mathcal{E}\right)$ is a cofree Cartesian $k$-differential category over $\mathbb{A}$ and $\left((\mathbb{X}, \mathsf{D}), \mathcal{E}^\prime\right)$ is a cofree Cartesian $k$-differential category over $\mathbb{A}^\prime$, then there exists a unique Cartesian $k$-linear isomorphism $\mathcal{A}: \mathbb{A} \to \mathbb{A}^\prime$, so $\mathbb{A} \cong \mathbb{A}^\prime$, such that the following diagram commutes: 
 \begin{align*} \xymatrixcolsep{5pc}\xymatrix{ &   \mathbb{X} \ar[dr]^-{\mathcal{E}^\prime} \ar[dl]_-{\mathcal{E}} \\
\mathbb{A} \ar@{-->}[rr]_-{\exists! ~\mathcal{A}}^-{\cong}  & & \mathbb{A}^\prime }
\end{align*}
\end{corollary}
\begin{proof} By Proposition \ref{prop:cofree-diffcon}, we get a chain of unique isomorphisms $\mathbb{A}^\prime \cong \mathsf{D}\text{-}\mathsf{con}\left[ \mathbb{X} \right] \cong \mathbb{A}$ which is compatible with $\mathcal{E}$ and $\mathcal{E}^\prime$ in the desired way. 
\hfill \end{proof}

As such, we may now provide another characterisation of cofree Cartesian $k$-differential categories as having a differential constant unit and being cofree over their category of differential constants. 

\begin{theorem} \label{thm:diffcon1} For a Cartesian $k$-differential category $(\mathbb{X}, \mathsf{D})$, the following are equivalent: 
\begin{enumerate}[{\em (i)}]
\item $(\mathbb{X}, \mathsf{D})$ is cofree; 
\item $(\mathbb{X}, \mathsf{D})$ has a $\mathsf{D}$-constant unit $\varsigma$ such that the triple $\left((\mathbb{X}, \mathsf{D}), \mathcal{E}_{\varsigma} \right)$ is a cofree Cartesian $k$-differential category over $\mathsf{D}\text{-}\mathsf{con}\left[ \mathbb{X} \right]$;
\item \label{thm:diffcon1.iii} $(\mathbb{X}, \mathsf{D})$ has a $\mathsf{D}$-constant unit $\varsigma$ and the induced unique Cartesian $k$-differential functor $\mathcal{E}_\varsigma^\flat: (\mathbb{X}, \mathsf{D}) \to (\mathsf{F}\left[\mathsf{D}\text{-}\mathsf{con}\left[ \mathbb{X} \right] \right], \mathsf{D})$ which makes the following diagram commute
\begin{align*} \xymatrixcolsep{5pc}\xymatrix{\mathbb{X} \ar[dr]_-{\mathcal{E}_\varsigma}  \ar@{-->}[r]^-{\exists! ~ \mathcal{E}^\flat}  & \mathsf{F}\left[\mathsf{D}\text{-}\mathsf{con}\left[ \mathbb{X} \right] \right]\ar[d]^-{\mathcal{E}_{\mathsf{D}\text{-}\mathsf{con}\left[ \mathbb{X} \right]}} \\
  & \mathsf{D}\text{-}\mathsf{con}\left[ \mathbb{X} \right] }
\end{align*}
which is explicitly defined on objects as $\mathcal{E}_\varsigma^\flat(X) = X$ and on maps as $\mathcal{E}_\varsigma^\flat(f) = \left( \varsigma \circ \partial^n[f] \right)_{n \in \mathbb{N}}$, is an isomorphism, so $\mathbb{X} \cong \mathsf{F}\left[\mathsf{D}\text{-}\mathsf{con}\left[ \mathbb{X} \right] \right]$. 
\end{enumerate}
\end{theorem}
\begin{proof} Observe that $(ii) \Leftrightarrow (iii)$ is a specific case of Lemma \ref{lem:cofree-faa2}. Therefore, it suffices to prove that $(i) \Leftrightarrow (ii)$. By definition, $(ii) \Rightarrow (i)$ is automatic. For $(i) \Rightarrow (ii)$, let $\left((\mathbb{X}, \mathsf{D}), \mathcal{E}\right)$ be a cofree Cartesian $k$-differential category over a Cartesian left $k$-linear category $\mathbb{A}$. By Proposition \ref{prop:cofree-diffcon}, $\mathbb{X}$ has a differential constant unit $\varsigma$ and $\overline{\mathcal{E}}_{\varsigma}: \mathsf{D}\text{-}\mathsf{con}\left[ \mathbb{X} \right] \to \mathbb{A}$ is a Cartesian $k$-linear isomorphism, such that $\overline{\mathcal{E}}_{\varsigma}^{-1} \circ \mathcal{E} = \mathcal{E}_\varsigma$. Then by Lemma \ref{lem:cofree-lem0}.(\ref{lem:cofreebase.ii}), $\left((\mathbb{X}, \mathsf{D}), \mathcal{E}_{\varsigma} \right)$ is a cofree Cartesian $k$-differential category over $\mathsf{D}\text{-}\mathsf{con}\left[ \mathbb{X} \right]$. 
\end{proof}

Theorem \ref{thm:diffcon1}.(\ref{thm:diffcon1.iii}) will be particularly useful in the proofs of Theorem \ref{thm:2} and Theorem \ref{thm:faaalg-cofree} below, which provide base independent characterisations of cofree Cartesian $k$-differential categories. We can also apply the results of Section \ref{sec:cofreelin} to describe the differential linear and $k$-linear maps of a cofree Cartesian $k$-differential category as $k$-linear differential constants.  

\begin{corollary} Let $(\mathbb{X}, \mathsf{D})$ be a Cartesian $k$-differential category which is cofree. Then $\mathsf{D}\text{-}\mathsf{lin}\left[ \mathbb{X} \right] \cong k\text{-}\mathsf{lin}\left[ \mathsf{D}\text{-}\mathsf{con}\left[ \mathbb{X} \right] \right]$ and $k\text{-}\mathsf{lin}\left[ \mathbb{X} \right] \cong k\text{-}\mathsf{lin}\left[ \mathsf{D}\text{-}\mathsf{con}\left[ \mathbb{X} \right] \right]^\Delta$.  
\end{corollary}

\begin{example}\label{ex:diffconunit} \normalfont Here are the induced differential constant units in our examples of cofree Cartesian $k$-differential categories. 
\begin{enumerate}[{\em (i)}]
\item\label{ex:diffcoununit.faa} For a Cartesian left $k$-linear category $\mathbb{A}$, in its Faà di Bruno construction, the differential constant unit is the Faà di Bruno sequence ${\varsigma_\bullet}_X: X \to X$ defined as the identity map in the first term and zero for the rest of the sequence, ${\varsigma_\bullet}_X := (1_X, 0, 0, \hdots)$. 
\item For a $k$-linear category $\mathbb{B}$ with finite biproducts, in $\mathbb{B}^\Delta$ the $\mathsf{D}^\Delta$-constant unit is the pair of an identity map with a zero map, $\varsigma_A:= (1_A,0)$. 
\item In $k\text{-}\mathsf{MOD}^\mathcal{Q}$, the differential constant unit is the $k$-linear morphism $\varsigma_M: \mathcal{Q}(M) \to M$ defined on pure symmetrized tensors as follows
\begin{align*}
    \varsigma_M \left( \vert 1 \rangle_{(x_0)} \right) = x_0 && \varsigma_M\left(\vert x_1, \hdots, x_n \rangle_{x_0} \right) = 0
\end{align*}
\end{enumerate}
\end{example}

\begin{example}\label{ex:notcofree} \normalfont We may now also explain why the Cartesian $k$-differential categories given by differentiating smooth functions, polynomials, and linear maps are not cofree. 
\begin{enumerate}[{\em (i)}]
\item Suppose that $k\text{-}\mathsf{POLY}$ had a differential constant unit. So in particular we would have a map $\varsigma_1: 1 \to 1$, which is of course a polynomial in one variable $\varsigma(x) = \sum^{n}_{j=0} a_j x^j$. Since this is a differential constant, this implies that the derivative of $\varsigma(x)$ is zero, so $\varsigma^\prime(x) = 0$ which in particular implies that $a_1=0$. On the other hand, by Lemma \ref{lem:sigma1}.(\ref{lem:sigma1.iii}), $\varsigma(x)$ is $k$-linear, so in particular $\varsigma(0) = 0$, and so $a_0=0$. So our polynomial would have to be of the form $\varsigma(x) = \sum^{n}_{j=2} a_j x^j$. Now by Lemma \ref{lem:sigma1}.(\ref{lem:sigma1.ii}), we also have that $\varsigma(\varsigma(x)) = \varsigma(x)$. However, since $\varsigma(x)$ has no degree 0 or 1 terms, $\varsigma(\varsigma(x))$ has no degree 2 or degree 3 terms. So from $\varsigma(\varsigma(x)) = \varsigma(x)$ we get that $a_2 = 0$ and $a_3 =0$. So we now have that $\varsigma(x) = \sum^{n}_{j=4} a_j x^j$. But from this, we see that $\varsigma(\varsigma(x))$ would have no degree 4 or degree 5 terms. Then again from $\varsigma(\varsigma(x)) = \varsigma(x)$, we get that $a_4 =0$ and $a_5=0$. So by repeating this argument, we will eventually get that $a_j =0$ for all $0 \leq j \leq n$, and so the differential constant unit would have to be zero, $\varsigma(x) =0$. Now observe that all $r \in k$ gives a differential constant (seen as a constant polynomial).Then by definition of a being a differential constant unit, we would have that $\varsigma(r) = r$ for all $r \in k$. But since we have shown that $\varsigma(x) =0$, this implies that $r=0$ for all $r \in k$, that is, $k$ is trivial. Therefore if $k$ is non-trivial, $k\text{-}\mathsf{POLY}$ cannot have a differential constant unit, and hence $k\text{-}\mathsf{POLY}$ is not cofree.  

\item In $\mathsf{SMOOTH}$, recall that the differential constants were precisely the constant functions. So if $\mathsf{SMOOTH}$ had a differential constant unit, then we would have a constant function $\varsigma_\mathbb{R}: \mathbb{R} \to \mathbb{R}$. By Lemma \ref{lem:sigma1}.(\ref{lem:sigma1.iii}), $\varsigma_\mathbb{R}$ would also be $\mathbb{R}$-linear. However, there is no constant function $\mathbb{R} \to \mathbb{R}$ which is also $\mathbb{R}$-linear. Thus $\mathsf{SMOOTH}$ does not have a differential constant unit, and hence it is not cofree. 

\item Recall that in any $k$-linear category $\mathbb{B}$ with finite biproducts, the only $\mathsf{D}^{\mathsf{lin}}$-constants were the zero maps. As such, trivially $0$ is a differential constant unit given by $0$, and so $\mathsf{D}\text{-}\mathsf{con}\left[ \mathbb{B} \right]$ is equivalent to the terminal category $\ast$, which has only one object and one map. However, the Faà di Bruno construction over $\ast$ is precisely $\ast$. Therefore if $\mathbb{B}$ is non-trivial, that is, has at least one non-zero object, then $\mathbb{B}$ is not equivalent to $\ast$ and therefore cannot be cofree. 
\end{enumerate}
\end{example}

\section{Characterisation via Differential Constants}\label{sec:ChDC}

By Theorem \ref{thm:diffcon1}, a Cartesian $k$-differential category is cofree if and only if it has a differential constant unit and is cofree over its category of differential constants. This characterisation still requires one to check the couniversal property, and therefore we have to work ``outside'' of the Cartesian $k$-differential category. In this section, we provide necessary and sufficient conditions on differential constants to give us the desired couniversal property. In particular, this provides a base independent and completely internal characterisation of cofree Cartesian $k$-differential categories. 

As explained in Section \ref{sec:ultrametric}, the hom-sets of a cofree Cartesian $k$-differential categories have a canonical ultrametric induced by the functor down to the base category. To provide a base-independent description of the ultrametric, we may re-express it using the differential constant unit. 

\begin{definition}\label{def:dconcom} A Cartesian $k$-differential category $(\mathbb{X}, \mathsf{D})$ with a $ \mathsf{D}$-constant unit $\varsigma$ is said to be \textbf{differential constant complete}, or $\mathsf{D}$-constant complete for short, if for each hom-set the function defined as follows: 
\begin{gather*}
\mathsf{d}_\varsigma: \mathbb{X}(X,Y) \times \mathbb{X}(X,Y) \to \mathbb{R}_{\geq 0} \\ 
\mathsf{d}_\varsigma(f,g) = \begin{cases} 2^{-n} & \text{ where $n \in \mathbb{N}$ is the smallest natural number} \\
&~ \text{such that $\varsigma_Y \circ \partial^n[f] \neq \varsigma_Y \circ \partial^n[g]$}\\
& \\ 
0 & \text{ if all $n \in \mathbb{N}$, $\varsigma_Y \circ \partial^n[f] = \varsigma_Y \circ \partial{D}^n[g]$} 
\end{cases}
\end{gather*} 
is an ultrametric on $\mathbb{X}(X,Y)$ such that $(\mathbb{X}(X,Y), \mathsf{d}_\varsigma)$ is a complete ultrametric space. 
\end{definition}

We first show that cofree Cartesian $k$-differential categories are differential constant complete and that the ultrametric given by the differential constant unit from Definition \ref{def:dconcom} is the same as the ultrametric given by the functor down to its base from Proposition \ref{prop:um}. 

\begin{lemma}\label{lem:metric:diffcon} Let $\left((\mathbb{X}, \mathsf{D}), \mathcal{E}\right)$ be a cofree Cartesian $k$-differential category over a Cartesian left $k$-linear category $\mathbb{A}$, and let $\varsigma$ be the induced $\mathsf{D}$-constant unit. Then $(\mathbb{X}, \mathsf{D})$ is $\mathsf{D}$-constant complete and furthermore, $\mathsf{d}_\varsigma(-,-) = \mathsf{d}_\mathcal{E}(-,-)$. 
\end{lemma}
\begin{proof} It suffices to prove that $\mathsf{d}_\varsigma(-,-) = \mathsf{d}_\mathcal{E}(-,-)$. By Proposition \ref{prop:cofree-diffcon},  $\overline{\mathcal{E}}_\varsigma$ is an isomorphism, so in particular is injective on maps, and also that $\overline{\mathcal{E}}_\varsigma \circ \mathcal{E}_\varsigma = \mathcal{E}$. Therefore, we have that: 
\begin{align*}
\mathsf{d}_\mathcal{E}\left( f, g \right) &= \begin{cases} 2^{-n} & \text{ where $n \in \mathbb{N}$ is the smallest natural number} \\
&~\text{such that $\mathcal{E}\left( \partial^{n}[f] \right) \neq \mathcal{E}\left( \partial^{n}[g] \right)$}\\ 
& \\ 
0 & \text{ if all $n \in \mathbb{N}$, $\mathcal{E}\left( \partial^n[f] \right) = \mathcal{E}\left( \partial^n[g] \right)$} 
\end{cases} \\
&=\begin{cases} 2^{-n} & \text{ where $n \in \mathbb{N}$ is the smallest natural number} \\
&~ \text{such that $\overline{\mathcal{E}}_\varsigma\left( \mathcal{E}_\varsigma\left( \partial^{n}[f] \right)\right) \neq \overline{\mathcal{E}}_\varsigma\left( \mathcal{E}_\varsigma\left( \partial^{n}[g] \right)\right) $}\\
& \\ 
0 & \text{ if all $n \in \mathbb{N}$, $\overline{\mathcal{E}}_\varsigma\left( \mathcal{E}_\varsigma\left( \partial^{n}[f] \right)\right) = \overline{\mathcal{E}}_\varsigma\left( \mathcal{E}_\varsigma\left( \partial^{n}[g] \right)\right) $} 
\end{cases} \\
&= \begin{cases} 2^{-n} & \text{ where $n \in \mathbb{N}$ is the smallest natural number} \\
&~ \text{such that $\overline{\mathcal{E}}_\varsigma\left( \varsigma \circ \partial^{n}[f] \right) \neq \overline{\mathcal{E}}_\varsigma\left( \varsigma \circ \partial^{n}[g] \right) $}\\ 
& \\ 
0 & \text{ if all $n \in \mathbb{N}$, $\overline{\mathcal{E}}_\varsigma\left( \varsigma \circ \partial^{n}[f] \right) = \overline{\mathcal{E}}_\varsigma\left( \varsigma \circ \partial^{n}[g] \right)  $} 
\end{cases} \\
&= \begin{cases} 2^{-n} & \text{ where $n \in \mathbb{N}$ is the smallest natural number} \\
&~ \text{such that $\varsigma_Y \circ \partial^n[f] \neq \varsigma_Y \circ \partial^n[g]$}\\
& \\ 
0 & \text{ if all $n \in \mathbb{N}$, $\varsigma_Y \circ \partial^n[f] = \varsigma_Y \circ \partial{D}^n[g]$} 
\end{cases} \\ 
&= \mathsf{d}_\varsigma(f,g) 
\end{align*}
As such, we conclude that $(\mathbb{X}, \mathsf{D})$ is $\mathsf{D}$-constant complete. 
\end{proof}

We also require the ability to express maps in terms of converging infinite sums based on their higher-order derivatives, which, in a way, means that maps are formal power series or Hurwitz series \cite{keigher2000hurwitz}. 

\begin{definition}\label{def:dconconv} In a Cartesian $k$-differential category $(\mathbb{X}, \mathsf{D})$ with a differential constant unit $\varsigma$, for every $n \in \mathbb{N}$, a \textbf{differential constant homogeneous map of degree $n$}, or $\mathsf{D}$-constant homogenous map for short, is a map ${f: X \to Y}$ such that $\varsigma_Y \circ \partial^m[f] = 0$ for all $m \neq n$. A Cartesian $k$-differential category $(\mathbb{X}, \mathsf{D})$ with a differential constant unit $\varsigma$ is said to have \textbf{convenient differential constants}, or convenient $\mathsf{D}$-constants for short, if for all $n \in \mathbb{N}$ and every differential constant $f: X \times X^n \to Y$ which is $k$-multilinear and symmetric in its last $n$-arguments, there exists a unique  $\mathsf{D}$-constant homogeneous map of degree $n$, $f^\natural: X \to Y$, such that $\varsigma_Y \circ \partial^n[f^\natural] = f$. 
\end{definition}

\begin{lemma}\label{lem:diffconvenient} Let $(\mathbb{X}, \mathsf{D})$ be a Cartesian $k$-differential category with a $\mathsf{D}$-constant unit $\varsigma$ which is $\mathsf{D}$-constant complete and also has convenient $\mathsf{D}$-constants. 
\begin{enumerate}[{\em (i)}]
\item \label{lem:diffconvenient.1} Suppose that $f_{\bullet}: X \to Y$ is a $\mathsf{D}\text{-}\mathsf{con}\left[ \mathbb{X} \right]$-Faà di Bruno sequence, that is, for each $n \in \mathbb{N}$, $f_n: X \times X^n \to Y$ is a $\mathsf{D}$-constant which $k$-multilinear and symmetric in its last $n$-arguments. For each $n \in \mathbb{N}$, consider the unique $\mathsf{D}$-constant homogenous map ${f_n}^\natural: X \to Y$ such that $\varsigma_Y \circ \partial^n[{f_n}^\natural] = f_n$. Then the series $\sum \limits^{\infty}_{n=0} {f_n}^\natural $ converges in $(\mathbb{X}(X,Y), \mathsf{d}_\varsigma)$, and furthermore, for all $m \in \mathbb{N}$, $\varsigma_Y \circ \partial^m\left[   \sum \limits^{\infty}_{n=0} {f_n}^\natural \right] = f_m$. 
\item \label{lem:diffconvenient.2} For any map $f: X \to Y$, define $f_{(n)}: X \to Y$ to be $f_{(n)} = (\varsigma_Y \circ \partial^n[f])^\natural$, that is, $f_{(n)}$ is the unique map such that $\varsigma_Y \circ \partial^n[f_{(n)}] = \varsigma_Y \circ \partial^n[f]$ and $\varsigma_Y \circ \partial^m[f_{(n)}] = 0$ for all $m \neq n$. Then $f_{(0)} = \varsigma_Y \circ f$ and the series $\sum \limits^{\infty}_{n=0}   f_{(n)}$ converges in $(\mathbb{X}(X,Y), \mathsf{d}_\varsigma)$ to $f$, that is, $\sum \limits^{\infty}_{n=0}   f_{(n)} = f$. 
\end{enumerate}
\end{lemma}
\begin{proof} For (\ref{lem:diffconvenient.1}), since $(\mathbb{X}(X,Y), \mathsf{d}_\varsigma)$ is a complete metric space, to show that the series converges it suffices to show that the sequence $\left( \sum \limits^{m}_{n=0} {f_n}^\natural  \right)_{m \in \mathbb{N}}$ is Cauchy. Note that in this case, a sequence $(f_n)_{n\in \mathbb{N}}$ in $\left( \mathbb{X}(X,Y), \mathsf{d}_\varsigma \right)$ is Cauchy if and only if for all $m \in \mathbb{N}$, there exists a $N_m \in \mathbb{N}$ such that for all $n \geq N_m$, $\mathsf{d}_\varsigma(f_n,f_{n+1}) < 2^{-m}$, or equivalently, $\varsigma_Y \circ \partial^j[f_n] = \varsigma_Y \circ \partial^j[f_{n+1}]$ for all $0 \leq j \leq m$. Now for any natural number $p \in \mathbb{N}$, for all $m \geq p$, by homogeneity and \textbf{[HD.1]}, we have that for any $0 \leq j \leq p$ that $\varsigma_Y \circ \partial^j\left[  \sum \limits^{m}_{n=0} {f_n}^\natural \right] = f^j$. Therefore, for all $m \geq p$, $\mathsf{d}_\varsigma\left( \sum \limits^{m}_{n=0} {f_n}^\natural, \sum \limits^{m+1}_{n=0} f_n^\natural \right) < 2^{-p}$. Thus it follows that $\left( \sum \limits^{m}_{n=0} {f_n}^\natural  \right)_{m \in \mathbb{N}}$ is Cauchy and therefore the series $\sum \limits^{\infty}_{n=0} {f_n}^\natural $ converges in $(\mathbb{X}(X,Y), \mathsf{d}_\varsigma)$. 

To show the other desired identity, we first need to explain why the differential combinator preserves infinite sums. By similar arguments that were done in the proof of Proposition \ref{prop:um}, we note that if $\mathsf{d}_\varsigma(f,g) \leq 2^{-(n+1)}$ then $\mathsf{d}_\varsigma\left( \mathsf{D}[f], \mathsf{D}[g] \right) \leq 2^{-n}$. Now suppose that $\sum \limits^{\infty}_{n=0} f_n$ is a converging series. We must show that $\mathsf{D}\left[ \sum \limits^{\infty}_{n=0} f_n\right]$ is the limit of the sequence $\left(\sum \limits^{m}_{n=0} \mathsf{D}[f_n] \right)_{m \in \mathbb{N}}$. So let $\varepsilon >0$ be an arbitrary non-zero positive real number and let $p$ be the smallest natural number such that $2^{-p} \leq \varepsilon$. Since $\sum \limits^{\infty}_{n=0} f_n$ converges, there exists a natural number $N$ such that for all $m \geq N$, $\mathsf{d}_\varsigma\left( \sum \limits^{m}_{n=0} f_n, \sum \limits^{\infty}_{n=0} f_n \right) < 2^{-(p+2)}$. So we have that $\mathsf{d}_\varsigma\left( \mathsf{D}\left[\sum \limits^{m}_{n=0} f_n \right], \mathsf{D}\left[\sum \limits^{\infty}_{n=0} f_n \right] \right) \leq 2^{-(p+1)}$. By \textbf{[CD.2]}, $\mathsf{D}\left[\sum \limits^{m}_{n=0} f_n \right] = \sum \limits^{m}_{n=0}  \mathsf{D}\left[f_n \right]$. Therefore, we have that $\mathsf{d}_\varsigma\left(\sum \limits^{m}_{n=0}  \mathsf{D}\left[f_n \right], \mathsf{D}\left[\sum \limits^{\infty}_{n=0} f_n \right] \right) \leq 2^{-(p+1)} < 2^{-p} \leq \varepsilon$. So $\sum \limits^{\infty}_{n=0}  \mathsf{D}\left[f_n \right] = \mathsf{D}\left[\sum \limits^{\infty}_{n=0} f_n \right]$. Similarly, we can show that $\partial^m$ also preserves infinite sums. 

As such, we finally have that for all $m \in \mathbb{N}$ 
\[ \varsigma_Y \circ \partial^m\left[   \sum \limits^{\infty}_{n=0} {f_n}^\natural \right] = \varsigma_Y \circ   \sum \limits^{\infty}_{n=0} \partial^m\left[ {f_n}^\natural \right] = \varsigma_Y \circ (0 + 0 + \hdots + 0 + f_m + 0 + 0 + \hdots) = \varsigma_Y \circ f_m = f_m  \]
So $\varsigma_Y \circ \partial^m\left[   \sum \limits^{\infty}_{n=0} {f_n}^\natural \right] = f_m$ as desired. 

For (\ref{lem:diffconvenient.2}), note that $\left(\varsigma_Y \circ \partial^n[f] \right)_{n \in \mathbb{N}}$ is a $\mathsf{D}\text{-}\mathsf{con}\left[ \mathbb{X} \right]$-Faà di Bruno sequence. So by (\ref{lem:diffconvenient.1}) the series $\sum \limits^{\infty}_{n=0}   f_{(n)}$ converges. Next note that for all $n \in \mathbb{N}$, by (\ref{lem:diffconvenient.1}) we also have that, $\varsigma_Y \circ \partial^n\left[   \sum \limits^{\infty}_{n=0} {f_{(n)}}^\natural \right] = \varsigma_Y \circ \partial^n[f]$. But this implies that $\mathsf{d}_\varsigma\left(\sum\limits^{\infty}_{n=0} {f_{(n)}}^\natural,f \right)=0$. Since $\mathsf{d}_\varsigma$ is an ultrametric, we conclude that $\sum \limits^{\infty}_{n=0}   f_{(n)} = f$. 
\end{proof}

We may now state the main result of this section: 

\begin{theorem}\label{thm:2} A Cartesian $k$-differential category is cofree if and only if it has a differential constant unit, is differential constant complete, and has convenient differential constants. 
\end{theorem}
\begin{proof} For the $\Rightarrow$ direction, suppose that $\left( \mathbb{X}, \mathsf{D} \right)$ is a Cartesian $k$-differential category which is cofree. By Theorem \ref{thm:diffcon1}, this implies that $\left( \mathbb{X}, \mathsf{D} \right)$ has a $\mathsf{D}$-constant unit $\varsigma$ and the unique induced Cartesian $k$-differential functor ${\mathcal{E}_\varsigma^\flat: \left( \mathbb{X}, \mathsf{D} \right) \to \left(\mathsf{F}\left[\mathsf{D}\text{-}\mathsf{con}\left[ \mathbb{X} \right]  \right],\mathsf{D} \right)}$ is an isomorphism. By Lemma \ref{lem:metric:diffcon}, we already have that $\left( \mathbb{X}, \mathsf{D} \right)$ is $\mathsf{D}$-constant complete. So it suffices to prove that $\left( \mathbb{X}, \mathsf{D} \right)$ also has convenient $\mathsf{D}$-constants. First observe that by the couniversal property of $\left(\left( \mathbb{X}, \mathsf{D} \right), \mathcal{E}_{\varsigma} \right)$, that we have a unique Cartesian $k$-differential functor ${\mathcal{E}_\varsigma^\flat}^{-1}: \left(\mathsf{F}\left[\mathsf{D}\text{-}\mathsf{con}\left[ \mathbb{X} \right]  \right],\mathsf{D} \right) \to \left( \mathbb{X}, \mathsf{D} \right)$ which makes the following diagram commute: 
\begin{align*} \xymatrixcolsep{5pc}\xymatrix{\mathsf{F}\left[\mathsf{D}\text{-}\mathsf{con}\left[ \mathbb{X} \right] \right] \ar[dr]_-{\mathcal{E}_{\mathsf{D}\text{-}\mathsf{con}\left[ \mathbb{X} \right]}}  \ar@{-->}[r]^-{\exists! ~ {\mathcal{E}_\varsigma^\flat}^{-1}}  & \mathbb{X} \ar[d]^-{\mathcal{E}_\varsigma} \\
  & \mathbb{A} }
\end{align*}
Then it follows that on objects ${\mathcal{E}_\varsigma^\flat}^{-1}(X) =X$ and that for a $\mathsf{D}\text{-}\mathsf{con}\left[ \mathbb{X} \right]$-Faà di Bruno sequence $(f_n)_{n \in \mathbb{N}}: X \to Y$, the following equality holds for all $n \in \mathbb{N}$, $\varsigma_Y \circ {\mathcal{E}_\varsigma^\flat}^{-1}\left( (f_n)_{n\in \mathbb{N}} \right) = f_0$. Now let $f: X \times X^n \to Y$ be a differential constant map which is also $k$-multilinear and symmetric in its last $n$-arguments. Then the sequence $(f^\sharp_m)_{m\in \mathbb{N}}$ defined as $f^\sharp_n = f_n$ and $f^\sharp_m = 0$ if $m \neq n$, is a $\mathsf{D}\text{-}\mathsf{con}\left[ \mathbb{X} \right]$-Faà di Bruno sequence. So $(f^\sharp_m)_{m\in \mathbb{N}}: X \to Y$ is a map in $\mathsf{F}\left[\mathsf{D}\text{-}\mathsf{con}\left[ \mathbb{X} \right] \right]$ and therefore ${\mathcal{E}_\varsigma^\flat}^{-1}\left( (f^\sharp_m)_{m\in \mathbb{N}} \right): X \to Y$ is a map in $\mathbb{X}$. So define $f^\natural := {\mathcal{E}_\varsigma^\flat}^{-1}\left( (f^\sharp_m)_{m\in \mathbb{N}} \right)$.  Then we compute: 
\begin{gather*}
    \varsigma_Y \circ \partial^n[f^\natural] = \varsigma_Y \circ \partial^n\left[{\mathcal{E}_\varsigma^\flat}^{-1}\left( (f^\sharp_m)_{m\in \mathbb{N}} \right) \right] = \varsigma_Y \circ {\mathcal{E}_\varsigma^\flat}^{-1}\left( \partial^n\left[(f^\sharp_m)_{m\in \mathbb{N}}  \right]   \right) = \partial^n\left[(f^\sharp_m)_{m\in \mathbb{N}}  \right]_0 = \begin{cases} f_n & \text{ if } m=n \\
     0 & \text{ if } m \neq n
 \end{cases}
\end{gather*}
So $\varsigma_Y \circ \partial^n[f^\natural] = f$ and $\varsigma_Y \circ \partial^m[f^\natural] = 0$ for all $m \neq n$ as desired. For uniqueness, suppose that there is another map $h: X \to Y$ such that $\varsigma_Y \circ \partial^n[h] = f$ and $\varsigma_Y \circ \partial^m[h] = 0$ for all $m \neq n$. By definition of $\mathcal{E}_\varsigma^\flat$, as given in Theorem \ref{thm:diffcon1}, it follows that $\mathcal{E}_\varsigma^\flat(h) = (f^\sharp_m)_{m\in \mathbb{N}}$. Thus, $h = {\mathcal{E}_\varsigma^\flat}^{-1}\left( (f^\sharp_m)_{m\in \mathbb{N}} \right)$, and so $f^\natural$ is unique. Therefore we conclude that $\left( \mathbb{X}, \mathsf{D} \right)$ has convenient $\mathsf{D}$-constants. 

For the $\Leftarrow$ direction, suppose that $\left( \mathbb{X}, \mathsf{D} \right)$ is a Cartesian $k$-differential category which has a $\mathsf{D}$-constant unit $\varsigma$, and which is $\mathsf{D}$-constant complete and has convenient $\mathsf{D}$-constants. Then by Theorem \ref{thm:diffcon1}, to prove that $\left( \mathbb{X}, \mathsf{D} \right)$ is cofree, it suffices to prove that the unique induced Cartesian $k$-differential functor ${\mathcal{E}_\varsigma^\flat: \left( \mathbb{X}, \mathsf{D} \right) \to \left(\mathsf{F}\left[\mathsf{D}\text{-}\mathsf{con}\left[ \mathbb{X} \right]  \right],\mathsf{D} \right)}$ is an isomorphism. To do so, it suffices to prove that $\mathcal{E}_\varsigma^\flat$ is bijective on objects and maps, allowing us to avoid working with the composition in $\mathsf{F}\left[\mathsf{D}\text{-}\mathsf{con}\left[ \mathbb{X} \right] \right]$. By construction, $\mathcal{E}_\varsigma^\flat$ is bijective on objects via Corollary \ref{cor:Eobj}. Then for maps, starting with injectivity, suppose that $\mathcal{E}_\varsigma^\flat(f) = \mathcal{E}_\varsigma^\flat(g)$. This implies that for all $n \in \mathbb{N}$, $\varsigma \circ \partial^n[f] = \varsigma \circ \partial^n[g]$. By Lemma \ref{lem:diffconvenient}.(\ref{lem:diffconvenient.2}), we have that: 
\[ f = \sum \limits^{\infty}_{n=0}   f_{(n)} = \sum \limits^{\infty}_{n=0} (\varsigma \circ \partial^n[f])^\natural = \sum \limits^{\infty}_{n=0} (\varsigma \circ \partial^n[g])^\natural = \sum \limits^{\infty}_{n=0}   g_{(n)} = g \]
So $f=g$, and thus $\mathcal{E}_\varsigma^\flat$ is injective on maps. For surjectivity, let $(f_n)_{n \in \mathbb{N}}$ be a $\mathsf{D}\text{-}\mathsf{con}\left[ \mathbb{X} \right]$-Faà di Bruno sequence. By Lemma \ref{lem:diffconvenient}.(\ref{lem:diffconvenient.1}), we have that the series $\sum \limits^{\infty}_{n=0} f^\natural_n$ converges (and is a map in $\mathbb{X}$). So we compute that: 
\begin{align*}
    \mathcal{E}_\varsigma^\flat\left(\sum \limits^{\infty}_{n=0} f^\natural_n \right) = \left(\varsigma \circ \partial^n\left[   \sum \limits^{\infty}_{n=0} f^\natural_n \right] \right)_{n \in \mathbb{X}} =  (f_n)_{n \in \mathbb{N}}
\end{align*}
So $\mathcal{E}_\varsigma^\flat$ is surjective. Thus, $\mathcal{E}_\varsigma^\flat$ is bijective on maps as well, and we conclude that $\mathcal{E}_\varsigma^\flat$ is an isomorphism. Explicitly, the inverse ${\mathcal{E}_\varsigma^\flat}^{-1}: \mathsf{F}\left[\mathsf{D}\text{-}\mathsf{con}\left[ \mathbb{X} \right] \right] \to \mathbb{X}$ is defined on objects as ${\mathcal{E}_\varsigma^\flat}^{-1}(X) = X$ and on maps as ${\mathcal{E}_\varsigma^\flat}^{-1}\left( (f_n)_{n \in \mathbb{N}} \right) = \sum \limits^{\infty}_{n=0} f^\natural$. Therefore, we conclude that $\mathbb{X}$ is cofree.
\end{proof}

We now explicitly express the couniversal property of a Cartesian $k$-differential category which has a differential constant unit, is differential constant complete, and has convenient differential constants.

\begin{corollary} Let $(\mathbb{X}, \!\mathsf{D})$ be a Cartesian $k$-differential category which has a $\mathsf{D}$-constant unit $\varsigma$, and which is $\mathsf{D}$-constant complete and has convenient $\mathsf{D}$-constants. Then the triple $\left((\mathbb{X}, \mathsf{D}), \mathcal{E}_{\varsigma} \right)$ is a cofree Cartesian $k$-differential category over $\mathsf{D}\text{-}\mathsf{con}\left[ \mathbb{X} \right]$, where in particular for any Cartesian $k$-differential category $(\mathbb{Y},\mathsf{D})$ and Cartesian $k$-linear functor ${\mathcal{F}: \mathbb{Y} \to \mathbb{A}}$, the unique Cartesian $k$-differential functor $\mathcal{F}^\flat: (\mathbb{Y}, \mathsf{D}) \to (\mathbb{X}, \mathsf{D})$ such that the following diagram commutes: 
\begin{align*} \xymatrixcolsep{5pc}\xymatrix{\mathbb{Y} \ar[dr]_-{\mathcal{F}}  \ar@{-->}[r]^-{\exists! ~ \mathcal{F}^\flat}  & \mathbb{X} \ar[d]^-{\mathcal{E}_\varsigma} \\
  & \mathsf{D}\text{-}\mathsf{con}\left[ \mathbb{X} \right] }
\end{align*}
is defined on objects as $\mathcal{F}^\flat(Y) = \mathcal{F}(Y)$ and on maps as $\mathcal{F}^\flat(f) = \sum \limits^{\infty}_{n=0} \mathcal{F}\left(\partial^n[f] \right)^\natural_n$. 
\end{corollary}

\begin{example}\label{ex:cofreeChCD} \normalfont Here are the differential constant homogenous maps and infinite sum decomposition in our examples of cofree Cartesian $k$-differential categories. 
\begin{enumerate}[{\em (i)}]
\item\label{ex:faaChCD} For a Cartesian left $k$-linear category $\mathbb{A}$, in its Faà di Bruno construction, a Faà di Bruno sequence $f_{\bullet}$ is differential constant homogenous of degree $n$ if $f_m =0$ for all $m \neq n$, so $f_{\bullet}=(0,0,\hdots, 0, f_n, 0, 0, \hdots)$ for some map $f_n$ in $\mathbb{A}$ which is $k$-multilinear and symmetric in its last $n$-arguments. For a Faà di Bruno sequence $f_{\bullet}$, ${f_{\bullet}}_{(n)} = (0,0,\hdots, 0, f_n, 0, 0, \hdots)$, and therefore $f_{\bullet} = (f_0, 0, \hdots) + (0, f_1, 0 \hdots) + (0, 0, f_2, 0, \hdots) + \hdots$ as was desired. 
\item For a $k$-linear category $\mathbb{B}$ with finite biproducts, in $\mathbb{B}^\Delta$ there are no $\mathsf{D}^\Delta$-constant homogenous map of degree $n \geq 2$ other than $(0,0)$. Those of degree $0$ are of the form $(f,0)$ and those of degree $1$ are of the form $(0,g)$, for any maps $f$ and $g$ in $\mathbb{B}$. Then for an arbitrary map $(f,g)$ in $\mathbb{B}^\Delta$, $(f,g)_{(0)} = (f,0)$ and $(f,g)_{(1)} = (0,g)$, and clearly $(f,g) = (f,0) + (0,g)$. 
\item In $k\text{-}\mathsf{MOD}^\mathcal{Q}$, a $k$-linear morphism $f: \mathcal{Q}(M) \to M$ is differential constant homogenous of degree $0$ if for all $m \in \mathbb{N}$, $f\left(\vert x_1, \hdots, x_{m+1} \rangle_{x_0} \right) = 0$, and is of degree $n+1$ if $f \left( \vert 1 \rangle_{(x_0)} \right) = 0$ and $f\left(\vert x_1, \hdots, x_m \rangle_{x_0} \right) = 0$ for all $m \neq n+1$. Then for a $k$-linear morphism $f: \mathcal{Q}(M) \to M$, $f_{(0)}$ is defined on pure symmetrized tensors as $f_{(0)} \left( \vert 1 \rangle_{(x_0)} \right) = 0$ and $f_{(0)}\left(\vert x_1, \hdots, x_m \rangle_{x_0} \right) = 0$, while $f_{(n+1)}$ is defined on pure symmetrized tensors as $f_{(n+1)} \left( \vert 1 \rangle_{(x_0)} \right) = 0$ and $f_{(n+1)}\left(\vert x_1, \hdots, x_m \rangle_{x_0} \right) = 0$ if $m \neq n$, and $f_{(n+1)}\left(\vert x_1, \hdots, x_n \rangle_{x_0} \right) = f\left(\vert x_1, \hdots, x_n \rangle_{x_0} \right)$. 
\end{enumerate}
\end{example}

\section{Characterisation as Monad Algebras}\label{sec:cofreealg}

In this section, we characterise cofree Cartesian $k$-differential categories this time as algebras of a monad on the category of Cartesian $k$-differential categories. This monad arises from the Faà di Bruno adjunction \cite[Corollary 3.13]{garner2020cartesian}, which is the adjunction between the category of Cartesian $k$-differential categories and the category of Cartesian left $k$-linear categories where the left adjoint is the forgetful functor and the right adjoint is given by the Faà di Bruno construction. The resulting comonad was first studied by Cockett and Seely in \cite[Theorem 2.2.2]{cockett2011faa} where they also showed that the coalgebras of this comonad were precisely Cartesian $k$-differential categories \cite[Theorem 3.2.4]{cockett2011faa}. In other words, the Faà di Bruno adjunction is comonadic. In this section, we will also explain why the Faà di Bruno adjunction is also monadic, that is, why the algebras of the induced monad correspond precisely to Cartesian $k$-left linear categories. As such, the Faà di Bruno comonad is of effective descent type \cite[Section 2]{mesablishvili2006monads}.  

\begin{remark}\normalfont We note that this entire section can be reformulated in terms of the $\mathcal{D}$-sequence construction \cite{lemay2018tangent}, since it also provides a right adjoint to the forgetful functor. Therefore, cofree Cartesian $k$-differential categories can also be characterised as algebras of the monad induced by the $\mathcal{D}$-sequence construction. 
\end{remark}

Let $\mathbb{CLLC}_k$ be the category whose objects are Cartesian left $k$-linear categories and whose maps are Cartesian $k$-linear functors between them. Similarly, let $\mathbb{CDC}_k$ be the category whose objects are Cartesian $k$-differential categories and whose maps are Cartesian $k$-differential functors between them. Let $\mathfrak{U}: \mathbb{CDC}_k \to \mathbb{CLLC}_k$ be the forgetful functor which sends a Cartesian $k$-differential category to its underlying Cartesian $k$-linear category, $\mathfrak{U} (\mathbb{X}, \mathsf{D}) = \mathbb{X}$, and sends a Cartesian $k$-differential functor to itself viewed as a Cartesian $k$-linear functor, $\mathfrak{U}(\mathcal{F}) = \mathcal{F}$. Define the Faà di Bruno functor $\mathfrak{F}: \mathbb{CLLC}_k \to \mathbb{CDC}_k$ which sends a Cartesian left $k$-linear category to its Faà di Bruno construction, $\mathfrak{F}(A) = \left( \mathsf{F}[\mathbb{A}], \mathsf{D} \right)$, and sends a Cartesian $k$-linear functor $\mathcal{A}: \mathbb{A} \to \mathbb{A}^\prime$ to the induced unique Cartesian $k$-differential functor ${\mathsf{F}[\mathcal{A}]: \left( \mathsf{F}[\mathbb{A}], \mathsf{D} \right) \to \left( \mathsf{F}[\mathbb{A}^\prime], \mathsf{D} \right)}$ which makes the following diagram commute: 
\begin{align*} \xymatrixcolsep{5pc}\xymatrix{\mathsf{F}[\mathbb{A}] \ar[d]_-{\mathcal{E}_\mathbb{A}}  \ar@{-->}[r]^-{\exists! ~ \mathsf{F}[\mathcal{A}]}  & \mathsf{F}[\mathbb{A}^\prime] \ar[d]^-{\mathcal{E}_{\mathbb{A}^\prime}} \\
 \mathbb{A} \ar[r]_-{\mathcal{A}} & \mathbb{A}^\prime }
\end{align*}
so $\mathfrak{F}(\mathcal{A}) = \mathsf{F}[\mathcal{A}]$. Explicitly, on objects $\mathsf{F}[\mathcal{A}](A) = A$ and on maps $\mathsf{F}[\mathcal{A}](f_{\bullet}) = \mathcal{A}(f_{\bullet})$. 

The Faà di Bruno functor $\mathfrak{F}$ is a right adjoint of the forgetful functor $\mathfrak{U}$ \cite[Corollary 3.13]{garner2020cartesian}, so $\mathfrak{U} \dashv \mathfrak{F}$, which we call the \textbf{Faà di Bruno adjunction}. For a Cartesian left $k$-linear category $\mathbb{A}$, the counit of the Faà di Bruno adjunction is precisely ${\mathcal{E}_\mathbb{A}: \mathsf{F}[\mathbb{A}] \to \mathbb{A}}$. On the other hand, for a Cartesian $k$-differential category $(\mathbb{X}, \mathsf{D})$, the unit of the Faà di Bruno adjunction is the induced unique Cartesian $k$-differential functor $\mathcal{N}_{(\mathbb{X},\mathsf{D})}: (\mathbb{X},\mathsf{D}) \to \left( \mathsf{F}[\mathbb{X}], \mathsf{D} \right)$ that makes the following diagram commute: 
 \begin{align*} \xymatrixcolsep{5pc}\xymatrix{ \mathbb{X} \ar@{=}[dr]^-{} \ar@{-->}[r]^-{\exists! ~\mathcal{N}_{(\mathbb{X},\mathsf{D})}} & \mathsf{F}[\mathbb{X}] \ar[d]^-{\mathcal{E}_\mathbb{X}}     \\
 & \mathbb{X} 
 }
\end{align*}
Explicitly, on objects $\mathcal{N}_{(\mathbb{X},\mathsf{D})}(X) =X$, while on maps $\mathcal{N}_{(\mathbb{X},\mathsf{D})}(f)^\bullet = \partial^\bullet[f]=  (\partial^n[f])_{n \in \mathbb{N}}$ (as defined in Lemma \ref{lem:partialbullet}).  

The Faà di Bruno adjunction induces a comonad $\underline{\mathfrak{F}} = \mathfrak{U} \circ \mathfrak{F}$ on $\mathbb{CLLC}_k$ \cite[Theorem 2.2.2]{cockett2011faa} whose coalgebras correspond precisely to Cartesian $k$-differential categories \cite[Theorem 3.2.4]{cockett2011faa}. In other words, the Faà di Bruno adjunction is comonadic \cite[Corollary 3.13]{garner2020cartesian}, that is, the adjunction's induced comparison functor \cite[Section 2.3]{barr2000toposes} from $\mathbb{CDC}_k$ to the coEilenberg-Moore category of $\underline{\mathfrak{F}}$ is a an equivalence of categories.

Of course on the other hand, the Faà di Bruno adjunction also induces a monad $\overline{\mathfrak{F}} = \mathfrak{F} \circ \mathfrak{U}$ on $\mathbb{CDC}_k$. Let us unpack the description of this monad. The functor $\overline{\mathfrak{F}}: \mathbb{CDC}_k \to \mathbb{CDC}_k$ is defined on objects as $\overline{\mathfrak{F}}(\mathbb{X}, \mathsf{D}) := \left( \mathsf{F}[\mathbb{X}], \mathsf{D} \right)$, completely forgetting about the differential combinator of the input, and on maps as $\overline{\mathfrak{F}}(\mathcal{F}) = \mathsf{F}[\mathcal{F}]$. For a Cartesian $k$-differential category $(\mathbb{X}, \mathsf{D})$, the monad unit is ${\mathcal{N}_{(\mathbb{X},\mathsf{D})}: (\mathbb{X},\mathsf{D}) \to \left( \mathsf{F}[\mathbb{X}], \mathsf{D} \right)}$, while the monad multiplication $\mathcal{M}_{(\mathbb{X},\mathsf{D})}: \left( \mathsf{F}\left[ \mathsf{F}[\mathbb{X}] \right], \mathsf{D} \right) \to  \left( \mathsf{F}[\mathbb{X}], \mathsf{D} \right)$ is defined as $\mathcal{M}_{(\mathbb{X},\mathsf{D})} = \mathsf{F}[\mathcal{E}_{\mathbb{X}}]$, that is, the unique Cartesian $k$-differential functor which makes the following diagram commute: 
 \begin{align*} \xymatrixcolsep{5pc}\xymatrix{  \mathsf{F}\left[\mathsf{F}[\mathbb{X}] \right] \ar@{-->}[r]^-{\exists! ~\mathcal{M}_{(\mathbb{X},\mathsf{D})}}  \ar[d]_-{\mathcal{E}_{\mathsf{F}[\mathbb{X}]}} &   \mathsf{F}[\mathbb{X}] \ar[d]^-{\mathcal{E}_{\mathbb{X}}}    \\
 \mathsf{F}[\mathbb{X}] \ar[r]_-{\mathcal{E}_{\mathbb{X}}} & \mathbb{X}
 }
\end{align*}
Explicitly, on objects $\mathcal{M}_{(\mathbb{X},\mathsf{D})}(X)= X$, while for maps, which is a Faà di Bruno sequence of Faà di Bruno sequences, it picks out the first term of each sequence, $\mathcal{M}_{(\mathbb{X},\mathsf{D})}\left( \left( f_{\bullet,n} \right)_{n \in \mathbb{N}} \right) = (f_{0,n})_{n \in \mathbb{N}}$. Now recall that an $\overline{\mathfrak{F}}$-algebra is a pair $\left((\mathbb{X}, \mathsf{D}), \mathcal{X} \right)$ consisting of a Cartesian $k$-differential category $(\mathbb{X}, \mathsf{D})$ and a Cartesian $k$-differential functor $\mathcal{X}: \left( \mathsf{F}[\mathbb{X}], \mathsf{D} \right) \to \left( \mathbb{X}, \mathsf{D} \right)$, and such that the following diagrams commute: 
 \begin{align*} \xymatrixcolsep{5pc}\xymatrix{ \mathbb{X} \ar@{=}[dr]^-{} \ar[r]^-{\mathcal{N}_{(\mathbb{X},\mathsf{D})}} & \mathsf{F}[\mathbb{X}] \ar[d]^-{\mathcal{X}} &  \mathsf{F}\left[\mathsf{F}[\mathbb{X}] \right] \ar[r]^-{\mathcal{M}_{(\mathbb{X},\mathsf{D})}}  \ar[d]_-{\mathsf{F}[\mathcal{X}]} &   \mathsf{F}[\mathbb{X}] \ar[d]^-{\mathcal{X}}    \\
 & \mathbb{X} & \mathsf{F}[\mathbb{X}] \ar[r]_-{\mathcal{X}} & \mathbb{X}
 }
\end{align*}
The left diagram implies that that on objects $\mathcal{X}(X)=X$ and for any map $f$ of $\mathbb{X}$ that $\mathcal{X}(f, \partial^1[f], \partial^2[f], \hdots) = f$. The right diagram says that for a Faà di Bruno sequence of Faà di Bruno sequences that $\mathcal{X}\left( \left( \mathcal{X}\left( f_{\bullet,n} \right) \right)_{n \in \mathbb{N}} \right) = \mathcal{X}\left( (f_{0,n})_{n \in \mathbb{N}} \right)$. Later in Corollary \ref{cor:faastruct} we will in fact show that $\overline{\mathfrak{F}}$-algebra structure is unique, which means that being an $\overline{\mathfrak{F}}$-algebra is a property of a Cartesian $k$-differential category. Of course, from well known facts about adjunctions we have that for any Cartesian left $k$-linear category $\mathbb{A}$, $\left( \left( \mathsf{F}[\mathbb{A}], \mathsf{D} \right), \mathsf{F}[\mathcal{E}_\mathbb{A}] \right)$ is an $\overline{\mathfrak{F}}$-algebra.  

Our objective is to show that every $\overline{\mathfrak{F}}$-algebra is also a cofree Cartesian $k$-differential category. To help us do so, let us first show that an $\overline{\mathfrak{F}}$-algebra has a differential constant unit. 

\begin{lemma}\label{lem:faaalg-diffunit} Let $\left((\mathbb{X}, \mathsf{D}), \mathcal{X} \right)$ be an $\overline{\mathfrak{F}}$-algebra. Then $(\mathbb{X}, \mathsf{D})$ has a $\mathsf{D}$-constant unit $\varsigma$ defined as $\varsigma_X = \mathcal{X}(1_X,0, 0, \hdots)$, and we also have that: 
\begin{enumerate}[{\em (i)}]
\item \label{lem:fdu.i} The following diagram commutes: 
 \begin{align*} \xymatrixcolsep{5pc}\xymatrix{   \mathsf{F}[\mathbb{X}] \ar[r]^-{\mathcal{X}}  \ar[d]_-{\mathcal{E}_{\mathbb{X}}} &   \mathbb{X} \ar[d]^-{\mathcal{E}_\varsigma}    \\
\mathbb{X} \ar[r]_-{\mathcal{E}_\varsigma} & \mathsf{D}\text{-}\mathsf{con}\left[ \mathbb{X} \right]
 }
\end{align*} 
so in particular, for every $\mathbb{X}$-Faà di Bruno sequence $f_{\bullet}: X \to Y$, $\varsigma_Y \circ \mathcal{X}\left(f_{\bullet} \right) = \varsigma_Y \circ f_0$. 
\item  \label{lem:fdu.ii} For every map $f: X \to Y$ in $\mathbb{X}$, $( \varsigma_Y \circ \partial^n[f])_{n \in \mathbb{N}}$ is a $\mathbb{X}$-Faà di Bruno sequence and $\mathcal{X}\left( ( \varsigma_Y \circ \partial^n[f])_{n \in \mathbb{N}} \right) = f$.  
\end{enumerate}
\end{lemma}
\begin{proof} As explained in Example \ref{ex:diffconunit}.(\ref{ex:diffcoununit.faa}), ${\varsigma_\bullet}_X := (1_X,0, 0, \hdots)$ is a $\mathsf{D}$-constant unit in $\left( \mathsf{F}[\mathbb{X}], \mathsf{D} \right)$. Since $\mathcal{X}$ is a Cartesian $k$-differential functor, by Lemma \ref{lem:dconlem1}, $\varsigma_X = \mathcal{X}({\varsigma_\bullet}_X)$ is also a $\mathsf{D}$-constant. If $f: X \to Y$ is a $\mathsf{D}$-constant in $\mathbb{X}$, then since $\mathcal{N}_{(\mathbb{X}, \mathsf{D})}$ is also a Cartesian $k$-differential functor, we have that $\partial^\bullet[f]$ is also a $\mathsf{D}$-constant in $\left( \mathsf{F}[\mathbb{X}], \mathsf{D} \right)$. So we compute that: 
\[ f \circ \varsigma_X = \mathcal{X}\left( \partial^\bullet[f] \right) \circ \mathcal{X}({\varsigma_\bullet}_X) =  \mathcal{X}\left(\partial^\bullet[f] \circ {\varsigma_\bullet}_X \right) = \mathcal{X}\left( \partial^\bullet[f] \right) = f \]
and similarly that $\varsigma_Y \circ f = f$. So we conclude that $\varsigma$ is a $\mathsf{D}$-constant unit for $(\mathbb{X}, \mathsf{D})$. For (\ref{lem:fdu.i}), observe that in $\left( \mathsf{F}[\mathbb{X}], \mathsf{D} \right)$, by Proposition \ref{prop:cofree-diffcon}, that $\mathcal{E}_{\mathbb{X}}({\varsigma_\bullet} \circ f_{\bullet}) = f_0 = \mathcal{E}_{\mathbb{X}}({\varsigma_\bullet} \circ \partial^\bullet[f_0])$. By Lemma \ref{lem:dconlem1}.(\ref{dconcomp}), ${\varsigma_\bullet} \circ f_{\bullet}$ and ${\varsigma_\bullet} \circ  \partial^\bullet[f_0]$ are both $\mathsf{D}$-constants, so by Lemma \ref{lem:Econ}, since $\mathcal{E}_{\mathbb{X}}({\varsigma_\bullet} \circ f_{\bullet}) = \mathcal{E}_{\mathbb{X}}({\varsigma_\bullet} \circ  \partial^\bullet[f_0])$, then ${\varsigma_\bullet} \circ f_{\bullet} = {\varsigma_\bullet} \circ \partial^\bullet[f_0]$. Applying $\mathcal{X}$ to this equality, we then obtain $\varsigma \circ \mathcal{X}\left(f_{\bullet} \right) = \varsigma \circ f_0$, and the desired diagram commutes by definition. For (\ref{lem:fdu.ii}), by Lemma \ref{lem:sigma1}, $\varsigma$ is $k$-linear so $( \varsigma \circ \partial^n[f])_{n \in \mathbb{N}}$ is indeed an $\mathbb{X}$-Faà di Bruno sequence. Then using (\ref{lem:fdu.i}), we compute that:
\begin{gather*}
    \mathcal{X}\left( ( \varsigma \circ \partial^n[f])_{n \in \mathbb{N}} \right) = \mathcal{X}\left( \left( \varsigma \circ \mathcal{X}\left(  \partial^\bullet\left[ \partial^n f \right] \right)\right)_{n \in \mathbb{N}} \right) = \mathcal{X}\left( \left( \mathcal{X}\left( {\varsigma_\bullet}\circ \partial^\bullet\left[ \partial^n f \right] \right)\right)_{n \in \mathbb{N}} \right) =\mathcal{X}\left( \left( \partial^n[f] \right)_{n \in \mathbb{N}} \right) = f \\
\end{gather*}
So $\mathcal{X}\left( ( \varsigma \circ \partial^n[f])_{n \in \mathbb{N}} \right) = f$ as desired. 
\end{proof}

We may now prove the main result of this section. 

\begin{theorem}\label{thm:faaalg-cofree} A Cartesian $k$-differential category is cofree if and only if it is an $\overline{\mathfrak{F}}$-algebra. 
\end{theorem}
\begin{proof} For the $\Rightarrow$ direction, let $(\mathbb{X}, \mathsf{D})$ be a cofree Cartesian $k$-differential category, which by Theorem \ref{thm:diffcon1}, implies $\mathbb{X}$ has a differential constant unit $\varsigma$ such that $\left((\mathbb{X}, \mathsf{D}), \mathcal{E}_{\varsigma} \right)$ is a cofree Cartesian $k$-differential category over $\mathsf{D}\text{-}\mathsf{con}\left[ \mathbb{X} \right]$. Then, using the couniversal property of $\left((\mathbb{X}, \mathsf{D}), \mathcal{E}_{\varsigma} \right)$, define $\mathcal{X}: (\mathsf{F}[\mathbb{X}], \mathsf{D}) \to (\mathbb{X}, \mathsf{D})$ as the unique Cartesian $k$-differential functor which makes the following diagram commute: 
 \begin{align*} \xymatrixcolsep{5pc}\xymatrix{   \mathsf{F}[\mathbb{X}] \ar@{-->}[r]^-{\exists! ~\mathcal{X}}  \ar[d]_-{\mathcal{E}_\mathbb{X}} &   \mathbb{X} \ar[d]^-{\mathcal{E}_\varsigma}    \\
\mathbb{X} \ar[r]_-{\mathcal{E}_{\varsigma}} & \mathsf{D}\text{-}\mathsf{con}\left[ \mathbb{X} \right]
 }
\end{align*}
That $\mathcal{X}$ is an $\overline{\mathfrak{F}}$-algebra structure follows from the fact the diagrams on the left commute, which by Lemma \ref{lem:cofree-lem0}.(\ref{lem:cofree-lem1}) implies that the diagrams on the right commute. 
 \begin{align*} \begin{array}[c]{c}
\xymatrixcolsep{5pc}\xymatrix{  \mathbb{X} \ar@{=}[dr]^-{} \ar[r]^-{\mathcal{N}_{(\mathbb{X},\mathsf{D})}} &  \mathsf{F}[\mathbb{X}] \ar[r]^-{\mathcal{X}}  \ar[d]_-{\mathcal{E}_\mathbb{X}} &   \mathbb{X} \ar[d]^-{\mathcal{E}_\varsigma}    \\
& \mathbb{X} \ar[r]_-{\mathcal{E}_{\varsigma}} & \mathsf{D}\text{-}\mathsf{con}\left[ \mathbb{X} \right]
 }
   \end{array} \!\! \stackrel{\text{Lem.\ref{lem:cofree-lem0}.(\ref{lem:cofree-lem1})}}{\Longrightarrow}\!\!  \begin{array}[c]{c} \xymatrixcolsep{5pc}\xymatrix{ \mathbb{X} \ar@{=}[dr]^-{} \ar[r]^-{\mathcal{N}_{(\mathbb{X},\mathsf{D})}} & \mathsf{F}[\mathbb{X}] \ar[d]^-{\mathcal{X}}   \\
 & \mathbb{X} 
 }
   \end{array}  \\
   \begin{array}[c]{c}
\xymatrixcolsep{2pc}\xymatrix{ \mathsf{F}\left[\mathsf{F}[\mathbb{X}] \right] \ar[rrr]^-{ \mathcal{M}_{\mathbb{X}} } \ar[dr]^-{ \mathcal{E}_{\mathsf{F}[\mathbb{X}]} } \ar[dd]_-{ \mathsf{F}[\mathcal{X}] } &&& \mathsf{F}[\mathbb{X}]  \ar[r]^-{\mathcal{X}} \ar[d]_-{\mathcal{E}_{\mathbb{X}}} & \mathbb{X} \ar[ddd]^-{\mathcal{E}_\varsigma}    \\ 
&  \mathsf{F}[\mathbb{X}] \ar[rr]_-{\mathcal{E}_{\mathbb{X}}} \ar[dr]_-{\mathcal{X}}  && \mathbb{X} \ar[ddr]_-{\mathcal{E}_{\varsigma}} \\ 
\mathsf{F}[\mathbb{X}] \ar[rr]^-{\mathcal{E}_\varsigma}  \ar[d]_-{\mathcal{X}}  && \mathbb{X}  \ar[drr]^-{\mathcal{E}_\varsigma} && \\
\mathbb{X}\ar[rrrr]_-{\mathcal{E}_{\varsigma}} &&  && \mathsf{D}\text{-}\mathsf{con}\left[ \mathbb{X} \right]
 } 
   \end{array} \!\! \stackrel{\text{Lem.\ref{lem:cofree-lem0}.(\ref{lem:cofree-lem1})}}{\Longrightarrow} \!\!  \begin{array}[c]{c} \xymatrixcolsep{5pc}\xymatrix{   \mathsf{F}\left[\mathsf{F}[\mathbb{X}] \right] \ar[r]^-{\mathcal{M}_{(\mathbb{X},\mathsf{D})}}   \ar[d]_-{\mathsf{F}[\mathcal{X}]} &   \mathsf{F}[\mathbb{X}] \ar[d]^-{\mathcal{X}}    \\
 \mathsf{F}[\mathbb{X}] \ar[r]_-{\mathcal{X}} & \mathbb{X}
 }
   \end{array}
\end{align*}
So we conclude that $\left((\mathbb{X}, \mathsf{D}), \mathcal{X} \right)$ is an $\overline{\mathfrak{F}}$-algebra.

For the $\Leftarrow$ direction, let $\left((\mathbb{X}, \mathsf{D}), \mathcal{X} \right)$ be an $\overline{\mathfrak{F}}$-algebra. By Lemma \ref{lem:faaalg-diffunit}, $(\mathbb{X}, \mathsf{D})$ has a differential constant unit $\varsigma$. Then, by Theorem \ref{thm:diffcon1}, to prove that $(\mathbb{X}, \mathsf{D})$ is cofree, it suffices to prove that the induced Cartesian $k$-differential functor $\mathcal{E}_\varsigma^\flat: (\mathbb{X}, \mathsf{D}) \to (\mathsf{F}\left[\mathsf{D}\text{-}\mathsf{con}\left[ \mathbb{X} \right] \right], \mathsf{D})$ is an isomorphism. As before, to prove that $\mathcal{E}_\varsigma^\flat$ is an isomorphism, it is sufficient to prove that it is bijective on objects and maps, which allows us to avoid working with composition in the Faà di Bruno construction. Recall that $\mathcal{E}_\varsigma^\flat$ was worked out explicitly in Theorem \ref{thm:diffcon1}. Clearly, $\mathcal{E}_\varsigma^\flat$ is bijective on objects. On maps, we will explain why $\mathcal{E}_\varsigma^\flat$ is injective and surjective. Starting with injectivity, suppose that $\mathcal{E}_\varsigma^\flat(f) = \mathcal{E}_\varsigma^\flat(g)$. This implies that for all $n \in \mathbb{N}$, $\varsigma \circ \partial^n[f] = \varsigma \circ \partial^n[g]$, so we have an equality of $\mathbb{X}$-Faà di Bruno sequences $\left( \varsigma \circ \partial^n[f] \right)_{n \in \mathbb{N}} = \left( \varsigma \circ \partial^n[g] \right)_{n \in \mathbb{N}}$. Then by Lemma \ref{lem:faaalg-diffunit}.(\ref{lem:fdu.ii}), we have that: 
\[ f = \mathcal{X}\left( \left( \varsigma \circ \partial^n[f] \right)_{n \in \mathbb{N}} \right) = \mathcal{X}\left( \left( \varsigma \circ \partial^n[g] \right)_{n \in \mathbb{N}} \right) = g  \]
So $f=g$, and so $\mathcal{E}_\varsigma^\flat$ is injective on maps. For surjectivity, let $f_{\bullet}$ be a $\mathsf{D}\text{-}\mathsf{con}\left[ \mathbb{X} \right]$-Faà di Bruno sequence. This means that $f_{\bullet}$ is also a $\mathbb{X}$-Faà di Bruno sequence and so $\mathcal{X}\left( f_{\bullet} \right)$ is a map in $\mathbb{X}$. Then by Lemma \ref{lem:faaalg-diffunit}.(\ref{lem:fdu.i}), that $\partial^n[f_{\bullet}]_0 = f_n$, and that $\varsigma$ is a $\mathsf{D}$-constant unit, we have that: 
\begin{gather*}
    \mathcal{E}_\varsigma^\flat\left( \mathcal{X}\left( f_{\bullet} \right) \right) =  \left( \varsigma \circ \partial^n\left[ \mathcal{X}\left( f_{\bullet} \right)\right] \right)_{n \in \mathbb{N}} = \left( \varsigma \circ \mathcal{X}\left( \partial^n\left[ f_{\bullet}\right]  \right)\right)_{n \in \mathbb{N}}  = \left( \varsigma \circ \partial^n\left[ f_{\bullet} \right]_0 \right)_{n \in \mathbb{N}} = (\varsigma \circ f_n)_{n \in \mathbb{N}} =  (f_n)_{n \in \mathbb{N}} 
\end{gather*}
So $\mathcal{E}_\varsigma^\flat\left( \mathcal{X}\left( (f_n)_{n \in \mathbb{N}} \right) \right) = (f_n)_{n \in \mathbb{N}}$, thus $\mathcal{E}_\varsigma^\flat$ is surjective on maps. Therefore, $\mathcal{E}_\varsigma^\flat$ is bijective on maps, and we conclude that $\mathcal{E}_\varsigma^\flat$ is an isomorphism. Explicitly, the inverse functor ${\mathcal{E}_\varsigma^\flat}^{-1}: \mathsf{F}\left[\mathsf{D}\text{-}\mathsf{con}\left[ \mathbb{X} \right] \right] \to \mathbb{X}$ is defined on objects as ${\mathcal{E}_\varsigma^\flat}^{-1}(X) = X$ and on maps as ${\mathcal{E}_\varsigma^\flat}^{-1}\left( (f_n)_{n \in \mathbb{N}} \right) = \mathcal{X}\left( (f_n)_{n \in \mathbb{N}} \right)$. Therefore, we conclude that $(\mathbb{X}, \mathsf{D})$ is cofree. 
\end{proof}

We may explicitly write down the couniversal property for $\overline{\mathfrak{F}}$-algebras and explain how to use the $\overline{\mathfrak{F}}$-algebra structure to construct the induced unique Cartesian $k$-differential functors.

\begin{corollary}\label{cor:faacor3} Let $\left((\mathbb{X}, \mathsf{D}), \mathcal{X} \right)$ be an $\overline{\mathfrak{F}}$-algebra. Then $(\mathbb{X}, \mathsf{D})$ has a differential constant unit $\varsigma$ and $\left((\mathbb{X}, \mathsf{D}), \mathcal{E}_{\varsigma} \right)$ is a cofree Cartesian $k$-differential category over $\mathsf{D}\text{-}\mathsf{con}\left[ \mathbb{X} \right]$, where in particular for any Cartesian $k$-differential category $(\mathbb{Y},\mathsf{D})$ and Cartesian $k$-linear functor ${\mathcal{F}: \mathbb{Y} \to \mathbb{A}}$, the unique Cartesian $k$-differential functor $\mathcal{F}^\flat: (\mathbb{Y},\mathsf{D}) \to (\mathbb{X},\mathsf{D})$ such that the following diagram commutes: 
\begin{align*} \xymatrixcolsep{5pc}\xymatrix{\mathbb{Y} \ar[dr]_-{\mathcal{F}}  \ar@{-->}[r]^-{\exists! ~ \mathcal{F}^\flat}  & \mathbb{X} \ar[d]^-{\mathcal{E}_\varsigma} \\
  & \mathsf{D}\text{-}\mathsf{con}\left[ \mathbb{X} \right] }
\end{align*}
is defined on objects as $\mathcal{F}^\flat(Y) = \mathcal{F}(Y)$ and on maps as $\mathcal{F}^\flat(f) = \mathcal{X}\left( \mathcal{F}\left(\partial^\bullet[f] \right)\right)$. 
\end{corollary}

We can now show that a Cartesian $k$-differential category has at most one possible $\overline{\mathfrak{F}}$-algebra structure. 

\begin{corollary}\label{cor:faastruct} $\overline{\mathfrak{F}}$-algebra structure on a Cartesian $k$-differential category, if it exists, is unique.
\end{corollary}
\begin{proof} Let $(\mathbb{X}, \mathsf{D})$ be a Cartesian $k$-differential category with two $\overline{\mathfrak{F}}$-algebra structures ${\mathcal{X}: \mathsf{F}[\mathbb{X}]  \to \mathbb{X}}$ and ${\mathcal{X}^\prime: \mathsf{F}[\mathbb{X}]  \to \mathbb{X}}$. By Lemma \ref{lem:faaalg-diffunit}, $(\mathbb{X}, \mathsf{D})$ has a $\mathsf{D}$-constant unit $\varsigma$. By Lemma \ref{lem:sigma1}, since $\mathsf{D}$-constant units are unique, we have that $\mathcal{X}({\varsigma_\bullet}) = \varsigma= \mathcal{X}^\prime({\varsigma_\bullet})$. By Corollary \ref{cor:faacor3}, $\left((\mathbb{X}, \mathsf{D}), \mathcal{E}_{\varsigma} \right)$ is a cofree Cartesian $k$-differential category over $\mathsf{D}\text{-}\mathsf{con}\left[ \mathbb{X} \right]$. By Lemma \ref{lem:faaalg-diffunit}.(\ref{lem:fdu.i}), we have that the following diagram commutes:
 \begin{align*} \xymatrixcolsep{5pc}\xymatrix{   \mathsf{F}[\mathbb{X}] \ar[rr]^-{\mathcal{X}} \ar[dd]_-{\mathcal{X}^\prime}  \ar[dr]_-{\mathcal{E}_{\mathbb{X}}} &&   \mathbb{X} \ar[dd]^-{\mathcal{E}_\varsigma}    \\ 
 & \mathbb{X} \ar[dr]_-{\mathcal{E}_\varsigma} \\ 
\mathbb{X} \ar[rr]_-{\mathcal{E}_\varsigma} && \mathsf{D}\text{-}\mathsf{con}\left[ \mathbb{X} \right]
 }
\end{align*}
and since $\mathcal{X}$ and $\mathcal{X}^\prime$ are Cartesian $k$-differential functors, by Lemma \ref{lem:cofree-lem0}.(\ref{lem:cofree-lem1}), it follows that $\mathcal{X} = \mathcal{X}^\prime$. 
\end{proof}

We conclude this section by explaining why the Faà di Bruno adjunction is monadic. First recall that an $\overline{\mathfrak{F}}$-algebra morphism $\mathcal{F}: \left((\mathbb{X}, \mathsf{D}), \mathcal{X} \right) \to \left((\mathbb{X}^\prime, \mathsf{D}), \mathcal{X}^\prime \right)$ is a Cartesian $k$-differential functor $\mathcal{F}: (\mathbb{X}, \mathsf{D}) \to (\mathbb{X}^\prime, \mathsf{D})$ such that the following diagram commutes: 
 \begin{align*} \xymatrixcolsep{5pc}\xymatrix{   \mathsf{F}\left[\mathbb{X} \right] \ar[r]^-{\mathsf{F}[\mathcal{F}]}  \ar[d]_-{\mathcal{X}} &   \mathsf{F}[\mathbb{X}^\prime] \ar[d]^-{\mathcal{X}^\prime}    \\
  \mathbb{X} \ar[r]_-{\mathcal{F}} & \mathbb{X}^\prime
 }
\end{align*}
Explicitly, this says that for an $\mathbb{X}$-Faà di Bruno sequence, we have that $\mathcal{F}(\mathcal{X}(f_{\bullet})) = \mathcal{X}^\prime( \mathcal{F}(f_{\bullet}) )$. Let $\mathbb{ALG}(\overline{\mathfrak{F}})$ be the category of $\overline{\mathfrak{F}}$-algebras and $\overline{\mathfrak{F}}$-algebra morphisms between them, and let $\mathfrak{C}: \mathbb{CLLC}_k \to \mathbb{ALG}(\overline{\mathfrak{F}})$ be the induced comparison functor \cite[Section 2.3]{barr2000toposes}, which is defined on objects as $\mathfrak{C}(\mathbb{A}) := \left( \left( \mathsf{F}[\mathbb{A}], \mathsf{D} \right), \mathsf{F}[\mathcal{E}_\mathbb{A}] \right)$ and maps as $\mathfrak{C}(\mathcal{A}) = \mathsf{F}[\mathcal{F}]$. Our goal is now to explain why the comparison functor is in fact an equivalence of categories. So in order to build a functor of type $\mathbb{ALG}(\overline{\mathfrak{F}}) \to \mathbb{CLLC}_k$, we will need the following lemma. 

\begin{lemma}\label{lem:faaalgmap-DCON} Let $\mathcal{F}: \left((\mathbb{X}, \mathsf{D}), \mathcal{X} \right) \to \left((\mathbb{X}^\prime, \mathsf{D}), \mathcal{X}^\prime \right)$ be an $\overline{\mathfrak{F}}$-algebra morphism. Define the functor ${\mathsf{D}\text{-}\mathsf{con}\left[ \mathcal{F} \right]: \mathsf{D}\text{-}\mathsf{con}\left[ \mathbb{X} \right] \to \mathsf{D}\text{-}\mathsf{con}\left[ \mathbb{X}^\prime \right]}$ on objects and maps as $\mathsf{D}\text{-}\mathsf{con}\left[ \mathcal{F} \right](-) = \mathcal{F}(-)$. Then $\mathsf{D}\text{-}\mathsf{con}\left[ \mathcal{F} \right]$ is a Cartesian $k$-linear functor and the following diagram commutes: 
 \begin{align*} \xymatrixcolsep{5pc}\xymatrix{  \mathbb{X} \ar[r]^-{\mathcal{F}}  \ar[d]_-{\mathcal{E}_\varsigma} &   \mathbb{X}^\prime \ar[d]^-{\mathcal{E}_\varsigma}    \\
\mathsf{D}\text{-}\mathsf{con}\left[ \mathbb{X} \right] \ar[r]_-{\mathsf{D}\text{-}\mathsf{con}\left[ \mathcal{F} \right]} & \mathsf{D}\text{-}\mathsf{con}\left[ \mathbb{X}^\prime \right]
 }
\end{align*} 
\end{lemma}
\begin{proof} Since $\mathcal{F}$ is Cartesian $k$-differential functor, by Lemma \ref{lem:dconlem1} it preserves $\mathsf{D}$-constants and so $\mathsf{D}\text{-}\mathsf{con}\left[ \mathcal{F} \right]$ is well-defined. Clearly, $\mathsf{D}\text{-}\mathsf{con}\left[ \mathcal{F} \right]$ preserves composition, and so it remains to show it also preserves identities. To do so, we must show that $\mathcal{F}$ preserves the differential constant unit. Now since $\mathcal{F}$ preserves zeros and identities, we have that $\mathcal{F}({\varsigma_\bullet}_X) = {\varsigma_\bullet}_{\mathcal{F}(X)}$. Therefore, we have that: 
\[ \mathcal{F}(\varsigma_X) = \mathcal{F}(\mathcal{X}({\varsigma_\bullet}_X)) = \mathcal{X}(\mathcal{F}({\varsigma_\bullet}_X)) = \mathcal{X}^\prime({\varsigma_\bullet}_{\mathcal{F}(X)}) = \varsigma_{\mathcal{F}(X)} \]
So $\mathsf{D}\text{-}\mathsf{con}\left[ \mathcal{F} \right]$ preserves identities and so is a functor. Since $\mathcal{F}$ is Cartesian $k$-linear then so is $\mathsf{D}\text{-}\mathsf{con}\left[ \mathcal{F} \right]$. 
\end{proof}

As such, from the above lemma, we obtain a functor $\mathfrak{D}: \mathbb{ALG}(\overline{\mathfrak{F}}) \to \mathbb{CLLC}_k$ which sends an $\overline{\mathfrak{F}}$-algebra to its category of differential constants, so $\mathfrak{D}\left((\mathbb{X}, \mathsf{D}), \mathcal{X} \right) = \mathsf{D}\text{-}\mathsf{con}\left[ \mathbb{X} \right]$, and sends an $\overline{\mathfrak{F}}$-algebra morphism to $\mathfrak{D}(\mathcal{F})=\mathsf{D}\text{-}\mathsf{con}\left[ \mathcal{F} \right]$, as defined above. Now by Lemma \ref{lem:cofreebase}, we have that $\mathfrak{D}\left( \mathfrak{C} (\mathbb{A}) \right) = \mathsf{D}\text{-}\mathsf{con}\left[ \mathsf{F}[\mathbb{A}] \right] \cong \mathbb{A}$. On the other hand, it follows from Theorem \ref{thm:faaalg-cofree}, Lemma \ref{lem:cofree-lem0}.(\ref{lem:cofree-lem2.i}), and Corollary \ref{cor:faastruct} that $\mathfrak{C}\left( \mathfrak{D}\left((\mathbb{X}, \mathsf{D}), \mathcal{X} \right) \right) =  \left( \left( \mathsf{F}\left[\mathsf{D}\text{-}\mathsf{con}\left[ \mathbb{X} \right] \right], \mathsf{D} \right), \mathsf{F}[\mathcal{E}_{\mathsf{D}\text{-}\mathsf{con}\left[ \mathbb{X} \right]}] \right) \cong \left((\mathbb{X}, \mathsf{D}), \mathcal{X} \right)$ as well. Thus we get that that $\mathbb{ALG}(\overline{\mathfrak{F}})$ and $\mathbb{CLLC}_k$ are equivalent via the Faà di Bruno adjunction's comparison functor. So, we conclude this paper with the following statement:

\begin{proposition}\label{prop:monadic} The Faà di Bruno adjunction $\mathfrak{U} \dashv \mathfrak{F}$ is monadic, so $\mathbb{ALG}(\overline{\mathfrak{F}}) \simeq \mathbb{CLLC}_k$. 
\end{proposition}

\bibliographystyle{plain}      % mathematics and physical sciences
\bibliography{cofreeCDCbib}   % name your BibTeX data base
\end{document}